\documentclass[10pt]{amsart}

\usepackage{amsmath}
\usepackage{amssymb}
\usepackage{amsfonts}%
\usepackage{graphicx}

\newtheorem{theorem}{Theorem}[section]
\newtheorem{lemma}{Lemma}[section]
\theoremstyle{definition}

\theoremstyle{remark}
\newtheorem{remark}{Remark}[section]
\numberwithin{equation}{section}
\theoremstyle{plain}

\newtheorem{proposition}{Proposition}[section]
\newtheorem{example}{Example}[section]

\begin{document}
\title[Hitting Times and Positions in Rare Events]{\textbf{Hitting Times and Positions in Rare Events}}
\author{Roland Zweim\"{u}ller}
\address{Fakult\"{a}t f\"{u}r Mathematik, Universit\"{a}t Wien, Oskar-Morgenstern-Platz
1, 1090 Vienna, Austria }
\email{roland.zweimueller@univie.ac.at}
\urladdr{http://www.mat.univie.ac.at/\symbol{126}zweimueller/ }

\begin{abstract}
We establish abstract limit theorems which provide sufficient conditions for a
sequence $(A_{l})$ of rare events in an ergodic probability preserving
dynamical system to exhibit Poisson asymptotics, and for the consecutive
positions inside the $A_{l}$ to be asymptotically iid (spatiotemporal Poisson
limits). The limit theorems only use information on what happens to $A_{l}$
before some time $\tau_{l}$ which is of order $o(1/\mu(A_{l}))$. In
particular, no assumptions on the asymptotic behavior of the system akin to
classical mixing conditions are used. We also discuss some general questions
about the asymptotic behaviour of spatial and spatiotemporal processes, and
illustrate our results in a setup of simple prototypical systems.
\newline\newline


\end{abstract}
\maketitle

\section{Introduction}

Consider an ergodic measure-preserving map $T$ on the probability space
$(X,\mathcal{A},\mu)$, and a sequence $(A_{l})_{l\geq1}$ of sets for which
$0<\mu(A_{l})\rightarrow0$. Let $\varphi_{A_{l}}$ denote the first hitting
time function of $A_{l}$. The asymptotic behavior of the distributions of the
rescaled hitting times $\mu(A_{l})\varphi_{A_{l}}$ as $l\rightarrow\infty$ is
a well-studied circle of questions. In many interesting situations, mixing
properties have been used to show that these \emph{hitting time distributions}
converge to an exponential law,
\begin{equation}
\mu(\mu(A_{l})\varphi_{A_{l}}\leq t)\longrightarrow1-e^{-t}\text{ \quad as
}l\rightarrow\infty\text{\quad for }t>0\text{, } \label{Eq_ExpHTS}%
\end{equation}
and so do the corresponding \emph{return distributions},
\begin{equation}
\mu_{A_{l}}(\mu(A_{l})\varphi_{A_{l}}\leq t)\longrightarrow1-e^{-t}\text{
\quad as }l\rightarrow\infty\text{\quad for }t>0\text{,} \label{Eq_ExpRTS}%
\end{equation}
where $\mu_{A_{l}}$ denotes the normalized restriction of $\mu$ to $A_{l}$. In
fact, (\ref{Eq_ExpHTS}) and (\ref{Eq_ExpRTS}) are equivalent: as a consequence
of Theorem 2.1 of \cite{HSV} one has

\begin{proposition}
[\textbf{Characterizing convergence to the exponential law}]Let $T$ be an
ergodic measure-preserving map on the probability space $(X,\mathcal{A},\mu)$,
and $(A_{l})_{l\geq1}$ a sequence in $\mathcal{A}$ such that $0<\mu
(A_{l})\rightarrow0$. Then each of (\ref{Eq_ExpHTS}) and (\ref{Eq_ExpRTS}) is
equivalent to
\begin{equation}
\mu(\mu(A_{l})\varphi_{A_{l}}\leq t)-\mu_{A_{l}}(\mu(A_{l})\varphi_{A_{l}}\leq
t)\longrightarrow0\text{ \quad as }l\rightarrow\infty\text{\quad for
}t>0\text{.} \label{Eq_Char_EXPHRTS}%
\end{equation}

\end{proposition}

This fact leads to one standard approach (of many) to proving (\ref{Eq_ExpHTS}%
) and (\ref{Eq_ExpRTS}): Checking condition (\ref{Eq_Char_EXPHRTS}) means to
show that, asymptotically, changing $\mu_{A_{l}}$ to $\mu$ does not affect the
resulting distribution. This is not trivial, because the $\mu_{A_{l}}$ become
increasingly singular, being concentrated on smaller and smaller sets.

For many classes of concrete systems, strong results on the decay of
correlations (or mixing properties) provide information on how the system
forgets the difference between two initial probabilities over time. A basic
form of this might state that $d_{\mathfrak{N}}(\nu\circ T^{-n},\mu)\leq
c_{n}$ for $\nu\in\mathfrak{N}$ and $n\geq0$, where $c_{n}\rightarrow0$, and
$\mathfrak{N}$ is a family of normalized measures, typically rather small, and
equipped with some metric $d_{\mathfrak{N}}$.

We can then hope to establish (\ref{Eq_Char_EXPHRTS}) once we check it is
possible to replace the measures $\mu_{A_{l}}$ there by push-forwards
$\mu_{A_{l}}\circ T^{-\tau_{l}}$ (with integers $\tau_{l}$) which are nicer in
that they belong to $\mathfrak{N}$, and hence allow comparison to $\mu$ via
the control of $d_{\mathfrak{N}}(\nu\circ T^{-n},\mu)$ on $\mathfrak{N}$.
Taking the push-forward means to skip the first $\tau_{l}$ time steps, and we
need those to be negligible compared to the variable $\mu(A_{l})\varphi
_{A_{l}}$ itself, meaning that $\mu(A_{l})\tau_{l}\rightarrow0$. Also, one has
to check that in skipping these steps, we do not miss (with positive
asymptotic probability) the awaited visit to $A_{l}$, which amounts to
requiring that $\mu_{A_{l}}(\varphi_{A_{l}}\leq\tau_{l})\longrightarrow0$.

Note that the existence of $\tau_{l}$ meeting the last two conditions is in
fact necessary for (\ref{Eq_ExpRTS}) (otherwise, the limit law contains an
atom at the origin), and that this is a property (\textquotedblleft no short
returns\textquotedblright) of the specific sequence $(A_{l})$ of sets, which
always needs to be checked (since every system contains sequences for which it
fails). The correlation decay, on the other hand, is a feature of the whole
system, and the two are tied together by the requirement that $\mu_{A_{l}%
}\circ T^{-\tau_{l}}\in\mathfrak{N}$ for all $l\geq1$.%

\vspace{0.3cm}%

The first purpose of the present paper is to point out that the same strategy
can be used, even for functional versions for the processes of consecutive
hitting-times, without assuming any information on the decay of correlations
(and without the system being mixing). We only ask for a sequence of (not
necessarily constant) delay times $\tau_{l}$ satisfying the necessary
conditions $\mu(A_{l})\tau_{l}\rightarrow0$ and $\mu_{A_{l}}(\varphi_{A_{l}%
}\leq\tau_{l})\longrightarrow0$, and such that the time-$\tau_{l}$-measures
$\mu_{A_{l}}\circ T^{-\tau_{l}}$ belong to some set $\mathfrak{K}$ of
probabilities which is compact in total variation norm. This latter condition
can be seen as a short-time (since $\tau_{l}=o(1/\mu(A_{l}))$) decorrelation
property of $(A_{l})$.

Second, we show that the same approach can be used to analyse distributional
limits of the sequences of consecutive positions, inside the $A_{l}$, of
orbits upon their visits to these small sets, and of joint time-position
processes. Such results on spatiotemporal Poisson limits have recently been
introduced in \cite{PS2}. We also include a general discussion of some aspects
of spatial and spatio-temporal process limits in the abstract setup, and
illustrate our results in the context of some simple prototypical systems.%

\vspace{0.3cm}%

There is a large body of literature devoted to the asymptotics of return- and
hitting-time distributions. While some sources directly relevant for the
present work are mentioned in the bibliography, we do not attempt to provide
a complete overview. A closer look at earlier articles which derive similar limit
theorems on the basis of mixing conditions or on correlation decay reveals
that, effectively, these assumptions are often only used to obtain control on
the time scale mentioned above, see for example \cite{Ab}, \cite{AS},
\cite{CC}, \cite{FFT1} \cite{HP}, \cite{HY}, \cite{PS}, \cite{PS2}.\ The
present paper makes this more explicit, and emphasizes that in the presence of
mere ergodicity a comparatively weak form of short-time control suffices.%

\vspace{0.3cm}%

\section{General setup and preparations}%

\noindent
\textbf{Hitting- and return-times. Inducing.} Throughout, $(X,\mathcal{A}%
,\mu)$ is a probability space, and $T:X\rightarrow X$ is an ergodic $\mu
$-preserving map. Also, $A$ and $A_{l}$ will always denote measurable sets of
(strictly) \emph{positive} measure. By ergodicity and the Poincar\'{e}
recurrence theorem, the measurable \emph{(first) hitting time} function of $A
$, $\varphi_{A}:X\rightarrow\overline{\mathbb{N}}:=\{1,2,\ldots,\infty\}$ with
$\varphi_{A}(x):=\inf\{n\geq1:T^{n}x\in A\}$, is finite a.e. on $X$. When
restricted to $A$ it is called the \emph{(first) return time} function of the
set. Define $T_{A}x:=T^{\varphi_{A}(x)}x$ for a.e. $x\in X$, which gives the
\emph{first entrance map} $T_{A}:X\rightarrow A$. It is a standard fact that
its restriction to $A$, the \emph{first return map} $T_{A}:A\rightarrow A$ is
an ergodic measure preserving map on the probability space $(A,\mathcal{A}\cap
A,\mu_{A})$, where $\mu_{A}(B):=\mu(A\cap B)/\mu(A)$, $B\in\mathcal{A}$. By
Kac' formula, $\int_{A}\varphi_{A}\,d\mu_{A}=1/\mu(A)$. That is, when regarded
as a random variable on $(A,\mathcal{A}\cap A,\mu_{A})$, the return time has
expectation $1/\mu(A)$, and we will often normalize these functions
accordingly, thus considering $\mu(A)\,\varphi_{A}$.

\emph{The focus of this work is on asymptotic distributions of such normalized
hitting (or return) times, and of the positions inside the target set at which
an orbit hits.} We shall study processes of consecutive hitting times and
hitting places in the limit of very small sets. Call $(A_{l})_{l\geq1}$ a
\emph{sequence of asymptotically rare events} (or an \emph{asymptotically rare
sequence}) provided that $A_{l}\in\mathcal{A}$ and $0<\mu(A_{l})\rightarrow0$.

It will be natural to view various observables defined on (parts of) $X$
through different probability measures $\nu$. In the present paper we shall
focus on the family $\mathfrak{P}:=\{\nu:$ probability measure on
$(X,\mathcal{A})$, $\nu\ll\mu\}$, equipped with the \emph{total variation
distance} $d_{\mathfrak{P}}(\nu,\nu^{\prime}):=2\sup_{A\in\mathcal{A}}\mid
\nu(A)-\nu^{\prime}(A)\mid$. This induces a topology so weak that mere
ergodicity ensures convergence of averages (via Theorem \ref{T_Yosida}), but
strong enough for the latter to have substantial consequences. (The results
below rely on compactness conditions which clearly would not be sufficient if
we used a w$^{\ast}$-topology instead.)

The push-forward of a measure $\nu$ by $T$ will be denoted $T_{\ast}\nu
:=\nu\circ T^{-1}$, and likewise for measurable maps other than $T$. Indeed,
we shall use suitable \emph{times}, that is, measurable functions
$\tau:B\rightarrow\mathbb{N}_{0}:=\{0,1,\ldots\}$ with $B\in\mathcal{A}$, to
define auxiliary \emph{induced maps} $T^{\tau}:B\rightarrow X$ via $T^{\tau
}x:=T^{\tau(x)}x$. Given $\nu\in\mathfrak{P}$, the push-forward $T_{\ast
}^{\tau}\nu:=\nu\circ(T^{\tau})^{-1}$ then is the distribution, at the
(possibly random) time $\tau$, of the process $(T^{n})_{n\geq0}$, all defined
on the probability space $(X,\mathcal{A},\nu)$.%

\vspace{0.3cm}%
%

\noindent
\textbf{Distributional convergence.} Let $(\mathfrak{E},d_{\mathfrak{E}})$ be
a compact metric space with Borel $\sigma$-algebra $\mathcal{B}_{\mathfrak{E}%
}$. As usual, a sequence $(Q_{l})_{l\geq1}$ of probability measures on
$(\mathfrak{E},\mathcal{B}_{\mathfrak{E}})$ is said to \emph{converge weakly}
to the probability measure $Q$ on $(\mathfrak{E},\mathcal{B}_{\mathfrak{E}})$,
written $Q_{l}\Longrightarrow Q$, if the integrals of all continuous real
functions $\chi$ on $\mathfrak{E}$ converge, $\int\chi\,dQ_{l}\longrightarrow
\int\chi\,dQ$ as $l\rightarrow\infty$ for $\chi\in\mathcal{C}(\mathfrak{E})$.
This is w$^{\ast}$-convergence in $\mathfrak{M}(\mathfrak{E})$, the set of
Borel probabilities on $\mathfrak{E}$, regarded as a subset of the space of
all finite signed Borel measures on $\mathfrak{E}$ (equipped with the total
variation norm), which by the Riesz representation theorem constitute the
topological dual space of $\mathcal{C}(\mathfrak{E})$.

If $R_{l}$, $l\geq1$, are measurable maps of $(X,\mathcal{A})$ into
$(\mathfrak{E},\mathcal{B}_{\mathfrak{E}})$, $\nu_{l}$ are probability
measures on $(X,\mathcal{A})$, and $R$ is another \emph{random element} of
$\mathfrak{E}$ (defined on some $(\Omega,\mathcal{F},\Pr)$), then we write
\begin{equation}
R_{l}\overset{\nu_{l}}{\Longrightarrow}R\text{ \quad as }l\rightarrow
\infty\label{Eq_NotationDistrCgeSeqMeas}%
\end{equation}
to indicate that $\mathrm{law}_{\nu_{l}}(R_{l}):=\nu_{l}\circ R_{l}%
^{-1}\Longrightarrow\mathrm{law}(R)=\Pr\circ R^{-1}$. This is
\emph{distributional convergence} to $R$ of the $R_{l}$ when the latter
functions are regarded as random variables on the probability spaces
$(X,\mathcal{A},\nu_{l})$, respectively. This includes the case of a single
measure $\nu$, where $R_{l}\overset{\nu}{\Longrightarrow}R$ means that the
distributions $\mathrm{law}_{\nu}(R_{l})=\nu\circ R_{l}^{-1}$ of the $R_{l}$
under $\nu$ converge weakly to the law of $R$.\newline

A sequence $R=(R^{(0)},R^{(1)},\ldots)$ of measurable functions $R^{(j)}%
:X\rightarrow\mathfrak{E}$ can be regarded as a single function into the
(compact) sequence space $\mathfrak{E}^{\mathbb{N}_{0}}=\{(r^{(j)})_{j\geq
0}:r^{(j)}\in\mathfrak{E}\}$, equipped with the \emph{product metric}
$d_{\mathfrak{E}^{\mathbb{N}_{0}}}(q,r):=\sum_{j\geq0}2^{-(j+1)}%
d_{\mathfrak{E}}(q^{(j)},r^{(j)})$.

Recall that weak convergence $Q_{l}\Longrightarrow Q$ in $\mathfrak{M}%
(\mathfrak{E}^{\mathbb{N}_{0}})$ of Borel probabilities on $\mathfrak{E}%
^{\mathbb{N}_{0}}$ is equvalent to convergence of all finite-dimensional
marginals, $\pi_{\ast}^{d}Q_{l}\Longrightarrow\pi_{\ast}^{d}Q$ in
$\mathfrak{M}(\mathfrak{E}^{d})$ for $d\geq1$, where $\pi^{d}:\mathfrak{E}%
^{\mathbb{N}_{0}}\rightarrow\mathfrak{E}^{d}$ denotes the canonical projection
onto the first $d$ factors.%

\vspace{0.3cm}%

\section{When do orbits hit small sets?\label{Sec_When}}%

\noindent
\textbf{Hitting-time and return-time processes.} To accommodate normalized
hitting times and their possible limits, we will use the (compact) target
space $(\mathfrak{E},d_{\mathfrak{E}})=([0,\infty],d_{[0,\infty]})$ with
$d_{[0,\infty]}(s,t):=\mid e^{-s}-e^{-t}\mid$. The sequence space
$\mathfrak{E}^{\mathbb{N}_{0}}=[0,\infty]^{\mathbb{N}_{0}}$ will be equipped
with the corresponding product metric $d_{[0,\infty]^{\mathbb{N}_{0}}}$ as above.

We first study, for sets $A$ as above, the random sequences of consecutive
return- and hitting-times, that is, we are going to consider the sequences
$\Phi_{A}:X\rightarrow\lbrack0,\infty]^{\mathbb{N}_{0}}$ of functions given
by
\begin{equation}
\Phi_{A}:=(\varphi_{A},\varphi_{A}\circ T_{A},\varphi_{A}\circ T_{A}%
^{2},\ldots)\quad\text{on }X\text{.} \label{Eq_DefHitProc}%
\end{equation}
When regarded as a random sequence defined on $(X,\mathcal{A},\nu)$, we shall
call $\Phi_{A}$ the \emph{hitting-time process} \emph{of} $A$ \emph{under}
$\nu$. If no measure is mentioned, this means that $\nu=\mu$. In case we
restrict $\Phi_{A}$ to $A$ and view it through $\mu_{A}$, we call it the
\emph{return-time process} of $A$. From the properties of $T_{A}$ on
$(A,\mathcal{A}\cap A,\mu_{A})$ it is immediate that
\begin{equation}
\text{any return-time process }\Phi_{A}\text{ is stationary and ergodic (under
}\mu_{A}\text{),} \label{Eq_bvxcnb}%
\end{equation}
and by relating return-time processes to hitting-time processes with different
initial measures, stationarity often carries over to limits of the latter. \ %

\vspace{0.3cm}%
%

\noindent
\textbf{Asymptotic hitting-time and return-time processes for rare events.}
Assume now that $(A_{l})_{l\geq1}$ is a sequence of asymptotically rare
events. It is immediate from (\ref{Eq_bvxcnb}) and Kac' formula that for any
random sequence $\widetilde{\Phi}=(\widetilde{\varphi}^{(0)},\widetilde
{\varphi}^{(1)},\ldots)$ in $[0,\infty]$,%
\begin{equation}
\text{if\quad}\mu(A_{l})\Phi_{A_{l}}\overset{\mu_{A_{l}}}{\Longrightarrow
}\widetilde{\Phi}\text{ as }l\rightarrow\infty\text{,\quad then\quad
}\widetilde{\Phi}\text{ is stationary with }\mathbb{E}[\widetilde{\varphi
}^{(0)}]\leq1\text{.} \label{Eq_BasicPropsAsyRetProc}%
\end{equation}
Beyond that, little can be said about the general \emph{asymptotic return-time
process} $\widetilde{\Phi}$. In fact, it has been shown in \cite{Z11} that
every stationary sequence $\widetilde{\Phi}$ with $\mathbb{E}[\widetilde
{\varphi}^{(0)}]\leq1$ does appear as the limit for a suitable asymptotically
rare sequence $(A_{l})$ if only $T$ acts on a nonatomic space $(X,\mathcal{A}%
,\mu)$.

Turning to \emph{asymptotic hitting-time processes} $\Phi$, that is,
distributional limits of hitting-time processes under one fixed probability
$\nu\in\mathfrak{P}$, we first recall that these do not depend on the
particular choice of $\nu$. (The following is Corollary 6 of \cite{Z7}.)

\begin{proposition}
[\textbf{Strong distributional convergence of }$\mu(A_{l})\Phi_{A_{l}}$%
]\label{P_SDCforHittiPs}Suppose that $(X,\mathcal{A},\mu,T)$ is an ergodic
probability preserving system, and $(A_{l})$ an asymptotically rare sequence
in $\mathcal{A}$. Let $\Phi$ be any random sequence in $[0,\infty]$. Then
\begin{equation}
\mu(A_{l})\Phi_{A_{l}}\overset{\nu}{\Longrightarrow}\,\Phi\text{ for
\emph{some} }\nu\in\mathfrak{P}\text{\qquad iff\qquad}\mu(A_{l})\Phi_{A_{l}%
}\overset{\nu}{\Longrightarrow}\,\Phi\text{ for \emph{all} }\nu\in
\mathfrak{P}\text{.}%
\end{equation}

\end{proposition}

Despite this, even if both exist, the asymptotic hitting-time process $\Phi$
for a given sequence $(A_{l})$ need not coincide with the asymptotic
return-time process $\widetilde{\Phi}$ for that very sequence. Indeed, the
relation between the two types of limit processes will be of central
importance in what follows.%

\vspace{0.3cm}%
%

\noindent
\textbf{Relating limit processes under }$\mu_{A_{l}}$ \textbf{to limit
processes under} $\mu$\textbf{. }It is well known that for any asymptotically
rare sequence $(A_{l})$, limit laws for the normalized first return-times,
$\mu(A_{l})\varphi_{A_{l}}$ under $\mu_{A_{l}}$, are intimately related to
limit laws of the normalized first hitting-times, $\mu(A_{l})\varphi_{A_{l}}$
under $\mu$ (see \cite{HaydnLacVai05}, \cite{AS}), and that this leads to an
efficient way of proving convergence (of both) to an exponential law. In
\cite{Z11} we have extended the crucial duality to processes $\mu(A_{l}%
)\Phi_{A_{l}}$, see Section \ref{S_PfsRH} below for more details.

A key ingredient of our present approach is the following generalization of
Proposition \ref{P_SDCforHittiPs} which provides conditions under which the
processes $\mu(A_{l})\Phi_{A_{l}}$, when started with suitable measures
$\nu_{l}$, exhibit the same asymptotic distributional behaviour as hitting
time processes started with $\mu$. The assumptions on the delay times
$\tau_{l}$ below are those already mentioned in the introduction, but we now
allow non-constant $\tau_{l}$. Parallel to (\ref{Eq_NotationDistrCgeSeqMeas})
we write, for measurable functions $R_{l},R:X\rightarrow\mathfrak{E}$,
\begin{equation}
R_{l}\overset{\nu_{l}}{\longrightarrow}R\text{ \quad as }l\rightarrow
\infty\label{Eq_NotationProbaCgeSeqMeas}%
\end{equation}
provided that $\nu_{l}(d_{\mathfrak{E}}(R_{l},R)>\varepsilon)\rightarrow0$ as
$l\rightarrow\infty$ whenever $\varepsilon>0$. This includes the case of a
single measure, $\nu_{l}=\nu$, in which case $R_{l}\overset{\nu}%
{\longrightarrow}R$ is the usual convergence in measure, $\nu(d_{\mathfrak{E}%
}(R_{l},R)>\varepsilon)\rightarrow0$ for $\varepsilon>0$. In the results to
follow, compact subsets $\mathfrak{K}$ of $(\mathfrak{P},d_{\mathfrak{P}})$
play the role of families of measures which only differ from $\mu$ in a
controllable way, as one can always assume w.l.o.g. that $\mu\in\mathfrak{K}$.

\begin{theorem}
[\textbf{Asymptotic hitting-time process - }$\nu_{l}$\textbf{\ versus }$\mu$%
]\label{T_ProcessesAsymptoticToHittingTimeProc}Let $(X,\mathcal{A},\mu,T)$ be
an ergodic probability preserving system, $(A_{l})_{l\geq1}$ a sequence of
asymptotically rare events, and $(\nu_{l})$ a sequence in $\mathfrak{P}$.
Assume that there are measurable functions $\tau_{l}:X\rightarrow
\mathbb{N}_{0}$ such that
\begin{equation}
\mu(A_{l})\,\tau_{l}\overset{\nu_{l}}{\longrightarrow}0\quad\text{as
}l\rightarrow\infty, \label{Eq_kllklklklklk}%
\end{equation}
and
\begin{equation}
\nu_{l}\left(  \tau_{l}<\varphi_{A_{l}}\right)  \longrightarrow1\quad\text{as
}l\rightarrow\infty, \label{Eq_opopopopo}%
\end{equation}
while there is some compact subset $\mathfrak{K}$ of $(\mathfrak{P}%
,d_{\mathfrak{P}})$ such that
\begin{equation}
T_{\ast}^{\tau_{l}}\nu_{l}\in\mathfrak{K}\text{ \quad for }l\geq1\text{,}
\label{Eq_hbhbhbhbhbkkkk}%
\end{equation}
Then, for any random element $\Phi$ of $[0,\infty]^{\mathbb{N}_{0}}$,
\begin{equation}
\mu(A_{l})\Phi_{A_{l}}\overset{\nu_{l}}{\Longrightarrow}\,\Phi\text{ \quad iff
\quad}\mu(A_{l})\Phi_{A_{l}}\overset{\mu}{\Longrightarrow}\,\Phi\text{\qquad
as }l\rightarrow\infty\text{.} \label{Eq_gvzgvzgvzgvzgvzgv}%
\end{equation}

\end{theorem}%

\vspace{0.1cm}%

\begin{remark}
For a constant sequence $(\nu_{l})=(\nu)$ we can take $\tau_{l}:=0$ for all
\thinspace$l$, and obtain Proposition \ref{P_SDCforHittiPs}. Given any
sequence $(\nu_{l})$ in $\mathfrak{P}$, the $\tau_{l}=0$ case of the theorem
shows that (\ref{Eq_gvzgvzgvzgvzgvzgv}) holds whenever all the $\nu_{l}$
belong to some compact subset $\mathfrak{K}$ of $\mathfrak{P}$.
\end{remark}

%

\vspace{0.3cm}%
%

\noindent
\textbf{Convergence to iid exponential limit processes.} In the most prominent
case the limit process is an iid sequence of normalized exponentially
distributed random variables, henceforth denoted by $\Phi_{\mathrm{Exp}}$.
This is the process of interarrival times of an elementary standard Poisson
(counting) process, and we shall say that $(A_{l})$ \emph{exhibits Poisson
asymptotics} if
\begin{equation}
\mu(A_{l})\Phi_{A_{l}}\overset{\mu}{\Longrightarrow}\,\Phi_{\mathrm{Exp}%
}\text{ \quad and \quad}\mu(A_{l})\Phi_{A_{l}}\overset{\mu_{A_{l}}%
}{\Longrightarrow}\,\Phi_{\mathrm{Exp}}\text{.}
\label{Eq_DefPoissonAsymptotics}%
\end{equation}
This terminology is further justified by

\begin{remark}
[\textbf{Convergence of associated counting processes}]\label{R_CgeCounting1}%
Any set $A\in\mathcal{A}$ of positive measure comes with an associated
\emph{normalized counting process} $\mathrm{N}_{A}:X\rightarrow\mathcal{D}%
[0,\infty)$ (the collection of cadlag paths $\mathrm{x}=(\mathrm{x}%
_{t})_{t\geq0}:[0,\infty)\rightarrow\mathbb{R}$) given by $\mathrm{N}%
_{A}=(\mathrm{N}_{A,t})_{t\geq0}$ with
\begin{equation}
\mathrm{N}_{A,t}:=%
{\textstyle\sum_{k=1}^{\left\lfloor t/\mu(A)\right\rfloor }}
1_{A}\circ T^{k}=%
{\textstyle\sum_{j\geq1}}
1_{\left[  \mu(A)\sum_{i=0}^{j-1}\varphi_{A}\circ T_{A}^{i},\infty\right)
}(t)\text{.} \label{Eq_DefNormalizedCountingProc}%
\end{equation}
Observe then that for any probability measures $\nu_{l}$ on $(X,\mathcal{A})
$,
\begin{equation}
\mu(A_{l})\Phi_{A_{l}}\overset{\nu_{l}}{\Longrightarrow}\,\Phi_{\mathrm{Exp}%
}\text{ \quad implies \quad}\mathrm{N}_{A_{l}}\overset{\nu_{l}}%
{\Longrightarrow}\mathrm{N}\text{ in }(\mathcal{D}[0,\infty),\mathcal{J}%
_{1})\text{,} \label{Eq_CgeOfCountingProcFollows}%
\end{equation}
where $\mathrm{N}=(\mathrm{N}_{t})_{t\geq0}$ denotes a standard Poisson
counting process. To see this, define $\mathcal{W}:=\{(\varphi^{(i)}%
)\in\lbrack0,\infty)^{\mathbb{N}_{0}}:\sum_{i=0}^{j-1}\varphi^{(i)}%
\nearrow\infty$ as $j\rightarrow\infty\}$, and $\Upsilon:[0,\infty
]^{\mathbb{N}_{0}}\rightarrow\mathcal{D}[0,\infty)$ by $\Upsilon
((\varphi^{(i)})_{i\geq0}):=\sum_{j\geq1}1_{\left[  \sum_{i=0}^{j-1}%
\varphi^{(i)},\infty\right)  }$ for $(\varphi^{(i)})_{i\geq0}\in\mathcal{W}$,
and $\Upsilon((\varphi^{(i)})_{i\geq0}):=0$ otherwise. An elementary argument
shows that $\Upsilon$ is continuous on $\mathcal{W}_{+}:=\mathcal{W}%
\cap(0,\infty)^{\mathbb{N}_{0}}$ for the Skorokhod $\mathcal{J}_{1}$-topology
of $\mathcal{D}[0,\infty)$ (see \cite{Bi}, \cite{Whitt}). Since $\Upsilon
(\Phi_{\mathrm{Exp}})$ has the same law as $\mathrm{N}$, and $\Phi
_{\mathrm{Exp}}\in\mathcal{W}_{+}$ almost surely,
(\ref{Eq_CgeOfCountingProcFollows}) thus follows by the \textquotedblleft
extended continuous mapping theorem\textquotedblright\ (Theorem 2.7 in
\cite{Bi}).
\end{remark}

\begin{remark}
[\textbf{Convergence of associated point processes}]\label{R_CgePoint1}%
Alternatively, consider the associated \emph{normalized point process of
visiting times}, $\mathcal{N}_{A}:X\rightarrow\mathcal{M}_{p}[0,\infty)$ (the
collection of Radon point measures $\mathrm{n}:\mathcal{B}_{[0,\infty
)}\rightarrow\{0,1,2,\ldots,\infty\}$, equipped with the topology of vague
convergence, see Chapter 3 of \cite{Resnick}) given by
\begin{equation}
\mathcal{N}_{A}:=%
{\textstyle\sum_{j\geq1}}
\delta_{\{\mu(A)\sum_{i=0}^{j-1}\varphi_{A}\circ T_{A}^{i}\}}\text{,}%
\end{equation}
where $\delta_{\{s\}}$ denotes the unit point mass at $s$. That is,
$\mathcal{N}_{A}(x)$ has distribution function $t\mapsto\mathrm{N}_{A,t}(x)$.
Then, for arbitrary probability measures $\nu_{l}$ on $(X,\mathcal{A})$,
\begin{equation}
\mu(A_{l})\Phi_{A_{l}}\overset{\nu_{l}}{\Longrightarrow}\,\Phi_{\mathrm{Exp}%
}\text{ \quad implies \quad}\mathcal{N}_{A_{l}}\overset{\nu_{l}}%
{\Longrightarrow}\mathrm{PRM}(\lambda_{\lbrack0,\infty)}^{1})\text{ in
}\mathcal{M}_{p}[0,\infty)\text{,}%
\end{equation}
with $\mathrm{PRM}(\lambda_{\lbrack0,\infty)}^{1})$ denoting the Poisson
random measure of intensity $\lambda^{1}$ on $[0,\infty)$. This, too, follows
via the \textquotedblleft extended continuous mapping
theorem\textquotedblright\ since it is easily seen that the map $\Theta
:[0,\infty]^{\mathbb{N}_{0}}\rightarrow\mathcal{M}_{p}[0,\infty)$ with
$\Theta((\varphi^{(i)})_{i\geq0}):=%
{\textstyle\sum_{j\geq1}}
\delta_{\{\sum_{i=0}^{j-1}\varphi^{(i)}\}}$ for $(\varphi^{(i)})_{i\geq0}%
\in\mathcal{W}$, and $\Theta((\varphi^{(i)})_{i\geq0}):=0$ otherwise, is
continuous on $\mathcal{W}$, while $\Theta(\Phi_{\mathrm{Exp}})$ has the same
law as $\mathrm{PRM}(\lambda_{\lbrack0,\infty)}^{1})$.
\end{remark}

We will show that this happens in the $\nu_{l}=\mu_{A_{l}}$ case of the
previous theorem:

\begin{theorem}
[\textbf{Convergence to an iid exponential sequence}]%
\label{T_MOreFlexiblePoissonLimitFromCompactness}Let $(X,\mathcal{A},\mu,T)$
be an ergodic probability preserving system, and $(A_{l})_{l\geq1}$ a sequence
of asymptotically rare events. Let $\nu_{l}:=\mu_{A_{l}}$ and assume that
there are measurable functions $\tau_{l}:X\rightarrow\mathbb{N}_{0}$ such
that
\begin{equation}
\mu(A_{l})\,\tau_{l}\overset{\nu_{l}}{\longrightarrow}0\quad\text{as
}l\rightarrow\infty, \label{Eq_DontWaitTooLong1}%
\end{equation}
and
\begin{equation}
\nu_{l}\left(  \tau_{l}<\varphi_{A_{l}}\right)  \longrightarrow1\quad\text{as
}l\rightarrow\infty, \label{Eq_HardlyAnyReturnsUpToTimeM1}%
\end{equation}
while there is some compact subset $\mathfrak{K}$ of $(\mathfrak{P}%
,d_{\mathfrak{P}})$ such that
\begin{equation}
T_{\ast}^{\tau_{l}}\nu_{l}\in\mathfrak{K}\text{ \quad for }l\geq1\text{,}
\label{Eq_TheFabulousCompactnessCondition1}%
\end{equation}
Then $(A_{l})$ exhibits Poisson asymptotics,
\begin{equation}
\mu(A_{l})\Phi_{A_{l}}\overset{\mu_{A_{l}}}{\Longrightarrow}\,\Phi
_{\mathrm{Exp}}\text{ \quad and \quad}\mu(A_{l})\Phi_{A_{l}}\overset{\mu
}{\Longrightarrow}\,\Phi_{\mathrm{Exp}}\text{ \quad as }l\rightarrow
\infty\text{.} \label{Eq_ExpoLimitInMyExpoLimitThm1}%
\end{equation}
More generally, (\ref{Eq_ExpoLimitInMyExpoLimitThm1}) follows provided that
for every $\varepsilon>0$ there are a sequence $(\nu_{l})_{l\geq1}$ in
$\mathfrak{P}$ with $d_{\mathfrak{P}}(\nu_{l},\mu_{A_{l}})<\varepsilon$ for
all $l$, a compact set $\mathfrak{K}\subseteq\mathfrak{P}$, and measurable
functions $\tau_{l}:X\rightarrow\mathbb{N}_{0}$ such that
(\ref{Eq_DontWaitTooLong1})-(\ref{Eq_TheFabulousCompactnessCondition1}) hold.
\end{theorem}

%

\vspace{0.3cm}%

We shall see that this opens up a very easy way of proving Poisson asymptotics
in several interesting situations.%

\vspace{0.3cm}%

\begin{remark}
\textbf{a)} An obvious necessary condition for the first component $\mu
(A_{l})\,\varphi_{A_{l}}$ to have an exponential limit law is the absence of a
point mass at zero in the limit. That is, for any sequence $(\tau_{l})$
satisfying condition (\ref{Eq_DontWaitTooLong1}) we need to have
(\ref{Eq_HardlyAnyReturnsUpToTimeM1}).\newline\textbf{b)} Note that condition
(\ref{Eq_TheFabulousCompactnessCondition1}) only uses information on what
happens to $A_{l}$ before the time $\tau_{l}$ which is of order $o(1/\mu
(A_{l}))$ as $l\rightarrow\infty$. We do not assume mixing (let alone any
quantitative mixing conditions), but only use what little asymptotic
information follows from ergodicity alone. \newline\textbf{c)} The assumptions
(\ref{Eq_DontWaitTooLong1})-(\ref{Eq_TheFabulousCompactnessCondition1}) make
precise the condition that the system should forget, sufficiently fast,
whether or not it started in $A_{l}$. We can regard (\ref{Eq_DontWaitTooLong1}%
)-(\ref{Eq_TheFabulousCompactnessCondition1}) as a short-time decorrelation
(or mixing) condition. \newline\textbf{d)} Of course, mixing properties of
specific systems can still be very useful for validating conditions
(\ref{Eq_DontWaitTooLong1})-(\ref{Eq_TheFabulousCompactnessCondition1}).
However, in our discusion of examples in Section \ref{Sec_GMmaps}, we make a
point of not using any asymptotic mixing properties for this purpose.
\newline\textbf{e)} Allowing measures $\nu_{l}$ more general than $\mu_{A_{l}%
}$ in the final statement of the theorem sometimes enables us to replace the
density $\mu(A_{l})^{-1}1_{A_{l}}$ of $\mu_{A_{l}}$ by an approximating
density of higher regularity for which (\ref{Eq_DontWaitTooLong1}%
)-(\ref{Eq_TheFabulousCompactnessCondition1}) are easier to verify (e.g. if
they belong to a space on which the transfer operator is well
understood).\newline\textbf{f)} Another way of using this flexibility is to
replace $\mu_{A_{l}}$ by $\mu_{A_{l}^{\prime}}$\ for nicer sets $A_{l}%
^{\prime}\in\mathcal{A}$. This works if, for every $\widetilde{\varepsilon}%
>0$, one can pick a sequence $(A_{l}^{\prime})$ such that $\mu_{A_{l}}%
(A_{l}\bigtriangleup A_{l}^{\prime})<\widetilde{\varepsilon}$ for all $l\geq1$
while the $\nu_{l}:=\mu_{A_{l}^{\prime}}$ admit $\mathfrak{K}$ and $\tau_{l}$
satisfying (\ref{Eq_DontWaitTooLong1}%
)-(\ref{Eq_TheFabulousCompactnessCondition1}).
\end{remark}

%

\vspace{0.3cm}%
%

\noindent
\textbf{Allowing immediate returns.} While for many important classes of
concrete dynamical systems one typically observes Poisson asymptotics for
natural families of rare events (cylinders or general $\varepsilon$-balls
shrinking to a typical point $x^{\ast}$), there are often distinguished
exceptional points $x^{\ast}$, like the periodic points of the system, to
which a positive proportion $1-\theta$ with $\theta\in(0,1)$ of a
neighbourhood can return after a fixed number of steps. This will result in a
point mass at zero in the limit of return time distributions. If the situation
is nice otherwise, the part which did escape in the first step may return
after a rescaled exponential time, so that $\mu(A_{l})\varphi_{A_{l}}%
\overset{\mu_{A_{l}}}{\Longrightarrow}\widetilde{\varphi}$, where the limit
variable $\widetilde{\varphi}$ is characterized by the distribution function
\begin{equation}
\widetilde{F}_{(\mathrm{Exp},\theta)}(t):=(1-\theta)+\theta(1-e^{-\theta
t})\text{, \quad}t\geq0\text{.} \label{Eq_DefOneDimDFComp}%
\end{equation}
Turning to processes, for any $\theta\in(0,1)$ we let $\widetilde{\Phi
}_{(\mathrm{Exp},\theta)}$ denote an iid sequence of random variables, each
distributed according to $\widetilde{F}_{(\mathrm{Exp},\theta)}$.

The following complement to Theorem
\ref{T_MOreFlexiblePoissonLimitFromCompactness} covers such situations. Here
and later, when given a sequence $(\nu_{l})$ of probabilities on
$(X,\mathcal{A})$ and some $B\in\mathcal{A}$, we shall simply write $\nu
_{l,B}$ for the normalized restriction given by $\nu_{l,B}(A):=\nu_{l}%
(B)^{-1}\nu_{l}(B\cap A)$, $A\in\mathcal{A}$.

\begin{theorem}
[\textbf{Convergence to an iid }$\widetilde{F}_{(\mathrm{Exp},\theta)}%
$\textbf{\ sequence}]\label{T_CPoissonViaCompactness}Let $(X,\mathcal{A}%
,\mu,T)$ be an ergodic probability preserving system, $(A_{l})_{l\geq1}$ a
sequence of asymptotically rare events. Let $\nu_{l}:=\mu_{A_{l}}$ and suppose
that $A_{l}=A_{l}^{\bullet}\cup A_{l}^{\circ}$ (disjoint) with
\begin{equation}
\nu_{l}(A_{l}^{\circ})\longrightarrow\theta\in(0,1)\text{ }\quad\text{as
}l\rightarrow\infty. \label{Eq_CompundingTheProportion}%
\end{equation}
Set $\nu_{l}^{\circ}:=\nu_{l,A_{l}^{\circ}}$ and $\nu_{l}^{\bullet}%
:=\nu_{l,A_{l}^{\bullet}}$ (the normalized restrictions of $\nu_{l}$ to
$A_{l}^{\circ}$ and $A_{l}^{\bullet}$, respectively). Assume further that
there are measurable functions $\tau_{l}:X\rightarrow\mathbb{N}_{0}$, $l\geq
1$, such that
\begin{equation}
\mu(A_{l})\,\tau_{l}\overset{\nu_{l}}{\longrightarrow}0\quad\text{as
}l\rightarrow\infty, \label{Eq_DontWaitTooLong2}%
\end{equation}
and
\begin{equation}
\nu_{l}^{\circ}\left(  \tau_{l}<\varphi_{A_{l}}\right)  \longrightarrow
1\quad\text{as }l\rightarrow\infty, \label{Eq_HardlyAnyReturnsUpToTimeM2}%
\end{equation}
while there is some compact subset $\mathfrak{K}$ of $(\mathfrak{P}%
,d_{\mathfrak{P}})$ such that
\begin{equation}
T_{\ast}^{\tau_{l}}\nu_{l}^{\circ}\in\mathfrak{K}\text{ \quad for }%
l\geq1\text{,} \label{Eq_TheFabulousCompactnessCondition2}%
\end{equation}
whereas
\begin{equation}
\nu_{l}^{\bullet}\left(  \tau_{l}\geq\varphi_{A_{l}}\right)  \longrightarrow
1\quad\text{as }l\rightarrow\infty, \label{Eq_ManyHappyReturns}%
\end{equation}
and
\begin{equation}
d_{\mathfrak{P}}((T_{A_{l}})_{\ast}\nu_{l}^{\bullet},\nu_{l})\longrightarrow
0\text{\quad\ as }l\rightarrow\infty\text{.} \label{Eq_ControlShortReturns}%
\end{equation}
Then,
\begin{equation}
\mu(A_{l})\Phi_{A_{l}}\overset{\mu_{A_{l}}}{\Longrightarrow}\,\widetilde{\Phi
}_{(\mathrm{Exp},\theta)}\text{ \quad as }l\rightarrow\infty\text{.}
\label{Eq_yxcvbnmnbvcxy}%
\end{equation}
More generally, (\ref{Eq_yxcvbnmnbvcxy}) follows provided that for every
$\varepsilon>0$ there are a sequence $(\nu_{l})_{l\geq1}$ in $\mathfrak{P}$
with $d_{\mathfrak{P}}(\nu_{l},\mu_{A_{l}})<\varepsilon$ for all $l$, a
compact set $\mathfrak{K}\subseteq\mathfrak{P}$, and measurable functions
$\tau_{l}:X\rightarrow\mathbb{N}_{0}$ such that (\ref{Eq_DontWaitTooLong2}%
)-(\ref{Eq_ControlShortReturns}) hold.
\end{theorem}

%

\vspace{0.2cm}%

In Section \ref{Sec_GMmaps} we illustrate how this can be used very easily in
some standard situations.

\section{Where do orbits hit small sets?}%

\noindent
\textbf{Local observables and local processes.} Let $(X,\mathcal{A},\mu,T)$ be
an ergodic probability preserving system. We introduce a large class of random
processes associated to the visits of orbits to a given small set
$A\in\mathcal{A}$. The idea is to focus on what exactly happens upon each
visit, and record the position inside $A$ by means of some function $\psi_{A}$
on this set. As we are interested in small sets and, ultimately, in limits as
the size of the sets tends to zero, it is natural to consider functions
encoding the \emph{relative} position inside $A$, thus effectively rescaling
the set.

For instance, if the relevant sets $A$ are subintervals of some larger
interval $X$, the \emph{normalizing interval charts} $\psi_{A}:A\rightarrow
\lbrack0,1]$ with $\psi_{A}(x):=(x-a)/(b-a)$ for $A=[a,b]$, are a natural
choice. This is the prototypical example to keep in mind, but for the general
theory we simply allow measurable maps $\psi_{A}:A\rightarrow\mathfrak{Z}$,
not necessarily invertible, into some space $\mathfrak{Z}$ which does not
depend on $A$.

In the following we fix some compact metric space $(\mathfrak{Z}%
,d_{\mathfrak{Z}})$ with Borel $\sigma$-algebra $\mathcal{B}_{\mathfrak{Z}}$
to represent the relative positions (or some other relevant aspect) of points
inside the distinguished small sets we wish to study. Any $\mathcal{A}%
$-$\mathcal{B}_{\mathfrak{Z}}$-measurable map $\psi_{A}:A\rightarrow
\mathfrak{Z}$ will be called an ($\mathfrak{Z}$-valued) \emph{local
observable} on $A$, and we shall use the uppercase $\Psi_{A}$ to denote the
sequence of consecutive local observations of an orbit which starts anywhere
in $X$,
\begin{equation}
\Psi_{A}:X\rightarrow\mathfrak{Z}^{\mathbb{N}}\text{, \quad}\Psi_{A}%
:=(\psi_{A}\circ T_{A},\psi_{A}\circ T_{A}^{2},\ldots)\text{.}
\label{Eq_LocProc}%
\end{equation}
Equip $\mathfrak{Z}^{\mathbb{N}}$ with its compact Polish product topology,
induced by the product metric $d_{\mathfrak{Z}^{\mathbb{N}}}$, then $\Psi_{A}$
is $\mathcal{A}$-$\mathcal{B}_{\mathfrak{Z}^{\mathbb{N}}}$-measurable (and can
thus be regarded as a single $\mathfrak{Z}^{\mathbb{N}}$-valued local
observable on $A$). We can include the local observable at time zero provided
that the orbit starts in $A$. To this end, define
\begin{equation}
\widetilde{\Psi}_{A}:A\rightarrow\mathfrak{Z}^{\mathbb{N}_{0}}\text{, \quad
}\widetilde{\Psi}_{A}:=(\psi_{A},\psi_{A}\circ T_{A},\psi_{A}\circ T_{A}%
^{2},\ldots)\text{.}%
\end{equation}
Given any probability measure $\nu$ on $(X,\mathcal{A})$, we can view
$\Psi_{A}$ as a random process on the probability space $(X,\mathcal{A},\nu)$.
If $\nu$ is concentrated on $A$, the same is true for $\widetilde{\Psi}_{A}$.
We shall refer to either variant as a \emph{local process under }$\nu$. Again,
the properties of $T_{A}$ on $(A,\mathcal{A}\cap A,\mu_{A})$ entail that
\begin{equation}
\text{any local process }\widetilde{\Psi}_{A}\text{ is stationary and ergodic
under }\mu_{A}\text{,} \label{Eq_fcsad}%
\end{equation}
and we shall exploit this by linking local processes with different initial
measures to this particular version. By (\ref{Eq_fcsad}), $\mathrm{law}%
_{\mu_{A}}(\widetilde{\Psi}_{A})=\mathrm{law}_{\mu_{A}}(\Psi_{A})$, so that we
need not distinguish between the two variants as far as the laws of individual
processes under the corresponding measures $\mu_{A}$ are concerned.

\begin{remark}
[\textbf{Normalized return times and processes as local observables and
processes}]The normalized return times $\mu(A)\varphi_{A}:A\rightarrow
\lbrack0,\infty]$ studied in the previous section also define local
observables, somewhat special in that they are defined in terms of the
dynamics. Moreover, the processes $\mu(A)\Phi_{A}:A\rightarrow\lbrack
0,\infty]^{\mathbb{N}_{0}}$ are particular local processes: using the notation
introduced above, we have $\mu(A)\Phi_{A}=\widetilde{\Psi}_{A}$ for $\psi
_{A}:=\mu(A)\varphi_{A}$.
\end{remark}

\begin{remark}
[\textbf{Warning regarding normalized hitting processes}]In contrast, the
normalized hitting time process $\mu(A)\Phi_{A}:X\rightarrow\lbrack
0,\infty]^{\mathbb{N}_{0}}$ with $\Phi_{A}$ as defined in (\ref{Eq_DefHitProc}%
) is \emph{not} a local process $\Psi_{A}$ as in (\ref{Eq_LocProc}). Indeed,
$\Phi_{A}$ is not constant on $\{x,Tx,\ldots,T^{\varphi_{A}(x)-1}x\}$ for
$x\in A^{c}$. This is why we do not need an analogue of condition
(\ref{Eq_kllklklklklk}) in Theorem \ref{T_AsyIntStateProc} below.
\end{remark}

%

\vspace{0.3cm}%
%

\noindent
\textbf{Local processes for asymptotically rare events.} Assume now that
$(A_{l})_{l\geq1}$ is a sequence of asymptotically rare events, and that for
each $A_{l}$ we are given a local observable $\psi_{A_{l}}:A\rightarrow
\mathfrak{Z}$. Our goal is to provide useful conditions under which the
sequence of local processes $(\widetilde{\Psi}_{A_{l}})_{l\geq1}$ or
$(\Psi_{A_{l}})_{l\geq1}$ converges in distribution as $l\rightarrow\infty$.
Here, again, it makes sense to study these random variables either through one
fixed initial probability $\nu$ (in case of $(\Psi_{A_{l}})_{l\geq1}$), say
$\nu=\mu$, or to view them through the sequence $(\mu_{A_{l}})$ of normalized
restrictions to these sets.%

\vspace{0.3cm}%

We first look at the $\widetilde{\Psi}_{A_{l}}$ under the measures $\mu
_{A_{l}}$. Due to (\ref{Eq_fcsad}) we see that for any random sequence
$\widetilde{\Psi}$ in $\mathfrak{Z}$,
\begin{equation}
\text{if\quad}\widetilde{\Psi}_{A_{l}}\overset{\mu_{A_{l}}}{\Longrightarrow
}\widetilde{\Psi}\text{ as }l\rightarrow\infty\text{,\quad then\quad
}\widetilde{\Psi}\text{ is stationary.} \label{Eq_uoljioljo1}%
\end{equation}
Beyond that, little can be said about the general \emph{asymptotic local
process} $\widetilde{\Psi}$. In fact, we are going to show that unless the
system acts on a discrete space (and hence is essentially a cyclic
permutation), every $\mathfrak{Z}$-valued stationary sequence arises as the
limit of local processes for any given sequence $(A_{l})$ if only we use
suitable local observables $\psi_{A_{l}}$. (This is parallel to Theorem 2.1 of
\cite{Z11}.) In particular, $\widetilde{\Psi}$ need not be independent, and
doesn't even have to be ergodic.

\begin{theorem}
[\textbf{Prescribing the asymptotic internal state process}]%
\label{T_PrescribingPSI}Let $T$ be an ergodic measure preserving map on the
nonatomic probability space $(X,\mathcal{A},\mu)$, let $(A_{l})$ be an
asymptotically rare sequence in $\mathcal{A}$, and let $\widetilde{\Psi}$ be
any $\mathfrak{Z}$-valued stationary sequence. Then there is a sequence
$(\psi_{A_{l}})$ of local observables for the $A_{l}$ such that
\begin{equation}
\widetilde{\Psi}_{A_{l}}\overset{\mu_{A_{l}}}{\Longrightarrow}\widetilde{\Psi
}\text{ \quad as }l\rightarrow\infty\text{.}
\label{Eq_dhbsdhfsdhfsdhdbhhbjhjbhbhbhb}%
\end{equation}

\end{theorem}%

\vspace{0.3cm}%

Observe that the distributions $\mathrm{law}_{\mu_{A_{l}}}(\psi_{A_{l}})$ of
the first components of the $\widetilde{\Psi}_{A_{l}}$ may not involve any
dynamics, but their convergence is of course necessary for convergence of the
processes as in (\ref{Eq_dhbsdhfsdhfsdhdbhhbjhjbhbhbhb}). For the abstract
theory we will therefore take the assumption
\begin{equation}
\psi_{A_{l}}\overset{\mu_{A_{l}}}{\Longrightarrow}\psi\text{ \qquad as
}l\rightarrow\infty\text{,} \label{Eq_cnbvmcbvmbmybvmyvbyb}%
\end{equation}
that they converge to the law of some particular random element $\psi$ of
$\mathfrak{Z}$, as our starting point. (Note that this is not particularly
restrictive. By compactness of $\mathfrak{M}(\mathfrak{Z})$, every sequence
contains a subsequence along which (\ref{Eq_cnbvmcbvmbmybvmyvbyb}) is
satisfied.) The question will then be under what conditions
(\ref{Eq_cnbvmcbvmbmybvmyvbyb}) entails convergence of the processes
$\Psi_{A_{l}}$ to some (or some particular) random sequence $\Psi$.

In some natural situations, condition (\ref{Eq_cnbvmcbvmbmybvmyvbyb}) relates
the local observables $\psi_{A_{l}}$ to the local regularity of $\mu$ on the
$A_{l}$:

\begin{example}
[\textbf{Normalizing interval charts}]Assume that $X$ is an interval, and
$\mu$ is absolutely continuous w.r.t. Lebesgue measure $\lambda$. Suppose that
the $A_{l}$ are subintervals which shrink to a distinguished point $x^{\ast
}\in X$ at which (a suitable version of) the density $d\mu/d\lambda$ is
continuous and strictly positive. Define $\psi_{A_{l}}:A_{l}\rightarrow
\lbrack0,1]$ to be the normalizing interval chart from the previous
subsection. Then (\ref{Eq_cnbvmcbvmbmybvmyvbyb}) holds, with the limit
variable $\psi$ uniformly distributed on $\mathfrak{Z}=[0,1]$.
\end{example}

%

\vspace{0.3cm}%

Considering a sequence of local processes under one probability measure $\nu$
which doesn't depend on $l$, we usually lose stationarity, but gain the
possibility of freely switching measures.

\begin{proposition}
[\textbf{Strong distributional convergence of }$\Psi_{A_{l}}$]%
\label{P_SDCInternal}Let $(X,\mathcal{A},\mu,T)$ be an ergodic probability
preserving system, $(A_{l})$ an asymptotically rare sequence in $\mathcal{A}$,
with $(\Psi_{A_{l}})_{l\geq1}$ a sequence of $\mathfrak{Z}$-valued local
processes for the $A_{l}$. Then%
\begin{equation}
\Psi_{A_{l}}\overset{\nu}{\Longrightarrow}\,\Psi\text{ for \emph{some} }\nu
\in\mathfrak{P}\text{\qquad iff\qquad}\Psi_{A_{l}}\overset{\nu}%
{\Longrightarrow}\,\Psi\text{ for \emph{all} }\nu\in\mathfrak{P}\text{.}%
\end{equation}

\end{proposition}

This is an immediate consequence of \cite{Z7}, see the start of Section
\ref{Sec_PfInternalStates} for details.\ Variations on this theme will be the
key to the limit theorems below.%

\vspace{0.3cm}%
%

\noindent
\textbf{Relating limit processes under }$\mu_{A_{l}}$ \textbf{to limit
processes under} $\mu$\textbf{. }As mentioned before, the intimate relation
between return- and hitting times, that is, the relation between the laws of
$\varphi_{A}$ under $\mu_{A}$ and $\mu$ respectively, is often crucial for the
analysis of these variables. For general local observables there is no such
principle:
\begin{equation}
\text{In general, }\psi_{A_{l}}\overset{\mu_{A_{l}}}{\Longrightarrow}%
\psi\text{ does not imply convergence of }\psi_{A_{l}}\circ T_{A_{l}}\text{
under }\mu\text{.} \label{Eq_AWarningLocObs}%
\end{equation}

\begin{example}
Let $(X,\mathcal{A},\mu,T)$ be an ergodic probability preserving system,
$(A_{l}^{\prime})$ an asymptotically rare sequence in $\mathcal{A}$ such that
$A_{l}^{\prime}\cap T^{-1}A_{l}^{\prime}=\varnothing$ for $l\geq1$. Set
$A_{l}:=T^{-1}A_{l}^{\prime}\cup A_{l}^{\prime}$, and $\psi_{A_{l}}%
:A_{l}\rightarrow\{0,1\}$ with $\psi_{A_{2l-1}}:=1_{A_{l}^{\prime}}$ while
$\psi_{A_{2l}}:=1_{T^{-1}A_{l}^{\prime}}$. Then $\psi_{A_{l}}\overset
{\mu_{A_{l}}}{\Longrightarrow}\psi$, with $\psi$ denoting a fair coin, while
$\psi_{A_{2l-1}}\circ T_{A_{l}}\overset{\mu}{\Longrightarrow}0$ and
$\psi_{A_{2l}}\circ T_{A_{l}}\overset{\mu}{\Longrightarrow}1$ since $T_{A_{l}%
}=T_{T^{-1}A_{l}^{\prime}}$ outside $T^{-1}A_{l}^{\prime}$.
\end{example}

Nonetheless, we can provide a very useful condition which ensures that
possibly localized measures $\nu_{l}$ can be replaced by any fixed probablity
$\nu\in\mathfrak{P}$.

\begin{theorem}
[\textbf{Asymptotic local process - }$\nu_{l}$\textbf{\ versus }$\mu$%
]\label{T_AsyIntStateProc}Let $(X,\mathcal{A},\mu,T)$ be an ergodic
probability preserving system, $(A_{l})$ an asymptotically rare sequence in
$\mathcal{A}$, and $(\psi_{A_{l}})_{l\geq1}$ a sequence of $\mathfrak{Z}%
$-valued local observables for the $A_{l}$, with corresponding local processes
$\Psi_{A_{l}}$. Let $(\nu_{l})$ be a sequence in $\mathfrak{P}$.

Assume that there are measurable $\tau_{l}:X\rightarrow\mathbb{N}_{0}$ such
that
\begin{equation}
\nu_{l}\left(  \tau_{l}<\varphi_{A_{l}}\right)  \longrightarrow1\quad\text{as
}l\rightarrow\infty\text{,} \label{Eq_hghghghgghghghghdhsf}%
\end{equation}
while there is some compact subset $\mathfrak{K}$ of $(\mathfrak{P}%
,d_{\mathfrak{P}})$ such that
\begin{equation}
T_{\ast}^{\tau_{l}}\nu_{l}\in\mathfrak{K}\text{ \quad for }l\geq1\text{,}
\label{Eq_cxdcdsewvxvxwfrvxsf}%
\end{equation}
Then, for any random sequence $\Psi$ in $\mathfrak{Z}$,
\begin{equation}
\Psi_{A_{l}}\overset{\nu_{l}}{\Longrightarrow}\Psi\text{\quad iff\quad}%
\Psi_{A_{l}}\overset{\mu}{\Longrightarrow}\Psi\text{\qquad as }l\rightarrow
\infty\text{.}%
\end{equation}

\end{theorem}

Of course, the most important case will be that of $\nu_{l}=\mu_{A_{l}}$.%

\vspace{0.3cm}%
%

\noindent
\textbf{Convergence to iid limit processes. }A variant of the above assumption
in which we now take $\nu_{l}$ to be $\mu_{A_{l}}$ conditioned on suitable
subsets of $A_{l}$ actually allows us to prove (under the necessary assumption
(\ref{Eq_cnbvmcbvmbmybvmyvbyb}) discussed above) convergence of the local
processes to an \emph{independent} stationary sequence.

\begin{theorem}
[\textbf{Convergence to an iid local process}]\label{T_AsyIntStateProcIndep}%
Let $(X,\mathcal{A},\mu,T)$ be an ergodic probability preserving system,
$(A_{l})$ an asymptotically rare sequence in $\mathcal{A}$, and $(\psi_{A_{l}%
})_{l\geq1}$ a sequence of $\mathfrak{Z}$-valued local observables for the
$A_{l}$ such that
\begin{equation}
\psi_{A_{l}}\overset{\mu_{A_{l}}}{\Longrightarrow}\psi\text{ \qquad as
}l\rightarrow\infty\text{,}%
\end{equation}
for some random element $\psi$ of $\mathfrak{Z}$. Let $\Psi_{A_{l}}$ be the
corresponding local processes. Suppose that $\mathcal{B}_{\mathfrak{Z}}^{\pi
}\subseteq\mathcal{B}_{\mathfrak{Z}}$ is a $\pi$-system generating
$\mathcal{B}_{\mathfrak{Z}}$ with $\mathfrak{Z}\in\mathcal{B}_{\mathfrak{Z}%
}^{\pi}$, while $\Pr[\psi\in\partial F]=0$ for all $F\in\mathcal{B}%
_{\mathfrak{Z}}^{\pi}$.

Set $\nu_{l}:=\mu_{A_{l}}$, and assume further that for every $F\in
\mathcal{B}_{\mathfrak{Z}}^{\pi}$ with $\Pr[\psi\in F]>0$ there are measurable
$\tau_{l,F}:X\rightarrow\mathbb{N}_{0}$, $l\geq1$, such that, letting
$\nu_{l,F}$ denote the normalized restriction $\nu_{l,\{\psi_{A_{l}}\in F\}}$
of $\nu_{l}$ to $\{\psi_{A_{l}}\in F\}$, we have
\begin{equation}
\nu_{l,F}\left(  \tau_{l}<\varphi_{A_{l}}\right)  \longrightarrow
1\quad\text{as }l\rightarrow\infty\text{,} \label{Eq_uiuiuiuiui}%
\end{equation}
while there is some compact subset $\mathfrak{K}_{F}$ of $(\mathfrak{P}%
,d_{\mathfrak{P}})$ such that
\begin{equation}
T_{\ast}^{\tau_{l,F}}(\nu_{l,F})\in\mathfrak{K}_{F}\text{ \quad for }%
l\geq1\text{,} \label{Eq_uiuiuiuiuiiiiiiiiii}%
\end{equation}
Then
\begin{equation}
\widetilde{\Psi}_{A_{l}}\overset{\mu_{A_{l}}}{\Longrightarrow}\Psi^{\ast
}\text{\quad and\quad}\Psi_{A_{l}}\overset{\mu}{\Longrightarrow}\Psi^{\ast
}\text{\qquad as }l\rightarrow\infty\text{,} \label{Eq_TheConclusion}%
\end{equation}
where $\Psi^{\ast}=(\psi^{\ast(j)})_{j\geq0}$ is an iid sequence in
$\mathfrak{Z}$ with $\mathrm{law}(\psi^{\ast(0)})=\mathrm{law}(\psi)$.

More generally, (\ref{Eq_TheConclusion}) follows if every $\varepsilon>0$
there are a sequence $(\nu_{l})_{l\geq1}$ in $\mathfrak{P}$ with
$d_{\mathfrak{P}}(\nu_{l},\mu_{A_{l}})<\varepsilon$ for all $l$, and, for
every $F\in\mathcal{B}_{\mathfrak{Z}}^{\pi}$ with $\Pr[\psi\in F]>0$, a
compact set $\mathfrak{K}_{F}\subseteq\mathfrak{P}$, measures $\nu_{l,F}%
\in\mathfrak{P}$ with $d_{\mathfrak{P}}(\nu_{l,F},\nu_{l,\{\psi_{A_{l}}\in
F\}})\rightarrow0$, and $\tau_{l,F}$ such that (\ref{Eq_uiuiuiuiui}) and
(\ref{Eq_uiuiuiuiuiiiiiiiiii}) hold.
\end{theorem}

\begin{remark}
[\textbf{Using rich generating families of conditioning events}]Note that
there are natural situations in which the condition $\Pr[\psi\in\partial F]=0$
for all $F\in\mathcal{B}_{\mathfrak{Z}}^{\pi}$ need not be checked explicitly.
For instance, if $\mathfrak{Z}=[a,b]$, then $\mathcal{G}:=\{(s,b]:s\in\lbrack
a,b)\}$ is a collection of sets with pairwise disjoint boundaries. Thus, the
requirement $\Pr[\psi\in\partial F]=0$ only rules out countably many
$F\in\mathcal{G}$, and the remaining family $\mathcal{B}_{\mathfrak{Z}}^{\pi
}:=\{F\in\mathcal{G}:\Pr[\psi\in\partial F]=0\}$ still is a generating $\pi$-system.

Therefore, if we can check the other conditions (\ref{Eq_uiuiuiuiui}) and
(\ref{Eq_uiuiuiuiuiiiiiiiiii}) for all $F\in\mathcal{G}$, then there is
automatically a suitable collection $\mathcal{B}_{\mathfrak{Z}}^{\pi}$.
\end{remark}

%

\vspace{0.3cm}%
%

\noindent
\textbf{Robustness of the asymptotic behaviour. }It will also be useful to
know that asymptotic local processes do not change if the sets $A_{l}$ are
replaced by sets $A_{l}^{\prime}$ which are asymptotically equivalent (mod
$\mu$) in that $\mu(A_{l}\bigtriangleup A_{l}^{\prime})=o(\mu(A_{l}))$, and if
the local observables $\psi_{A_{l}}$ are replaced by $\psi_{A_{l}^{\prime}%
}^{\prime}$ which are close in measure.

\begin{theorem}
[\textbf{Robustness of asymptotic local processes}]\label{Thm_RobustLocProc}%
Let $(X,\mathcal{A},\mu,T)$ be an ergodic probability preserving system,
$(A_{l})$ and $(A_{l}^{\prime})$ two asymptotically rare sequences in
$\mathcal{A}$, and $(\psi_{A_{l}})$, $(\psi_{A_{l}^{\prime}}^{\prime})$ two
sequences of $\mathfrak{Z}$-valued local observables for the sets $A_{l}$ and
$A_{l}^{\prime}$, respectively. Assume that
\begin{equation}
\mu(A_{l}\bigtriangleup A_{l}^{\prime})=o(\mu(A_{l}))\text{ \qquad as
}l\rightarrow\infty\text{,} \label{Eq_hbhbhbhbhb}%
\end{equation}
and
\begin{equation}
d_{\mathfrak{Z}}(\psi_{A_{l}},\psi_{A_{l}^{\prime}}^{\prime})\overset
{\mu_{A_{l}\cap A_{l}^{\prime}}}{\longrightarrow}0\text{ \qquad as
}l\rightarrow\infty\text{.} \label{Eq_LocalObsAsySameInMeasure}%
\end{equation}
Then, for any random sequence $\widetilde{\Psi}$ \ in $\mathfrak{Z}$,%
\begin{equation}
\widetilde{\Psi}_{A_{l}}\overset{\mu_{A_{l}}}{\Longrightarrow}\widetilde{\Psi
}\text{\qquad iff \qquad}\widetilde{\Psi}_{A_{l}^{\prime}}^{\prime}%
\overset{\mu_{A_{l}^{\prime}}}{\Longrightarrow}\widetilde{\Psi}\text{,}%
\end{equation}
where $\widetilde{\Psi}_{A_{l}}$ and $\widetilde{\Psi}_{A_{l}^{\prime}%
}^{\prime}$ are the local processes given by $\psi_{A_{l}}$ and $\psi
_{A_{l}^{\prime}}^{\prime}$, respectively.
\end{theorem}

For basic situations in which this applies consider the following.

\begin{example}
[\textbf{Normalizing interval charts}]Let $X$ be an interval containing the
sets $A_{l}=[a_{l},b_{l}]$, $A_{l}^{\prime}=[a_{l}^{\prime},b_{l}^{\prime}]$,
and let $\psi_{A_{l}}$ and $\psi_{A_{l}^{\prime}}^{\prime}$ be the
corresponding normalizing interval charts. It is immediate that $\sup
_{A_{l}\cap A_{l}^{\prime}}d_{[0,1]}(\psi_{A_{l}},\psi_{A_{l}^{\prime}%
}^{\prime})\rightarrow0$ provided that $\lambda(A_{l}\bigtriangleup
A_{l}^{\prime})=o(\lambda(A_{l}))$ as $l\rightarrow\infty$. In this case,
assumptions (\ref{Eq_hbhbhbhbhb}) and (\ref{Eq_LocalObsAsySameInMeasure}) of
the theorem are satisfied if $c\lambda\leq\mu\leq c^{-1}\lambda$ for some
constant $c>0$.
\end{example}

\begin{example}
[\textbf{Normalized return times}]\label{Ex_RobustRetTimes}Let $(X,\mathcal{A}%
,\mu,T)$ be ergodic and probability preserving, $(A_{l})$, $(A_{l}^{\prime})$
two asymptotically rare sequences, and consider the $[0,\infty]$-valued local
observables $\psi_{A_{l}}:=\mu(A_{l})\varphi_{A_{l}}\mid_{A_{l}}$,
$\psi_{A_{l}^{\prime}}^{\prime}:=\mu(A_{l}^{\prime})\varphi_{A_{l}^{\prime}%
}\mid_{A_{l}^{\prime}}$. Then,
\begin{equation}
\mu(A_{l}\bigtriangleup A_{l}^{\prime})=o(\mu(A_{l}))\text{ \qquad
implies\qquad}d_{[0,\infty]}(\psi_{A_{l}},\psi_{A_{l}^{\prime}}^{\prime
})\overset{\mu_{A_{l}\cap A_{l}^{\prime}}}{\longrightarrow}0\text{.}
\label{Eq_buewujksk}%
\end{equation}
In particular, Theorem 2.2 of \cite{Z11} is a special case of Theorem
\ref{Thm_RobustLocProc} above.

(To check (\ref{Eq_buewujksk}), assume w.l.o.g. that $A_{l}^{\prime}\subseteq
A_{l}$, take $\eta>0$, and note that due to $d_{[0,\infty]}(s,t)\leq\mid
s-t\mid$ it suffices to control the measure of $\{\mid\mu(A_{l})\varphi
_{A_{l}}-\mu(A_{l}^{\prime})\varphi_{A_{l}^{\prime}}\mid\geq\eta
\}\subseteq\{\mu(A_{l}^{\prime})\mid\varphi_{A_{l}}-\varphi_{A_{l}^{\prime}%
}\mid\geq\eta/2\}\cup\{\mu(A_{l}\setminus A_{l}^{\prime})\,\varphi_{A_{l}}%
\geq\eta/2\}$ which is easy since $(A_{l}\cup A_{l}^{\prime})\cap
\{\varphi_{A_{l}}\neq\varphi_{A_{l}^{\prime}}\}=T_{A_{l}\cup A_{l}^{\prime}%
}^{-1}(A_{l}\bigtriangleup A_{l}^{\prime})$ and $\int\mu(A_{l}^{\prime
})\varphi_{A_{l}^{\prime}}d\mu_{A_{l}^{\prime}}=1$ for all $l$.)
\end{example}

%

\vspace{0.3cm}%

\section{Joint limit processes}

Again, the space for local observables will be a compact metric space
$(\mathfrak{Z},d_{\mathfrak{Z}})$. Given an asymptotically rare sequence
$(A_{l})$ for $(X,\mathcal{A},\mu,T)$ and local observables $\psi_{A_{l}}$ we
now consider the joint distribution of $\mu(A_{l})\Phi_{A_{l}}$ and
$\widetilde{\Psi}_{A_{l}}$ under $\mu_{A_{l}}$, and that of $\mu(A_{l}%
)\Phi_{A_{l}}$ and $\Psi_{A_{l}}$ under $\mu$ (or some other fixed probability
$\nu\in\mathfrak{P}$). For the second variant we find, as expected:

\begin{proposition}
[\textbf{Strong distributional convergence of }$(\mu(A_{l})\Phi_{A_{l}}%
,\Psi_{A_{l}})$]\label{P_SDCJoint}Let $(X,\mathcal{A},\mu,T)$ be an ergodic
probability preserving system, $(A_{l})$ an asymptotically rare sequence in
$\mathcal{A}$, with $(\Psi_{A_{l}})_{l\geq1}$ a sequence of $\mathfrak{Z}%
$-valued local processes for the $A_{l}$. Let $(\Phi,\Psi)$ be any random
sequence in $[0,\infty]\times\mathfrak{Z}$. Then $(\mu(A_{l})\Phi_{A_{l}}%
,\Psi_{A_{l}})\overset{\nu}{\Longrightarrow}\,(\Psi,\Phi)$ for \emph{some}
$\nu\in\mathfrak{P}$ iff $(\mu(A_{l})\Phi_{A_{l}}\Psi_{A_{l}})\overset{\nu
}{\Longrightarrow}\,(\Psi,\Phi)$ for \emph{all} $\nu\in\mathfrak{P}$.
\end{proposition}

The main result of this section, Theorem \ref{T_SpatiotemporalIIDLimits}
below, gives sufficient conditions for convergence to an independent pair of
iid sequences. Before stating it, we record that, in certain situations, this
takes place under the measure $\mu$ iff it takes place under the measures
$\mu_{A_{l}}$. Recall (\ref{Eq_AWarningLocObs}), which shows that the latter
statement can only be correct under some extra condition. We will use the same
assumption, (\ref{Eq_mnhjnggnfjfggnjf1}) and (\ref{Eq_mnhjnggnfjfggnjf2})
below, which already appeared in Theorem \ref{T_AsyIntStateProc}.

\begin{theorem}
[\textbf{Independent joint limit processes - }$\mu_{A_{l}}$\textbf{\ versus
}$\mu$]\label{T_JointLimitProcMuVsMuA}Suppose $(X,\mathcal{A},\mu,T)$ is an
ergodic probability preserving system, $(A_{l})$ an asymptotically rare
sequence in $\mathcal{A}$, and $(\psi_{A_{l}})_{l\geq1}$ a sequence of
$\mathfrak{Z}$-valued local observables for the $A_{l}$ with corresponding
local processes $\Psi_{A_{l}}$. Assume there are measurable $\tau
_{l}:X\rightarrow\mathbb{N}_{0}$ s.t.
\begin{equation}
\mu_{A_{l}}\left(  \tau_{l}<\varphi_{A_{l}}\right)  \longrightarrow
1\quad\text{as }l\rightarrow\infty\text{,} \label{Eq_mnhjnggnfjfggnjf1}%
\end{equation}
while there is some compact subset $\mathfrak{K}$ of $(\mathfrak{P}%
,d_{\mathfrak{P}})$ such that
\begin{equation}
T_{\ast}^{\tau_{l}}\mu_{A_{l}}\in\mathfrak{K}\text{ \quad for }l\geq1\text{,}
\label{Eq_mnhjnggnfjfggnjf2}%
\end{equation}

Let $(\Phi_{\mathrm{Exp}},\Psi^{\ast})$ be an independent pair of iid
sequences. Then
\begin{equation}
(\mu(A_{l})\Phi_{A_{l}},\Psi_{A_{l}})\overset{\mu}{\Longrightarrow}%
(\,\Phi_{\mathrm{Exp}},\Psi^{\ast})\quad\text{as }l\rightarrow\infty\text{,}%
\end{equation}
iff
\begin{equation}
(\mu(A_{l})\Phi_{A_{l}},\widetilde{\Psi}_{A_{l}})\overset{\mu_{A_{l}}%
}{\Longrightarrow}\,(\Phi_{\mathrm{Exp}},\Psi^{\ast})\quad\text{as
}l\rightarrow\infty\text{.} \label{Eq_jomiasanmimradldo}%
\end{equation}

\end{theorem}%

\vspace{0.3cm}%

We can then formulate our abstract spatiotemporal Poisson limit theorem.

\begin{theorem}
[\textbf{Joint iid limit processes}]\label{T_SpatiotemporalIIDLimits}Let
$(X,\mathcal{A},\mu,T)$ be an ergodic probability preserving system, $(A_{l})$
an asymptotically rare sequence in $\mathcal{A}$, and $(\psi_{A_{l}})_{l\geq
1}$ a sequence of $\mathfrak{Z}$-valued local observables for the $A_{l}$ such
that
\begin{equation}
\psi_{A_{l}}\overset{\mu_{A_{l}}}{\Longrightarrow}\psi\text{ \qquad as
}l\rightarrow\infty\text{,}%
\end{equation}
for some random element $\psi$ of $\mathfrak{Z}$. Let $\Psi_{A_{l}}$ be the
corresponding local processes. Moreover, assume that

\textbf{(A)} for every $s\in\lbrack0\,,\infty)$ there are a compact subset
$\mathfrak{K}_{s}$ of $(\mathfrak{P},d_{\mathfrak{P}})$, measures $\nu
_{l,s}\in\mathfrak{P}$ with $d_{\mathfrak{P}}(\nu_{l,s},\mu_{A_{l}\cap
\{\mu(A_{l})\varphi_{A_{l}}>s\}})\rightarrow0$, and measurable $\tau
_{l,s}:X\rightarrow\mathbb{N}_{0}$ s.t.
\begin{equation}
\nu_{l,s}\left(  \tau_{l,s}<\varphi_{A_{l}}\right)  \longrightarrow
1\quad\text{as }l\rightarrow\infty, \label{Eq_xdrxrxxdxrxdrxd}%
\end{equation}
while $T_{\ast}^{\tau_{l,s}}\nu_{l,s}\in\mathfrak{K}_{s}$ for $l\geq1$, and

\textbf{(B)} there is some $\pi$-system $\mathcal{B}_{\mathfrak{Z}}^{\pi}$
generating $\mathcal{B}_{\mathfrak{Z}}$ with $\mathfrak{Z}\in\mathcal{B}%
_{\mathfrak{Z}}^{\pi}$, while $\Pr[\psi\in\partial F]=0$ for all
$F\in\mathcal{B}_{\mathfrak{Z}}^{\pi}$; in addition, for every $F\in
\mathcal{B}_{\mathfrak{Z}}^{\pi}$ with $\Pr[\psi\in F]>0$ there are a compact
subset $\mathfrak{K}_{F}$ of $(\mathfrak{P},d_{\mathfrak{P}})$, measures
$\nu_{l,F}\in\mathfrak{P}$ with $d_{\mathfrak{P}}(\nu_{l,F},\mu_{A_{l}%
\cap\{\psi_{A_{l}}\in F\}})\rightarrow0$, and measurable $\tau_{l,F}%
:X\rightarrow\mathbb{N}_{0}$ such that
\begin{equation}
\mu(A_{l})\,\tau_{l,F}\overset{\nu_{l,F}}{\longrightarrow}0\quad\text{as
}l\rightarrow\infty, \label{Eq_vzgvzgvzgvzgvzgv1}%
\end{equation}
and
\begin{equation}
\nu_{l,F}\left(  \tau_{l,F}<\varphi_{A_{l}}\right)  \longrightarrow
1\quad\text{as }l\rightarrow\infty\text{,} \label{Eq_owowowowowoowwowowow}%
\end{equation}
while $T_{\ast}^{\tau_{l,F}}\nu_{l,F}\in\mathfrak{K}_{F}$ \quad for $l\geq
1$.\newline\newline Then,%
\begin{equation}
(\mu(A_{l})\Phi_{A_{l}},\Psi_{A_{l}})\overset{\mu}{\Longrightarrow}%
(\,\Phi_{\mathrm{Exp}},\Psi^{\ast})\quad\text{as }l\rightarrow\infty\text{,}%
\end{equation}
and
\begin{equation}
(\mu(A_{l})\Phi_{A_{l}},\widetilde{\Psi}_{A_{l}})\overset{\mu_{A_{l}}%
}{\Longrightarrow}\,(\Phi_{\mathrm{Exp}},\Psi^{\ast})\quad\text{as
}l\rightarrow\infty\text{,} \label{Eq_qdqdqdqdqd}%
\end{equation}
where $(\,\Phi_{\mathrm{Exp}},\Psi^{\ast})$ is an independent pair of iid processes.
\end{theorem}

%

\vspace{0.3cm}%

\begin{remark}
[\textbf{Convergence of associated spatiotemporal point processes}%
]\label{R_CgeSpatoiTempPoint1}Generalizing Remark \ref{R_CgePoint1} we can
interpret the above as a result about convergence of associated spatiotemporal
point processes $\mathcal{N}_{A,\psi_{A}}:X\rightarrow\mathcal{M}%
_{p}([0,\infty)\times\mathfrak{Z})$ (the collection of Radon point measures
$\mathrm{n}:\mathcal{B}_{[0,\infty)\times\mathfrak{Z}}\rightarrow
\{0,1,2,\ldots,\infty\}$ with the topology of vague convergence,
\cite{Resnick}) defined by%
\begin{equation}
\mathcal{N}_{A,\psi_{A}}:=%
{\textstyle\sum_{j\geq1}}
\delta_{\{(\mu(A)\sum_{i=0}^{j-1}\varphi_{A}\circ T_{A}^{i},\psi_{A}\circ
T_{A}^{j})\}}\text{.}%
\end{equation}
Parallel to our earlier remarks, for $(\,\Phi_{\mathrm{Exp}},\Psi^{\ast})$ an
independent pair of iid processes and arbitrary probabilities $\nu_{l}$ on
$(X,\mathcal{A})$,
\begin{gather}
(\mu(A_{l})\Phi_{A_{l}},\Psi_{A_{l}})\overset{\nu_{l}}{\Longrightarrow}%
(\,\Phi_{\mathrm{Exp}},\Psi^{\ast})\quad\text{implies}\quad\\
\mathcal{N}_{A_{l},\psi_{A_{l}}}\overset{\nu_{l}}{\Longrightarrow}%
\mathrm{PRM}(\lambda_{\lbrack0,\infty)}^{1}\otimes\mathrm{law}(\psi))\text{ in
}\mathcal{M}_{p}([0,\infty)\times\mathfrak{Z})\text{,}\nonumber
\end{gather}
and $\Psi_{A_{l}}$ may be replaced by $\widetilde{\Psi}_{A_{l}}$ if for each
$l\geq1$, $\nu_{l}$ is supported on $A_{l}$. As before, this merely requires
an application of the \textquotedblleft extended continuous mapping
theorem\textquotedblright\ (Theorem 2.7 in \cite{Bi}), because the map
$\Theta^{\times}:[0,\infty]^{\mathbb{N}_{0}}\times\mathfrak{Z}^{\mathbb{N}%
}\rightarrow\mathcal{M}_{p}([0,\infty)\times\mathfrak{Z})$ with $\Theta
^{\times}((\varphi^{(i)})_{i\geq0},(\psi^{(i)})_{i\geq1}):=%
{\textstyle\sum_{j\geq1}}
\delta_{\{(\sum_{i=0}^{j-1}\varphi^{(i)},\psi^{(j)})\}}$ for $(\varphi
^{(i)})_{i\geq0}\in\mathcal{W}$, and $\Theta^{\times}((\varphi^{(i)})_{i\geq
0},(\psi^{(i)})_{i\geq1}):=0$ otherwise, is continuous on $\mathcal{W}%
\times\mathfrak{Z}^{\mathbb{N}}$.
\end{remark}

\begin{remark}
[\textbf{Robustness of joint limit processes}]\label{R_RobustJoints}The
conclusion in (\ref{Eq_qdqdqdqdqd}) is a statement about the $(\mu(A_{l}%
)\Phi_{A_{l}},\widetilde{\Psi}_{A_{l}})$ which can be regarded as local
processes taking values in $[0,\infty]\times\mathfrak{Z}$. Recalling Example
\ref{Ex_RobustRetTimes} we can therefore use Theorem \ref{Thm_RobustLocProc}
to see that (\ref{Eq_qdqdqdqdqd}) is equivalent to
\[
(\mu(A_{l}^{\prime})\Phi_{A_{l}^{\prime}},\widetilde{\Psi^{\prime}}%
_{A_{l}^{\prime}})\overset{\mu_{A_{l}^{\prime}}}{\Longrightarrow}%
\,(\Phi_{\mathrm{Exp}},\Psi^{\ast})\quad\text{as }l\rightarrow\infty\text{,}%
\]
whenever $\mu(A_{l}\bigtriangleup A_{l}^{\prime})=o(\mu(A_{l}))$ and
$d_{\mathfrak{Z}}(\psi_{A_{l}},\psi_{A_{l}^{\prime}}^{\prime})\overset
{\mu_{A_{l}\cap A_{l}^{\prime}}}{\longrightarrow}0$. This sometimes allows us
to replace the original sequence by one for which the conditions of the
present theorem can be verified more easily.
\end{remark}

%

\vspace{0.3cm}%

\section{Mean ergodic theory and distributions under varying measures}

The present section discusses the abstract core of our approach. (It is based
on ideas which arose in the study of probabilistic properties of infinite
measure preserving systems, see \cite{ZcompactRegeneration}, \cite{PSZ},
\cite{RZ}.) Throughout, $(\mathfrak{E},d_{\mathfrak{E}})$ is a compact metric
space with Borel $\sigma$-algebra $\mathcal{B}_{\mathfrak{E}}$. For a
Lipschitz function $\varkappa:\mathfrak{E}\rightarrow\mathbb{R}$ we let
$\mathrm{Lip}(\varkappa)$ denote its (optimal) Lipschitz constant. %

\vspace{0.3cm}%
%

\noindent
\textbf{More on distributional convergence.} Recall that weak convergence of
probabilities on $(\mathfrak{E},\mathcal{B}_{\mathfrak{E}})$ is metrisable:
There are various metrics $D_{\mathfrak{E}}$ on $\mathfrak{M}(\mathfrak{E})$
such that $D_{\mathfrak{E}}(Q_{l},Q)\rightarrow0$ iff $Q_{l}\Longrightarrow
Q$. For instance, it is well known that $\mathcal{C}(\mathfrak{E})$ is
separable, and by a standard argument for w$^{\ast}$-topologies, every dense
sequence $(\vartheta_{j})_{j\geq1}$ in $\mathcal{C}(\mathfrak{E})$ allows us
to define a suitable metric with $D_{\mathfrak{E}}\leq1$ by setting
\begin{equation}
D_{\mathfrak{E}}(Q,Q^{\prime}):=\sum_{j\geq1}2^{-(j+1)}\left\vert \int\chi
_{j}\,dQ-\int\chi_{j}\,dQ^{\prime}\right\vert \text{, \quad}Q,Q^{\prime}%
\in\mathfrak{M}(\mathfrak{E})\text{,}\label{Eq_DefMeasDist}%
\end{equation}
where $\chi_{j}:=c_{j}^{-1}\vartheta_{j}$ for constants $c_{j}\geq
\sup\left\vert \vartheta_{j}\right\vert $ (so that $\left\vert \chi
_{j}\right\vert \leq1$). By the Stone-Weierstrass theorem, Lipschitz functions
are dense in $\mathcal{C}(\mathfrak{E})$. As a consequence, we can
\emph{define} $D_{\mathfrak{E}}$ \emph{as in (\ref{Eq_DefMeasDist})}
\emph{using a particular dense sequence }$(\vartheta_{j})$\emph{, henceforth
fixed, of Lipschitz functions} $\vartheta_{j}:\mathfrak{E}\rightarrow
\mathbb{R}$, and $c_{j}:=\max(\sup\left\vert \vartheta_{j}\right\vert
,\mathrm{Lip}(\vartheta_{j}))$, which ensures that $\mathrm{Lip}(\varkappa
_{j})\leq1$ for all $j\geq1$.

Observe that for any convex combinations $Q=\theta Q_{\triangle}%
+(1-\theta)Q_{\triangledown}$ and $Q^{\prime}=\theta Q_{\triangle}^{\prime
}+(1-\theta)Q_{\triangledown}^{\prime}$ in $\mathfrak{M}(\mathfrak{E})$ with
the same $\theta\in\lbrack0,1]$, we have
\begin{align}
D_{\mathfrak{E}}(Q,Q^{\prime})  &  \leq\theta D_{\mathfrak{E}}(Q_{\triangle
},Q_{\triangle}^{\prime})+(1-\theta)D_{\mathfrak{E}}(Q_{\triangledown
},Q_{\triangledown}^{\prime})\nonumber\\
&  \leq D_{\mathfrak{E}}(Q_{\triangle},Q_{\triangle}^{\prime})+(1-\theta
)\text{.} \label{Eq_Tree2}%
\end{align}
It is straightforward that for $\nu,\widetilde{\nu}\in\mathfrak{P}$, and Borel
measurable $R:X\rightarrow\mathfrak{E}$,
\begin{equation}
D_{\mathfrak{E}}(\mathrm{law}_{\nu}(R),\mathrm{law}_{\widetilde{\nu}}(R))\leq
d_{\mathfrak{P}}(\nu,\widetilde{\nu})\text{,}
\label{Eq_ererrewrerwrertrerwerq}%
\end{equation}
so that for any sequences $(\nu_{l})$ and $(\widetilde{\nu}_{l})$ in
$\mathfrak{P}$, and Borel measurable $R_{l}:X\rightarrow\mathfrak{E}$,
\begin{equation}
d_{\mathfrak{P}}(\nu_{l},\widetilde{\nu}_{l})\rightarrow0\text{ \quad implies
\quad}D_{\mathfrak{E}}(\mathrm{law}_{\nu_{l}}(R_{l}),\mathrm{law}%
_{\widetilde{\nu}_{l}}(R_{l}))\rightarrow0\text{.}
\label{Eq_vcbycxvcygvchatecv356463456}%
\end{equation}
Also, for further measurable $R_{l}^{\prime}:X\rightarrow\mathfrak{E}$,
\begin{equation}
d_{\mathfrak{E}}(R_{l},R_{l}^{\prime})\overset{\nu_{l}}{\longrightarrow
}0\text{ \quad implies \quad}D_{\mathfrak{E}}(\mathrm{law}_{\nu_{l}}%
(R_{l}),\mathrm{law}_{\nu_{l}}(R_{l}^{\prime}))\rightarrow0\text{,}
\label{Eq_Tree1}%
\end{equation}
because $D_{\mathfrak{E}}(\mathrm{law}_{\nu_{l}}(R_{l}),\mathrm{law}_{\nu_{l}%
}(R_{l}^{\prime}))\leq%
{\textstyle\sum\nolimits_{j=1}^{J}}
\mathrm{Lip}(\chi_{j})\int d_{\mathfrak{E}}(R_{l},R_{l}^{\prime})\,d\nu
_{l}+2^{-J}$ for each $J\geq1$, and $\int d_{\mathfrak{E}}(R_{l},R_{l}%
^{\prime})\,d\nu_{l}\rightarrow0$ as $\mathfrak{E}$ is bounded. Finally, note
that
\begin{equation}
R=R^{\prime}\text{ on }A\in\mathcal{A}\text{ \quad implies \quad
}D_{\mathfrak{E}}(\mathrm{law}_{\nu}(R),\mathrm{law}_{\nu}(R^{\prime}))\leq
\nu(A^{c})\text{.}%
\end{equation}
%

\vspace{0.3cm}%
%

\noindent
\textbf{Strong distributional convergence.} If $(X,\mathcal{A},\mu,T)$ is an
ergodic probability preserving system, distributional limit theorems for
dynamically defined quantities are often stated in terms of the distinguished
measure $\nu:=\mu$. Since the latter is not the only potentially relevant
initial distribution for the process, it is both interesting and useful to
observe that in many cases such a limit theorem automatically carries over to
all probability measures $\nu$ absolutely continuous with respect to $\mu$.

For measurable maps $R_{l}$, $l\geq1$, of a probability space $(X,\mathcal{A}%
,\mu)$ into $(\mathfrak{E},\mathcal{B}_{\mathfrak{E}})$, \emph{strong
distributional convergence} to a random element $R$ of $\mathfrak{E}$,
written
\begin{equation}
R_{l}\overset{\mathcal{L}(\mu)}{\Longrightarrow}R\text{ \quad as }%
l\rightarrow\infty\text{,}%
\end{equation}
means that $R_{l}\overset{\nu}{\Longrightarrow}R$ for all probability measures
$\nu\ll\mu$. (This is equivalent to $R_{l}\overset{\mu}{\Longrightarrow}R$
(\emph{mixing}), meaning that $R_{l}\overset{\mu_{E}}{\Longrightarrow}R$ for
every fixed $E\in\mathcal{A}$ with $\mu(E)>0$. The latter concept dates back
to \cite{Renyi}, see also \cite{Ea}.)%
\newline

A property often responsible for this sort of behaviour is that the sequence
$(R_{l})$ be $T$-invariant in the long run. Let $T$ be a measure-preserving
map on the probability space $(X,\mathcal{A},\mu)$. For Borel measurable maps
$R_{l}:X\rightarrow\mathfrak{E}$, we call the sequence $(R_{l})$
\emph{asymptotically }$T$\emph{-invariant in (}$\mu$\emph{-)measure} in case
\begin{equation}
d_{\mathfrak{E}}(R_{l}\circ T,R_{l})\overset{\mu}{\longrightarrow}0\text{
\quad as }l\rightarrow\infty\label{Eq_AsyTInvar}%
\end{equation}
(with $\overset{\mu}{\longrightarrow}$ indicating convergence in measure,
recall (\ref{Eq_NotationProbaCgeSeqMeas})). Note that this is equivalent to
convergence in distribution, $d_{\mathfrak{E}}(R_{l}\circ T,R_{l})\overset
{\mu}{\Longrightarrow}0$, because the limit is constant. The important role of
this concept becomes clear through

\begin{theorem}
[\textbf{Strong distributional convergence of asymptotically invariant
sequences}]\label{T_MyOldStrongDistrCgeThm}Let $T$ be an ergodic
measure-preserving map on the probability space $(X,\mathcal{A},\mu)$. Suppose
that the sequence $(R_{l})$ of Borel measurable maps $R_{l}:X\rightarrow
\mathfrak{E}$ into the compact metric space $(\mathfrak{E},d_{\mathfrak{E}})$
is asymptotically $T$-invariant in measure. Then
\begin{equation}
D_{\mathfrak{E}}(\mathrm{law}_{\nu}(R_{l}),\mathrm{law}_{\overline{\nu}}%
(R_{l}))\longrightarrow0\quad\text{as }l\rightarrow\infty\text{\quad for }%
\nu,\overline{\nu}\in\mathfrak{P}\text{.} \label{Eq_bvjhbvjabvahkvaevfuerfkea}%
\end{equation}
Hence, for $R$ a random element of $\mathfrak{E}$, and any $\nu,\overline{\nu
}\in\mathfrak{P}$,
\begin{equation}
R_{l}\overset{\overline{\nu}}{\Longrightarrow}R\text{ \qquad implies \qquad
}R_{l}\overset{\nu}{\Longrightarrow}R\text{ }\quad\text{as }l\rightarrow
\infty\text{.} \label{Eq_cyxscvdsxfcfxdxcfsa}%
\end{equation}

\end{theorem}%

\vspace{0.3cm}%

\begin{proof}
It is clear that (\ref{Eq_bvjhbvjabvahkvaevfuerfkea}) entails
(\ref{Eq_cyxscvdsxfcfxdxcfsa}). The implication (\ref{Eq_cyxscvdsxfcfxdxcfsa})
is the content of Theorem 1 of \cite{Z7}. A stronger version of assertion
(\ref{Eq_bvjhbvjabvahkvaevfuerfkea}) is contained in Theorem
\ref{T_HTSforVaryingMeasures} below, whose proof does not use the present theorem.

Alternatively, it is not hard to check directly that
(\ref{Eq_cyxscvdsxfcfxdxcfsa}) implies (\ref{Eq_bvjhbvjabvahkvaevfuerfkea}):
Suppose that (\ref{Eq_bvjhbvjabvahkvaevfuerfkea}) fails, meaning that there
are $\nu,\overline{\nu}\in\mathfrak{P}$ with $\delta>0$ and $l_{j}%
\nearrow\infty$ such that
\begin{equation}
D_{\mathfrak{E}}(\mathrm{law}_{\nu}(R_{l_{j}}),\mathrm{law}_{\overline{\nu}%
}(R_{l_{j}}))\geq\delta\text{ \quad for }j\geq1\text{.}
\label{Eq_uioizpouipiupztiup}%
\end{equation}
By Alaoglu's theorem, the metric space $(\mathfrak{M}(\mathfrak{E}%
),D_{\mathfrak{E}})$ is compact, which allows us to select a further
subsequence $l_{i}^{\prime}=l_{j_{i}}\nearrow\infty$ of indices such that
$\mathrm{law}_{\overline{\nu}}(R_{l_{i}^{\prime}})\Longrightarrow Q$ for some
$Q\in\mathfrak{M}(\mathfrak{E})$. But then (\ref{Eq_cyxscvdsxfcfxdxcfsa})
shows that $\mathrm{law}_{\nu}(R_{l_{i}^{\prime}})\Longrightarrow Q$ as well,
which contradicts (\ref{Eq_uioizpouipiupztiup}).
\end{proof}

\bigskip%
\vspace{0.3cm}%
%

\noindent
\textbf{The transfer operator and mean ergodic theory.} Since we shall improve
on the above result, we review the main ingredient of its proof. Recall the
\emph{transfer operator} $\widehat{T}:L_{1}(\mu)\rightarrow L_{1}(\mu)$ of $T$
on $(X,\mathcal{A},\mu)$ which describes the evolution of probability
densities under $T$. That is, if $\nu$ has density $u$ w.r.t. $\mu$,
$u=d\nu/d\mu$, then $\widehat{T}u:=d(\nu\circ T^{-1})/d\mu$. Equivalently,
$\int(g\circ T)\cdot u\,d\mu=\int g\cdot\widehat{T}u\,d\mu$ for all $u\in
L_{1}(\mu)$ and $g\in L_{\infty}(\mu)$, i.e. $\widehat{T}$ is dual to
$g\longmapsto g\circ T$. We let $\mathcal{D}(\mu)$ denote the set of
probability densities w.r.t. $\mu$.%

\vspace{0.3cm}%

The following classical companion of the mean ergodic theorem is essentially
due to Yosida \cite{Yos} (see also \cite{Kr}, Theorem 2.1.3 or \cite{Z7},
Theorem 2). Statements (\ref{Eq_T_Yosida}) and (\ref{Eq_T_Yosida2}) below are
equivalent since we can identify $(\mathfrak{P},d_{\mathfrak{P}})$ with
$(\mathcal{D}(\mu),\left\Vert \centerdot\right\Vert _{L_{1}(\mu)})$ via
$\nu\mapsto d\nu/d\mu$, where $\left\Vert d\nu/d\mu-d\overline{\nu}%
/d\mu\right\Vert _{L_{1}(\mu)}=d_{\mathfrak{P}}(\nu,\overline{\nu})$ for
$\nu,\overline{\nu}\in\mathfrak{P}$.

\begin{theorem}
[\textbf{Characterization of ergodicity}]\label{T_Yosida}\textit{Let }%
$T$\textit{\ be a measure-preserving map on a probability space }%
$(X,\mathcal{A},\mu)$\textit{. Then }$T$ \textit{is ergodic if and only if }%
\begin{equation}
\left\Vert \frac{1}{n}\sum_{k=0}^{n-1}\widehat{T}^{k}(u-\overline
{u})\right\Vert _{L_{1}(\mu)}\longrightarrow0\text{ \quad as }l\rightarrow
\infty\text{\quad\textit{for all} }u,\overline{u}\in\mathcal{D}(\mu
)\text{,}\label{Eq_T_Yosida}%
\end{equation}
which is equivalent to
\begin{equation}
d_{\mathfrak{P}}\left(  \frac{1}{n}\sum_{k=0}^{n-1}T_{\ast}^{k}\nu,\frac{1}%
{n}\sum_{k=0}^{n-1}T_{\ast}^{k}\overline{\nu}\right)  \longrightarrow0\text{
\quad as }l\rightarrow\infty\text{\quad\textit{for all} }\nu,\overline{\nu}%
\in\mathfrak{P}\text{.}\label{Eq_T_Yosida2}%
\end{equation}

\end{theorem}%

\vspace{0.2cm}%
%

\noindent
\textbf{Uniform distributional convergence.} At the heart of the present paper
is a uniform version of the principle of strong distributional convergence for
asymptotically invariant sequences of observables quoted above. We capture the
key point in the following result.

\begin{theorem}
[\textbf{Uniform distributional convergence of asymptotically invariant
sequences}]\label{T_HTSforVaryingMeasures}Let $(X,\mathcal{A},\mu,T)$ be an
ergodic probability preserving system and $(R_{l})_{l\geq1}$ a sequence of
Borel measurable maps $R_{l}:X\rightarrow\mathfrak{E}$, asymptotically
$T$-invariant in measure, into a compact metric space $(\mathfrak{E}%
,d_{\mathfrak{E}})$. Let $\mathfrak{K}$ be a compact set in $(\mathfrak{P}%
,d_{\mathfrak{P}})$. Then,
\begin{equation}
D_{\mathfrak{E}}(\mathrm{law}_{\nu}(R_{l}),\mathrm{law}_{\overline{\nu}}%
(R_{l}))\longrightarrow0\text{ }\quad%
\begin{array}
[c]{c}%
\text{as }l\rightarrow\infty\text{,}\\
\text{uniformly in }\nu,\overline{\nu}\in\mathfrak{K}\text{.}%
\end{array}
\label{Eq_UniformCgeOnCptSetNew}%
\end{equation}
Hence, for $R$ a random element of $\mathfrak{E}$, and any two sequences
$(\nu_{l})$, $(\overline{\nu}_{l})$ in $\mathfrak{K}$,
\begin{equation}
R_{l}\overset{\overline{\nu}_{l}}{\Longrightarrow}R\text{ \qquad implies
\qquad}R_{l}\overset{\nu_{l}}{\Longrightarrow}R\quad\text{as }l\rightarrow
\infty\text{.} \label{Eq_UniformCgeOnCptSetU2}%
\end{equation}

\end{theorem}%

\vspace{0.3cm}%

The key to this refinement of Theorem \ref{T_MyOldStrongDistrCgeThm} is the
following basic principle.

\begin{remark}
[\textbf{Uniform convergence by equicontinuity}]\label{R_UnifCgeByEquiCty}Let
$(\mathfrak{P},d_{\mathfrak{P}})$ be any metric space and $\gamma
_{M}:\mathfrak{P}\rightarrow\mathfrak{P}$, $M\geq1$, a sequence of maps which
converges pointwise to the continuous map $\gamma:\mathfrak{P}\rightarrow
\mathfrak{P}$. If $(\gamma_{M})_{M\geq1}$ is equicontinuous, then
\[
\gamma_{M}\longrightarrow\gamma\text{ \qquad uniformly on }\mathfrak{K}\text{
\quad as }M\rightarrow\infty
\]
whenever $\mathfrak{K}$ is a compact subset of $\mathfrak{P}$. (Indeed, for
every $\varepsilon>0$ there is some $\delta>0$ such that $d_{\mathfrak{P}}%
(\nu,\widetilde{\nu})<\delta$ implies $d_{\mathfrak{P}}(\gamma_{M}(\nu
),\gamma_{M}(\widetilde{\nu}))<\varepsilon$ for all $M\geq1$. But the compact
set $\mathfrak{K}$ contains a finite $\delta$-dense subset, and on the latter
$\gamma_{M}\longrightarrow\gamma$ uniformly as $M\rightarrow\infty$.)
\end{remark}

%

\vspace{0.3cm}%

\begin{proof}
[\textbf{Proof of Theorem \ref{T_HTSforVaryingMeasures}.}]\textbf{(i)} The
second assertion, implication (\ref{Eq_UniformCgeOnCptSetU2}), is immediate
from (\ref{Eq_UniformCgeOnCptSetNew}), so we focus on proving the latter. In
steps (ii)-(iv) below we are going to show that for every $\varepsilon>0$,
there is some $l^{\prime}=l^{\prime}(\varepsilon)$ such that for any
$\chi:\mathfrak{E}\rightarrow\mathbb{R}$ with $\left\vert \chi\right\vert
\leq1$ and $\mathrm{Lip}(\varkappa)\leq1$,
\begin{equation}
\left\vert \int\chi\circ R_{l}\,d\overline{\nu}-\int\chi\circ R_{l}%
\,d\nu\right\vert <\varepsilon\text{ \quad for }l\geq l^{\prime}\text{ and
}\nu,\overline{\nu}\in\mathfrak{K}\text{.}\label{Eq_Abschaetzig}%
\end{equation}
To see that this implies (\ref{Eq_UniformCgeOnCptSetNew}), recall the
definition (\ref{Eq_DefMeasDist}) of $D_{\mathfrak{E}}$ via our specific
sequence $(\chi_{j})$, and that the $\chi_{j}$ satisfy the above assumptions
on $\chi$. Therefore (\ref{Eq_UniformCgeOnCptSetNew}) follows, since
\begin{align*}
D_{\mathfrak{E}}(\mathrm{law}_{\nu}(R_{l}),\mathrm{law}_{\overline{\nu}}%
(R_{l}))  & =\sum_{j\geq1}2^{-(j+1)}\left\vert \int\chi_{j}\circ
R_{l}\,d\overline{\nu}-\int\chi_{j}\circ R_{l}\,d\nu\right\vert \\
& \leq\sum\limits_{j\geq1}2^{-(j+1)}\varepsilon=\varepsilon\text{\quad for
}\nu,\overline{\nu}\in\mathfrak{K}\text{ \quad if }l\geq l^{\prime}\text{.}%
\end{align*}
\newline\textbf{(ii)} Consider the maps $\gamma_{M}:\mathfrak{P}%
\rightarrow\mathfrak{P}$ with $\gamma_{M}(\nu):=M^{-1}%
{\textstyle\sum_{m=0}^{M-1}}
T_{\ast}^{m}\nu$. By ergodicity and Theorem \ref{T_Yosida}, $(\gamma_{M})$
converges pointwise to the constant map $\gamma(\nu):=\mu$, as $\gamma_{M}%
(\nu)\rightarrow\mu$ for every $\nu\in\mathfrak{P}$. But $(\gamma_{M}%
)_{M\geq1}$ is equicontinuous. Indeed, due to the identification of the $\nu$
with their densities, it suffices to observe that all the operators $M^{-1}%
{\textstyle\sum\nolimits_{m=0}^{M-1}}
\widehat{T}^{m}$ have norm equal to $1$ on $L_{1}(\mu)$. Hence, compactness of
$\mathfrak{K}$ entails uniform convergence (see Remark
\ref{R_UnifCgeByEquiCty}),
\[
d_{\mathfrak{P}}\left(  \frac{1}{M}\sum_{m=0}^{M-1}T_{\ast}^{m}\overline{\nu
},\frac{1}{M}\sum_{m=0}^{M-1}T_{\ast}^{m}\nu\right)  \longrightarrow0\text{
\quad}%
\begin{array}
[c]{c}%
\text{as }M\rightarrow\infty\text{,}\\
\text{uniformly in }\nu,\overline{\nu}\in\mathfrak{K}\text{.}%
\end{array}
\]
Therefore there is some $M_{\varepsilon}\geq1$ (henceforth fixed) such that
\[
d_{\mathfrak{P}}\left(  \frac{1}{M_{\varepsilon}}\sum_{m=0}^{M_{\varepsilon
}-1}T_{\ast}^{m}\overline{\nu},\frac{1}{M_{\varepsilon}}\sum_{m=0}%
^{M_{\varepsilon}-1}T_{\ast}^{m}\nu\right)  <\frac{\varepsilon}{4}\text{ \quad
for }\nu,\overline{\nu}\in\mathfrak{K}\text{.}%
\]
Consequently (as $\left\vert \chi\right\vert \leq1$), for every $l$ and all
$\nu,\overline{\nu}\in\mathfrak{K}$,
\begin{equation}
\left\vert \int\chi\circ R_{l}\,d\left(  \frac{1}{M_{\varepsilon}}\sum
_{m=0}^{M_{\varepsilon}-1}T_{\ast}^{m}\overline{\nu}\right)  -\int\chi\circ
R_{l}\,d\left(  \frac{1}{M_{\varepsilon}}\sum_{m=0}^{M_{\varepsilon}-1}%
T_{\ast}^{m}\nu\right)  \right\vert <\frac{\varepsilon}{4}\text{.}%
\label{Eq_TheSmearedComparison}%
\end{equation}
\textbf{(iii)} As $\mathfrak{K}$ is compact in $(\mathfrak{P},d_{\mathfrak{P}%
})$, the family $\{d\nu/d\mu:\nu\in\mathfrak{K}\}$ is compact in $L_{1}(\mu)$,
and hence uniformly integrable. Thus, there is some $\delta=\delta
(\varepsilon)>0$ such that
\begin{equation}
\nu(A)=\int_{A}\frac{d\nu}{d\mu}\,d\mu<\frac{\varepsilon}{16M_{\varepsilon}%
}\text{ \quad for }\nu\in\mathfrak{K}\text{ and }A\in\mathcal{A}\text{ with
}\mu(A)<\delta\text{.}\label{Eq_HuHu}%
\end{equation}
Set $\eta:=(8M_{\varepsilon})^{-1}\varepsilon>0$ and define sequences
$(\Xi_{i,l})_{l\geq1}$ via $\Xi_{i,l}:=\chi\circ R_{l}\circ T^{i}$, $i\geq0$.
We claim that there is some $l^{\prime}=l^{\prime}(\varepsilon)$ s.t.
\begin{equation}
\mu\left(  \left\vert \Xi_{i,l}-\Xi_{i,l}\circ T\right\vert >\eta\right)
<\delta\text{ \quad for }i\geq0\text{ and }l\geq l^{\prime}\text{.}%
\label{Eq_DefiningLPrime}%
\end{equation}
Indeed, using asymptotic $T$-invariance in measure of $(R_{l})$ we see that
there is some $l^{\prime}$ such that
\[
\mu\left(  d_{\mathfrak{E}}(R_{l},R_{l}\circ T)>\eta\right)  <\delta\text{
\quad for }l\geq l^{\prime}\text{.}%
\]
Due to $T$-invariance of $\mu$ and $\mathrm{Lip}(\varkappa)\leq1$ we then find
that
\begin{align*}
\mu\left(  \left\vert \Xi_{i,l}-\Xi_{i,l}\circ T\right\vert >\eta\right)   &
=\mu\left(  \left\vert \chi\circ R_{l}-\chi\circ R_{l}\circ T^{i}\right\vert
>\eta\right)  \\
&  \leq\mu\left(  d_{\mathfrak{E}}(R_{l},R_{l}\circ T)>\eta\right)  \\
&  <\delta\text{ \quad for }i\geq0\text{ and }l\geq l^{\prime}\text{,}%
\end{align*}
as required. Using $\left\vert \Xi_{i,l}\right\vert \leq1$ and (\ref{Eq_HuHu})
we then see that
\begin{align}
\int\left\vert \Xi_{i,l}-\Xi_{i,l}\circ T\right\vert \,d\nu &  \leq\eta
+\int_{\{\left\vert \Xi_{i,l}-\Xi_{i,l}\circ T\right\vert >\eta\}}\left\vert
\Xi_{i,l}-\Xi_{i,l}\circ T\right\vert \,d\nu\nonumber\\
&  <\frac{\varepsilon}{4M_{\varepsilon}}\text{ \quad for }i\geq0\text{, }l\geq
l^{\prime}\text{ and }\nu\in\mathfrak{K}\text{.}\label{Eq_jshdvjhavfk}%
\end{align}
\textbf{(iv)} Note that by duality and (\ref{Eq_jshdvjhavfk}),
\begin{align}
&  \left\vert \int\chi\circ R_{l}\,d\nu-\int\chi\circ R_{l}\,d\left(  \frac
{1}{M_{\varepsilon}}\sum_{m=0}^{M_{\varepsilon}-1}T_{\ast}^{m}\nu\right)
\right\vert \nonumber\\
&  \leq\frac{1}{M_{\varepsilon}}\sum_{m=0}^{M_{\varepsilon}-1}\left\vert
\int(\chi\circ R_{l}-\chi\circ R_{l}\circ T^{m})\,d\nu\right\vert \nonumber\\
&  \leq\frac{1}{M_{\varepsilon}}\sum_{m=0}^{M_{\varepsilon}-1}\sum_{i=0}%
^{m-1}\int\left\vert \Xi_{i,l}-\Xi_{i,l}\circ T\right\vert \,d\nu\nonumber\\
&  <\frac{1}{M_{\varepsilon}}\sum_{m=0}^{M_{\varepsilon}-1}\sum_{i=0}%
^{m-1}\frac{\varepsilon}{4M_{\varepsilon}}\leq\frac{\varepsilon}{4}\text{
\quad for }l\geq l^{\prime}\text{ and }\nu\in\mathfrak{K}\text{.}%
\label{Eq_Smearing1}%
\end{align}
Now take any $\nu,\overline{\nu}\in\mathfrak{K}$, and combine
(\ref{Eq_TheSmearedComparison}) with an application of (\ref{Eq_Smearing1}) to
$\nu$ and another application of (\ref{Eq_Smearing1}) to $\overline{\nu}$ to
obtain (\ref{Eq_Abschaetzig}).
\end{proof}

%

\vspace{0.3cm}%
%

\noindent
\textbf{Waiting for good measure(s).} We shall say that the measurable
functions $\tau_{l}:X\rightarrow\mathbb{N}_{0}$ form an \emph{admissible delay
sequence} $(\tau_{l})_{l\geq1}$ \emph{for} $(R_{l})$ \emph{and} $(\nu_{l})$
if
\begin{equation}
D_{\mathfrak{E}}(\mathrm{law}_{\nu_{l}}(R_{l}),\mathrm{law}_{\nu_{l}}%
(R_{l}\circ T^{\tau_{l}}))\longrightarrow0\quad\text{as }l\rightarrow
\infty\text{.}%
\end{equation}
An easy sufficient condition for this is that
\begin{equation}
d_{\mathfrak{E}}(R_{l}\circ T^{\tau_{l}},R_{l})\overset{\nu_{l}}%
{\longrightarrow}0\text{ \quad as }l\rightarrow\infty\text{.}
\label{Eq_vbvbvbvbvbvbvbvvbvbvbvbvffff}%
\end{equation}
(By compactness of $\mathfrak{M}(\mathfrak{E})$ it suffices to show that for
any subsequence $l_{j}\nearrow\infty$ of indices and any random element $R$ of
$\mathfrak{E}$, $R_{l_{j}}\overset{\nu_{l_{j}}}{\Longrightarrow}R$ implies
$R_{l_{j}}\circ T^{\tau_{l}}\overset{\nu_{l_{j}}}{\Longrightarrow}R$, which
follows from (\ref{Eq_vbvbvbvbvbvbvbvvbvbvbvbvffff}) by a standard (Slutsky)
argument, see \cite{Bi}, Theorem 3.1).

Since $\mathrm{law}_{\nu_{l}}(R_{l}\circ T^{\tau_{l}})=\mathrm{law}%
_{\overline{\nu}_{l}}(R_{l})$ where $\overline{\nu}_{l}:=T_{\ast}^{\tau_{l}%
}\nu_{l}$, we can efficiently use admissible delays in situations where the
sequence $(\overline{\nu}_{l})$ of these push-forwards allows for better
control than $(\nu_{l})$. The latter phrase will mean that $\overline{\nu}%
_{l}\in\mathfrak{K}$ for all $l$, where $\mathfrak{K}\subseteq\mathfrak{P}$ is
compact, in which case we can use the following straightforward consequence of
Theorem \ref{T_HTSforVaryingMeasures}.

\begin{proposition}
[\textbf{Asymptotically invariant sequences - }$\nu_{l}$\textbf{\ versus }%
$\mu$]\label{P_AsyInvarSeqNuLVsMu}Let $(X,\mathcal{A},\mu,T)$ be an ergodic
probability preserving system, $(R_{l})_{l\geq1}$ a sequence of Borel
measurable maps $R_{l}:X\rightarrow\mathfrak{E}$, asymptotically $T$-invariant
in measure, into a compact metric space $(\mathfrak{E},d_{\mathfrak{E}})$, and
$(\nu_{l})_{l\geq1}$ a sequence in $\mathfrak{P}$. Suppose that $(\tau
_{l})_{l\geq1}$ is an admissible delay sequence for $(R_{l})$ and $(\nu_{l})$
such that there is some compact set $\mathfrak{K}$ in $(\mathfrak{P}%
,d_{\mathfrak{P}})$ for which
\begin{equation}
T_{\ast}^{\tau_{l}}\nu_{l}\in\mathfrak{K}\text{ \quad for }l\geq1\text{.}%
\end{equation}
Then
\begin{equation}
D_{\mathfrak{E}}(\mathrm{law}_{\nu_{l}}(R_{l}),\mathrm{law}_{\mu}%
(R_{l}))\longrightarrow0\quad\text{as }l\rightarrow\infty\text{.}%
\end{equation}

\end{proposition}

\begin{proof}
Letting $\overline{\nu}_{l}:=T_{\ast}^{\tau_{l}}\nu_{l}$ we have
$D_{\mathfrak{E}}(\mathrm{law}_{\nu_{l}}(R_{l}),\mathrm{law}_{\overline{\nu
}_{l}}(R_{l}))\rightarrow0$ due to admissibility of $(\tau_{l})$, but also
$D_{\mathfrak{E}}(\mathrm{law}_{\overline{\nu}_{l}}(R_{l}),\mathrm{law}_{\mu
}(R_{l}))\rightarrow0$ by Theorem \ref{T_HTSforVaryingMeasures} (We can assume
w.l.o.g. that $\mu\in\mathfrak{K}$).
\end{proof}

%

\vspace{0.3cm}%

In the specific situation of Theorems
\ref{T_MOreFlexiblePoissonLimitFromCompactness}, \ref{T_AsyIntStateProcIndep},
and \ref{T_SpatiotemporalIIDLimits} above, the $\nu_{l}=\mu_{A_{l}}$
concentrate on ever smaller sets, and hence can never stay inside a single
compact set $\mathfrak{K}$, while suitable push-forwards $\overline{\nu}_{l} $
sometimes do.%

\vspace{0.3cm}%
%

\noindent
\textbf{Independent limits for pairs of asymptotcally invariant sequences.} To
facilitate the analysis of distributional limits of processes involving
several asymptotically invariant sequences, we provide a natural method of
checking asymptotic independence. It relies on the following easy probability fact.

\begin{lemma}
[\textbf{Independence of limit variables by conditioning}]%
\label{L_IndependentLimitPair}Let $(\nu_{l})_{l\geq1}$ be probability measures
on $(X,\mathcal{A})$, and $R_{l}:X\rightarrow\mathfrak{E}$, $R_{l}^{\prime
}:X\rightarrow\mathfrak{E}^{\prime}$, $l\geq1$, Borel maps into the compact
metric spaces $\mathfrak{E}$ and $\mathfrak{E}^{\prime}$. Let $(R,R^{\prime})$
be a random element of $\mathfrak{E}\times\mathfrak{E}^{\prime}$ such that
\begin{equation}
(R_{l},R_{l}^{\prime})\overset{\nu_{l}}{\Longrightarrow}(R,R^{\prime})\text{
}\quad\text{as }l\rightarrow\infty\text{.} \label{Eq_hdgsvczwetv}%
\end{equation}
Assume that there is some $\pi$-system $\mathcal{B}_{\mathfrak{E}}^{\pi}$
generating $\mathcal{B}_{\mathfrak{E}}$ such that for all $E\in\mathcal{B}%
_{\mathfrak{E}}^{\pi}$ we have $\Pr[R\in\partial E]=0$, and, in case $\Pr[R\in
E]>0$, convergence in law holds under the $\nu_{l}$ conditioned on $\{R_{l}\in
E\}$,
\begin{equation}
R_{l}^{\prime}\overset{\nu_{l,\{R_{l}\in E\}}}{\Longrightarrow}R^{\prime
}\text{ }\quad\text{as }l\rightarrow\infty\text{.} \label{Eq_utioioio}%
\end{equation}
Then $R$ and $R^{\prime}$ are independent.
\end{lemma}

\begin{proof}
By Theorem 2.8 of \cite{Bi}, our assumption (\ref{Eq_hdgsvczwetv}) is
equivalent to
\begin{gather}
\nu_{l}(R_{l}\in E,R_{l}^{\prime}\in E^{\prime})\longrightarrow\Pr[R\in
E,R^{\prime}\in E^{\prime}]\text{ \quad as }l\rightarrow\infty
\label{Eq_hjgkhjkhjkfkhjkgfkhjkgfjh}\\
\text{whenever }\Pr[R\in\partial E]=\Pr[R^{\prime}\in\partial E^{\prime
}]=0\text{. }\nonumber
\end{gather}
We show that for any such $(E,E^{\prime})$ this limit coincides with $\Pr[R\in
E]\,\Pr[R^{\prime}\in E^{\prime}]$. Applying the same theorem from \cite{Bi}
again, then proves that $(R_{l},R_{l}^{\prime})$ converges to an independent
pair with marginals $R$, $R^{\prime}$.

Take any $E\in\mathcal{B}_{\mathfrak{E}}^{\pi}$. If $\Pr[R\in E]=0$, the
assertion is trivial. Assume therefore that $\Pr[R\in E]>0$. Pick any
$E^{\prime}\in\mathcal{B}_{\mathfrak{E}^{\prime}}$ with $\Pr[R^{\prime}%
\in\partial E^{\prime}]=0$. Due to (\ref{Eq_utioioio}),
\[
\nu_{l}(R_{l}\in E,R_{l}^{\prime}\in E^{\prime})\longrightarrow\Pr[R\in
E]\,\Pr[R^{\prime}\in E^{\prime}]\quad\text{as }l\rightarrow\infty\text{,}%
\]
and therefore
\[
\Pr[R\in E,R^{\prime}\in E^{\prime}]=\Pr[R\in E]\,\Pr[R^{\prime}\in E^{\prime
}]\text{.}%
\]
Fixing such an $E^{\prime}$, the standard uniqueness theorem for measures
shows that
\[
\Pr[R\in B,R^{\prime}\in E^{\prime}]=\Pr[R\in B]\,\Pr[R^{\prime}\in E^{\prime
}]\text{ \quad for all }B\in\mathcal{B}_{\mathfrak{E}}\text{.}%
\]
In particular, this is true for all $B\in\mathcal{B}_{\mathfrak{E}}$ with
$\Pr[R\in\partial B]=0$.
\end{proof}

%

\vspace{0.3cm}%

Combining this with the uniform distributional convergence principle, and the
idea that admissible time delays may result in good measures, we obtain

\begin{theorem}
[\textbf{Asymptotic independence of two sequences}]%
\label{T_AsyIndepForAsyInvarSequences}Let $(X,\mathcal{A},\mu,T)$ be an
ergodic probability preserving system and $(R_{l})_{l\geq1}$, $(R_{l}^{\prime
})_{l\geq1}$ sequences of Borel measurable maps, $(R_{l}^{\prime})$
asymptotically $T$-invariant in measure, into compact metric spaces
$(\mathfrak{E},d_{\mathfrak{E}})$ and $(\mathfrak{E}^{\prime},d_{\mathfrak{E}%
^{\prime}})$, respectively. Suppose that $\nu_{l}$, $l\geq1$, are
probabilities on $(X,\mathcal{A})$ such that
\begin{equation}
(R_{l},R_{l}^{\prime})\overset{\nu_{l}}{\Longrightarrow}(R,R^{\prime})\text{
}\quad\text{as }l\rightarrow\infty
\end{equation}
for some random element $(R,R^{\prime})$ of $\mathfrak{E}\times\mathfrak{E}%
^{\prime}$, and let $\mathcal{B}_{\mathfrak{E}}^{\pi}$ be some $\pi$-system
generating $\mathcal{B}_{\mathfrak{E}}$ such that $\mathfrak{E}\in
\mathcal{B}_{\mathfrak{E}}^{\pi}$ while $\Pr[R\in\partial E]=0$ for all
$E\in\mathcal{B}_{\mathfrak{E}}^{\pi}$.

Assume that for every $E\in\mathcal{B}_{\mathfrak{E}}^{\pi}$ with $\Pr[R\in
E]>0$ there is some compact set $\mathfrak{K}_{E}$ in $(\mathfrak{P}%
,d_{\mathfrak{P}})$, a sequence $(\nu_{l,E})_{l\geq1}$ in $\mathfrak{P}$ with
$d_{\mathfrak{P}}(\nu_{l,E},\nu_{l,\{R_{l}\in E\}})\rightarrow0$, and a
sequence $(\tau_{l,E})_{l\geq1}$ of admissible delays for $(R_{l}^{\prime
})_{l\geq1}$ and $(\nu_{l,E})_{l\geq1}$ such that $\overline{\nu}%
_{l,E}:=T_{\ast}^{\tau_{l,E}}\nu_{l,E}\in\mathfrak{K}_{E}$ for $l\geq1$. Then
$R$ and $R^{\prime}$ are independent.
\end{theorem}

(Note that $(R_{l})$ is not required to be asymptotically $T$-invariant.)

\begin{proof}
By Lemma \ref{L_IndependentLimitPair} it suffices to show that for every
$E\in\mathcal{B}_{\mathfrak{E}}^{\pi}$ with $\Pr[R\in E]>0$,
\[
R_{l}^{\prime}\overset{\nu_{l,\{R_{l}\in E\}}}{\Longrightarrow}R^{\prime
}\text{ }\quad\text{as }l\rightarrow\infty\text{,}%
\]
which, due to $d_{\mathfrak{P}}(\nu_{l,E},\nu_{l,\{R_{l}\in E\}})\rightarrow0$
and (\ref{Eq_vcbycxvcygvchatecv356463456}), is equivalent to
\begin{equation}
R_{l}^{\prime}\overset{\nu_{l,E}}{\Longrightarrow}R^{\prime}\text{ }%
\quad\text{as }l\rightarrow\infty\text{.} \label{Eq_LocalGoal}%
\end{equation}
Take any such $E$. Since $\overline{\nu}_{l,E}\in\mathfrak{K}_{E}$ for all
$l$, an application of Theorem \ref{T_HTSforVaryingMeasures} shows that
$D_{\mathfrak{E}}(\mathrm{law}_{\mu}(R_{l}^{\prime}),\mathrm{law}%
_{\overline{\nu}_{l,E}}(R_{l}^{\prime}))\rightarrow0$, and since the
$\tau_{l,E}$ are admissible delays for $(R_{l}^{\prime})_{l\geq1}$ and
$(\nu_{l,E})_{l\geq1}$, we conclude that
\begin{equation}
D_{\mathfrak{E}}(\mathrm{law}_{\mu}(R_{l}^{\prime}),\mathrm{law}_{\nu_{l,E}%
}(R_{l}^{\prime}))\longrightarrow0\text{ }\quad\text{as }l\rightarrow
\infty\text{.} \label{Eq_aaldhgtyh}%
\end{equation}
The case $E:=\mathfrak{E}$ yields $D_{\mathfrak{E}}(\mathrm{law}_{\mu}%
(R_{l}^{\prime}),\mathrm{law}_{\nu_{l}}(R_{l}^{\prime}))\rightarrow0$.
Together with $R_{l}^{\prime}\overset{\nu_{l}}{\Longrightarrow}R^{\prime}$ and
(\ref{Eq_aaldhgtyh}) this gives (\ref{Eq_LocalGoal}).
\end{proof}

%

\vspace{0.3cm}%

\begin{remark}
[\textbf{The auxiliary measures }$\nu_{l,E}$]In the simplest cases, we can
take $\nu_{l,E}:=\nu_{l,\{R_{l}\in E\}}$. However, constructing suitable
$\tau_{l,E}$ is sometimes easier if we use a slightly different sequence of
measures, obtained as follows.

It is easily seen that if $(B_{l})$ and $(B_{l}^{\prime})$ are sequences in
$\mathcal{A}$ with $\mu(B_{l}\bigtriangleup B_{l}^{\prime})=o(\mu(B_{l}))$,
then $d_{\mathfrak{P}}(\mu_{B_{l}},\mu_{B_{l}^{\prime}})\rightarrow0$.
Therefore, if $\nu_{l}=\mu_{A_{l}}$ and the sets $B_{l,E}^{\prime}%
\in\mathcal{A}$ are such that $\mu(B_{l,E}^{\prime}\bigtriangleup(A_{l}%
\cap\{R_{l}\in E\}))=o(\mu(A_{l}\cap\{R_{l}\in E\}))$, then the measures
$\nu_{l,E}:=\mu_{B_{l,E}^{\prime}}$ satisfy $d_{\mathfrak{P}}(\nu_{l,E}%
,\nu_{l,\{R_{l}\in E\}})\rightarrow0$.
\end{remark}

%

\vspace{0.3cm}%

\section{Proofs for return- and hitting-time processes\label{S_PfsRH}}%

\noindent
\textbf{Asymptotic invariance of hitting time processes.} A sequence of
hitting time processes for rare events, that is, a sequence $(R_{l})$ of
variables $R_{l}=\mu(A_{l})\Phi_{A_{l}}$, viewed through the single measure
$\mu$, is asymptotically $T$-invariant in measure (\cite{Z7}, Corollary 6). We
provide a more precise statement in the next proposition.

\begin{proposition}
[\textbf{Asymptotic invariance in measure of hitting-time processes}%
]\label{P_AsyInvarHitProc}Let $(X,\mathcal{A},\mu,T)$ be an ergodic
probability preserving system. \newline\textbf{a)} For every set
$A\in\mathcal{A}$ and integer $m\geq0$,%
\begin{equation}
d_{[0,\infty]^{\mathbb{N}_{0}}}(\mu(A)\Phi_{A}\circ T^{m},\mu(A)\Phi_{A})\leq
m\mu(A)\text{ \quad on }\{\varphi_{A}>m\}\text{.} \label{Eq_uzuzuzuzu}%
\end{equation}
\textbf{b)} Suppose that $(A_{l})_{l\geq1}$ is a sequence of asymptotically
rare events, and set $R_{l}:=\mu(A_{l})\Phi_{A_{l}}:X\rightarrow
\lbrack0,\infty]^{\mathbb{N}_{0}}$. Then $(R_{l})$\emph{\ }is asymptotically
$T$-invariant in measure.\newline\textbf{c)} Likewise, if $R_{l}:=\mu
(A_{l})\Phi_{A_{l}}\circ T_{A_{l}}:X\rightarrow\lbrack0,\infty]^{\mathbb{N}}$,
then $(R_{l})$\emph{\ }is asymptotically $T$-invariant in measure.
\end{proposition}

\begin{proof}
\textbf{a)} For any $A\in\mathcal{A}$ and integer $m\geq0$,
\begin{equation}
\Phi_{A}=\Phi_{A}\circ T^{m}+(m,0,0,\ldots)\text{ \quad on }\{\varphi
_{A}>m\}\text{,} \label{Eq_KeyToAsyInvarianceOfHittProc}%
\end{equation}
and hence $d_{[0,\infty]^{\mathbb{N}_{0}}}(\mu(A)\Phi_{A}\circ T^{m}%
,\mu(A)\Phi_{A})=d_{[0,\infty]}(\mu(A)(\varphi_{A}-m),\mu(A)\varphi_{A})$ on
that set. Since $d_{[0,\infty]}(s,s+\delta)=e^{-s}(1-e^{-\delta})\leq\delta$
for all $s,\delta\in\lbrack0,\infty)$, (\ref{Eq_uzuzuzuzu}) follows.\newline%
\textbf{b)} Whenever $(A_{l})_{l\geq1}$ is a sequence of asymptotically rare
events, the $R_{l}$ satisfy
\begin{equation}
d_{[0,\infty]^{\mathbb{N}_{0}}}(R_{l}\circ T,R_{l})\leq\mu(A_{l})\text{ \quad
outside }\{\varphi_{A_{l}}\leq1\}=T^{-1}A_{l}\text{.}%
\end{equation}
By assumption this upper bound for the distance tends to zero, and since
$\mu(\varphi_{A_{l}}\leq1)=\mu(A_{l})$, so does the measure of the set on
which the bound fails to apply.\newline\textbf{c)} Analogous, using that
$R_{l}\circ T=R_{l}$ outside $T^{-1}A_{l}$.
\end{proof}

%

\vspace{0.3cm}%

The simple estimate (\ref{Eq_uzuzuzuzu}) immediately leads to sufficient
conditions for time delays $\tau_{l}$ to be admissible for a given sequence
$(\nu_{l})$ of initial densities.

\begin{proposition}
[\textbf{Admissible time delays for return or hitting processes}%
]\label{P_CharAdmissTimeDelay}Let $T$ be a measure-preserving map on the
probability space $(X,\mathcal{A},\mu)$, $(A_{l})$ a sequence of
asymptotically rare events, $(\nu_{l})$ a sequence in $\mathfrak{P}$, and
$\tau_{l}:X\rightarrow\mathbb{N}_{0}$, $l\geq1$, measurable functions.
\newline\textbf{a)} $(\tau_{l})_{l\geq1}$ is an admissible delay sequence for
the variables $R_{l}:X\rightarrow\lbrack0,\infty]^{\mathbb{N}_{0}}$ given by
$R_{l}:=\mu(A_{l})\Phi_{A_{l}}$ and the measures $\nu_{l}$ provided that
\begin{equation}
\mu(A_{l})\,\tau_{l}\overset{\nu_{l}}{\longrightarrow}0\quad\text{as
}l\rightarrow\infty, \label{Eq_ijijijiuhminiuhijijiji}%
\end{equation}
and
\begin{equation}
\nu_{l}\left(  \tau_{l}<\varphi_{A_{l}}\right)  \longrightarrow1\quad\text{as
}l\rightarrow\infty\text{.} \label{Eq_qwqwqwqqwqw}%
\end{equation}
\textbf{b)} Condition (\ref{Eq_qwqwqwqqwqw}) alone is sufficient for
$(\tau_{l})_{l\geq1}$ to be an admissible delay sequence for the variables
$R_{l}:X\rightarrow\lbrack0,\infty]^{\mathbb{N}}$ given by $R_{l}:=\mu
(A_{l})\Phi_{A_{l}}\circ T_{A_{l}}$ and the $\nu_{l}$.
\end{proposition}

\begin{proof}
\textbf{a)} For arbitrary $A\in\mathcal{A}$ and any measurable $\tau
:X\rightarrow\mathbb{N}_{0}$, we can apply (\ref{Eq_uzuzuzuzu}) on each set
$\{\tau=m\}$, $m\geq0$, to see that
\begin{equation}
d_{[0,\infty]^{\mathbb{N}_{0}}}(\mu(A)\Phi_{A}\circ T^{\tau},\mu(A)\Phi
_{A})\leq\tau\mu(A)\text{ \quad on }\{\varphi_{A}>\tau\}\text{.}%
\end{equation}
Now take any $\varepsilon>0$. By the above we find that for every $l$,
\begin{align}
d_{[0,\infty]^{\mathbb{N}_{0}}}(R_{l}\circ T^{\tau_{l}},R_{l})  &  \leq
\tau_{l}\mu(A_{l})\text{ \quad on }\{\varphi_{A_{l}}>\tau_{l}\}\nonumber\\
&  <\varepsilon\text{ \qquad\qquad\ on }\{\varphi_{A_{l}}>\tau_{l}\}\cap
\{\mu(A_{l})\,\tau_{l}<\varepsilon\}\text{.}%
\end{align}
But (\ref{Eq_ijijijiuhminiuhijijiji}) and (\ref{Eq_qwqwqwqqwqw}) ensure that
$\nu_{l}(\{\varphi_{A_{l}}>\tau_{l}\}\cap\{\mu(A_{l})\,\tau_{l}<\varepsilon
\})\rightarrow1$ as $l\rightarrow\infty$, which proves our claim via the
sufficient condition (\ref{Eq_vbvbvbvbvbvbvbvvbvbvbvbvffff}).\newline%
\textbf{b)} Since, for any $A$ and measurable $\tau$ we have $T_{A}\circ
T^{\tau}=T_{A}$ on $\{\tau<\varphi_{A}\}$, we see that $R_{l}\circ T^{\tau
_{l}}=R_{l}$ on $\{\tau_{l}<\varphi_{A_{l}}\}$, and the result follows.
\end{proof}

%

\vspace{0.3cm}%

We can thus establish the first theorem advertised in this paper.

\begin{proof}
[\textbf{Proof of Theorem \ref{T_ProcessesAsymptoticToHittingTimeProc}}%
]Conditions (\ref{Eq_kllklklklklk}) and (\ref{Eq_opopopopo}) guarantee, via
Propositions \ref{P_AsyInvarHitProc} b) and \ref{P_CharAdmissTimeDelay} a),
that $(R_{l})$ is asymptotically $T$-invariant in measure, and that $(\tau
_{l})$ is an admissible delay sequence for $(R_{l}):=(\mu(A_{l})\Phi_{A_{l}})$
and $(\nu_{l})$. Now (\ref{Eq_hbhbhbhbhbkkkk}) allows us to apply Proposition
\ref{P_AsyInvarSeqNuLVsMu}.
\end{proof}

%

\vspace{0.3cm}%
%

\vspace{0.3cm}%
%

\noindent
\textbf{Finite-dimensional marginals and distributional convergence.} A
sequence $s\in\lbrack0,\infty]^{\mathbb{N}_{0}}$ will be called
\emph{finite-valued} in case $s\in\lbrack0,\infty)^{\mathbb{N}_{0}}$. Let
$\Phi=(\varphi^{(j)})_{j\geq0}$ be a \emph{random sequence in} $[0,\infty]$,
that is, a random element of $[0,\infty]^{\mathbb{N}_{0}}$. We let
$\Phi^{\lbrack d]}:=(\varphi^{(0)},\ldots,\varphi^{(d-1)})$ denote its initial
piece of length $d$, $d\geq1$. The (possibly degenerate) \emph{distribution
function} of the random vector $\Phi^{\lbrack d]}$ is $F^{[d]}:[0,\infty
)^{d}\rightarrow\lbrack0,1]$, $F^{[d]}(t_{0},\ldots,t_{d-1}):=\Pr
[\varphi^{(0)}\leq t_{0},\ldots,\varphi^{(d-1)}\leq t_{d-1}]$.

Assume that each $\Phi_{l}$, $l\geq1$, is a random sequence in $[0,\infty) $
with finite-dimensional distribution functions $F_{l}^{[d]}:[0,\infty
)^{d}\rightarrow\lbrack0,1]$, $d\geq1$. Abbreviating $\{F_{l}^{[d]}%
\}:=\{F_{l}^{[d]}\}_{d\geq1}$ we shall write
\begin{equation}
\{F_{l}^{[d]}\}\Longrightarrow\{F^{[d]}\}\text{\quad as }l\rightarrow
\infty\text{,}%
\end{equation}
if all $F_{l}^{[d]}$ converge weakly, as $l\rightarrow\infty$, to the
corresponding distribution functions $F^{[d]}$ of $\Phi$, that is, for every
$d\geq1$ we have $F_{l}^{[d]}(t_{0},\ldots,t_{d-1})\rightarrow F^{[d]}%
(t_{0},\ldots,t_{d-1})$ at all continuity points $(t_{0},\ldots,t_{d-1})$ of
$F^{[d]}$.

\begin{remark}
This is the mode of convergence studied in \cite{Z11}, where it was denoted by
$\Phi_{l}\Longrightarrow\Phi$. It is closely related to the present meaning of
$\Phi_{l}\Longrightarrow\Phi$ (distributional convergence of random elements
of $[0,\infty]^{\mathbb{N}_{0}}$), which clearly implies $\{F_{l}%
^{[d]}\}\Longrightarrow\{F^{[d]}\}$. In fact, the two notions coincide in case
the $\Phi_{l}$ and $\Phi$ are a.s. finite-valued, which is always the case for
return-time processes $\Phi_{A_{l}}$ viewed through $\mu_{A_{l}}$, and their
limits $\widetilde{\Phi}$ (recall (\ref{Eq_BasicPropsAsyRetProc})).
\end{remark}

%

\vspace{0.3cm}%
%

\noindent
\textbf{The general duality between return- and hitting-time processes.} It is
a basic fact that for any sequence $(A_{l})_{l\geq1}$ of asymptotically rare
events, its return-time statistics and its hitting-time statistics are
intimately related to each other, as established in \cite{HaydnLacVai05}. This
result has been extended to the level of processes in \cite{Z11} (see also
\cite{Ma}), where we proved

\begin{theorem}
[\textbf{Hitting-time process versus return-time process; }\cite{Z11}%
]\label{T_HittingTimeVsReturnTimeProc}Let $(X,\mathcal{A},\mu,T)$ be an
ergodic probability-preserving system, and $(A_{l})_{l\geq1}$ a sequence of
asymptotically rare events. Let $\{F_{l}^{[d]}\}$ and $\{\widetilde{F}%
_{l}^{[d]}\}$ be the collections of finite-dimensional distribution functions
of $\mu(A_{l})\Phi_{A_{l}}$ under $\mu$ and the $\mu_{A_{l}}$, respectively.
Then
\begin{equation}
\{F_{l}^{[d]}\}\Longrightarrow\,\{F^{[d]}\}\text{\quad for some process }%
\Phi\text{ in }[0,\infty]\text{ with d.f.s }\{F^{[d]}\}
\label{Eq_CgeOfHittingProc}%
\end{equation}
iff
\begin{equation}
\{\widetilde{F}_{l}^{[d]}\}\Longrightarrow\,\{\widetilde{F}^{[d]}\}\text{\quad
for some process }\widetilde{\Phi}\text{ in }[0,\infty]\text{ with d.f.s
}\{\widetilde{F}^{[d]}\}\text{.} \label{Eq_CgeOfRetProc}%
\end{equation}
In this case, the sub-probability distribution functions $F^{[d]}$ and
$\widetilde{F}^{[d]}$ of $\Phi^{\lbrack d]}$ and $\widetilde{\Phi}^{[d]}$
satisfy, for any $d\geq0$ (where $\widetilde{F}^{[0]}:=1$) and $t_{j}\geq0$,
\newline%
\begin{equation}%
{\textstyle\int_{0}^{t_{0}}}
\,[\widetilde{F}^{[d]}(t_{1},\ldots,t_{d})-\widetilde{F}^{[d+1]}%
(s,t_{1},\ldots,t_{d})]\,ds=F^{[d+1]}(t_{0},t_{1},\ldots,t_{d})\text{.}
\label{Eq_RelationBetweenHigherDimDFs}%
\end{equation}
Through (\ref{Eq_RelationBetweenHigherDimDFs}), the families $\{F^{[d]}\}$ and
$\{\widetilde{F}^{[d]}\}$ uniquely determine each other.
\end{theorem}

We will heavily rely on this duality.%

\vspace{0.3cm}%
%

\noindent
\textbf{The case of Poisson asymptotics.} Below we will primarily be
interested in the particular case where the limit process is an iid sequence
of normalized exponentially distributed random variables $\Phi_{\mathrm{Exp}}%
$. If this particular limit occurs, then it automatically occurs both for the
hitting-times and for the return times, because\ of Theorem
\ref{T_HittingTimeVsReturnTimeProc} and

\begin{proposition}
[\textbf{Characterizing }$\Phi_{\mathrm{Exp}}$; \cite{Z11}]%
\label{P_CharacterizeExpos}Let $\Phi$\ be some stationary random sequence in
$[0,\infty)$. Then $\Phi=\Phi_{\mathrm{Exp}}$ iff the finite-dimensional
marginals have distribution functions $F^{[d]}$ satisfying
\begin{equation}
F^{[d+1]}(t_{0},t_{1},\ldots,t_{d})=%
{\textstyle\int_{0}^{t_{0}}}
\,[F^{[d]}(t_{1},\ldots,t_{d})-F^{[d+1]}(s,t_{1},\ldots,t_{d})]\,ds\text{ }
\label{Eq_asdfg}%
\end{equation}
whenever $d\geq0$ and $t_{j}\geq0$.
\end{proposition}

%

\vspace{0.3cm}%

Combining the above with Theorem \ref{T_ProcessesAsymptoticToHittingTimeProc}
we can now prove the abstract temporal Poisson limit theorem.%

\vspace{0.2cm}%

\begin{proof}
[\textbf{Proof of Theorem \ref{T_MOreFlexiblePoissonLimitFromCompactness}}%
]\textbf{(i)} Let $R_{l}:=\mu(A_{l})\Phi_{A_{l}}$, $l\geq1$, which gives an
asymptotically invariant sequence of Borel measurable maps $R_{l}%
:X\rightarrow\lbrack0,\infty]^{\mathbb{N}_{0}}$. Due to compactness of
$\mathfrak{M}([0,\infty]^{\mathbb{N}_{0}})$, it suffices to show that for
every subsequence of indices $l_{j}\nearrow\infty$ along which $R_{l_{j}%
}\overset{\mu_{A_{l_{j}}}}{\Longrightarrow}\widetilde{\Phi}$ and $R_{l_{j}%
}\overset{\mu}{\Longrightarrow}\Phi$ for some random elements $\widetilde
{\Phi},\Phi$ of $[0,\infty]^{\mathbb{N}_{0}}$, both limits are iid with
marginal $\widetilde{F}_{\mathrm{Exp}}$. Focusing on such a subsequence, we
thus assume that
\begin{equation}
R_{l}\overset{\mu_{A_{l}}}{\Longrightarrow}\widetilde{\Phi}\text{ \quad and
\quad}R_{l}\overset{\mu}{\Longrightarrow}\Phi\text{ \quad as }l\rightarrow
\infty\text{.} \label{Eq_zdnczntzhzzz}%
\end{equation}
We need to show that
\begin{equation}
\mathrm{law}(\widetilde{\Phi})=\mathrm{law}(\Phi)=\mathrm{law}(\widetilde
{\Phi}_{\mathrm{Exp}}). \label{Eq_TheBigPrize2}%
\end{equation}
\textbf{(ii)} Take any $\varepsilon>0$ and choose $(\nu_{l})_{l\geq1}$ in
$\mathfrak{P}$ with $d_{\mathfrak{P}}(\nu_{l},\mu_{A_{l}})<\varepsilon$ for
all $l$, a compact set $\mathfrak{K}\subseteq\mathfrak{P}$, and measurable
functions $\tau_{l}:X\rightarrow\mathbb{N}_{0}$ such that
(\ref{Eq_DontWaitTooLong1})-(\ref{Eq_TheFabulousCompactnessCondition1}) hold.
By Theorem \ref{T_ProcessesAsymptoticToHittingTimeProc} we have $D_{[0,\infty
]^{\mathbb{N}_{0}}}(\mathrm{law}_{\nu_{l}}(R_{l}),\mathrm{law}_{\mu}%
(R_{l}))\rightarrow0$, so that
\begin{equation}
R_{l}\overset{\nu_{l}}{\Longrightarrow}\Phi\text{ \quad as }l\rightarrow
\infty\text{.} \label{Eq_ookmokmokmokm}%
\end{equation}
\textbf{(iii)} In view of (\ref{Eq_ererrewrerwrertrerwerq}), however,
$D_{[0,\infty]^{\mathbb{N}_{0}}}(\mathrm{law}_{\nu_{l}}(R_{l}),\mathrm{law}%
_{\mu_{A_{l}}}(R_{l}))<\varepsilon$ for all $l$, and letting $l\rightarrow
\infty$, (\ref{Eq_zdnczntzhzzz}) and (\ref{Eq_ookmokmokmokm}) allow us to
conclude that $D_{[0,\infty]^{\mathbb{N}_{0}}}(\mathrm{law}(\Phi
),\mathrm{law}(\widetilde{\Phi}))\leq\varepsilon$. But $\varepsilon>0$ was
arbitrary, and therefore%
\begin{equation}
\mathrm{law}(\Phi)=\mathrm{law}(\widetilde{\Phi})\text{.}%
\end{equation}
For the finite dimensional distribution functions $\{F^{[d]}\}$ and
$\{\widetilde{F}^{[d]}\}$ of $\Phi$ and $\widetilde{\Phi}$ this means that
$\{F^{[d]}\}=\{\widetilde{F}^{[d]}\}$. Together with the fact that
$\{F^{[d]}\}$ and $\{\widetilde{F}^{[d]}\}$ are related to each other as in
(\ref{T_HittingTimeVsReturnTimeProc}) of Theorem
\ref{T_HittingTimeVsReturnTimeProc}, this shows that $\{F^{[d]}\}$ satisfies
condition (\ref{Eq_asdfg}), and (\ref{Eq_TheBigPrize2}) follows by Proposition
\ref{P_CharacterizeExpos}.
\end{proof}

%

\vspace{0.3cm}%
%

\noindent
\textbf{Including a point mass at zero.} The argument for convergence to
$\widetilde{\Phi}_{(\mathrm{Exp},\theta)}$ is similar to that for
$\widetilde{\Phi}_{\mathrm{Exp}}$. As a warm-up we characterize the
one-dimensional distribution function $\widetilde{F}_{(\mathrm{Exp},\theta)} $
through a generalization of the fixed point equation $\widetilde{F}(t)=%
{\textstyle\int_{0}^{t}}
[1-\widetilde{F}(s)]\,ds$ distinguishing the exponential distribution function
$F=\widetilde{F}_{(\mathrm{Exp},1)}$. In a second step, we provide a
characterization of $\widetilde{\Phi}_{(\mathrm{Exp},\theta)}$ similar to
Proposition \ref{P_CharacterizeExpos}.

\begin{proposition}
[\textbf{Characterizing }$\widetilde{F}_{(\mathrm{Exp},\theta)}$ \textbf{and}
$\widetilde{\Phi}_{(\mathrm{Exp},\theta)}$]\label{P_HandsUp}Take any
$\theta\in(0,1]$. \newline\textbf{a)} If $\widetilde{F}$ is a probability
distribution function on $[0,\infty)$, then $\widetilde{F}=\widetilde
{F}_{(\mathrm{Exp},\theta)}$ iff
\begin{equation}
\widetilde{F}(t)=(1-\theta)+\theta%
{\textstyle\int_{0}^{t}}
[1-\widetilde{F}(s)]\,ds\text{ \quad for }t\geq0\text{.}
\label{Eq_CharDelayedRetDF}%
\end{equation}
\textbf{b)} If $\widetilde{\Phi}$ is a stationary sequence of random variables
in $[0,\infty)$ with finite-dimensional distribution functions $\widetilde
{F}^{[d]}$ (where $\widetilde{F}^{[0]}:=1$), then $\mathrm{law}(\widetilde
{\Phi})=\mathrm{law}(\widetilde{\Phi}_{(\mathrm{Exp},\theta)})$ iff
\begin{align}
\widetilde{F}^{[d+1]}(t_{0},t_{1},\ldots,t_{d})  &  =(1-\theta)\,\widetilde
{F}^{[d]}(t_{1},\ldots,t_{d})\label{Eq_CharHigherDimCompoundDF}\\
&  +\theta%
{\textstyle\int_{0}^{t_{0}}}
\,[\widetilde{F}^{[d]}(t_{1},\ldots,t_{d})-\widetilde{F}^{[d+1]}%
(s,t_{1},\ldots,t_{d})]\,ds\nonumber
\end{align}
whenever $d\geq0$ and $t_{j}\geq0$.
\end{proposition}

\begin{proof}
\textbf{a)} It is immediate that $\widetilde{F}_{(\mathrm{Exp},\theta)}$ from
(\ref{Eq_DefOneDimDFComp}) satisfies (\ref{Eq_CharDelayedRetDF}). For the
converse, assume (\ref{Eq_CharDelayedRetDF}) and let $F(t):=%
{\textstyle\int_{0}^{t}}
[1-\widetilde{F}(s)]\,ds$ for $t\geq0$. Due to (\ref{Eq_CharDelayedRetDF}) we
have $1-\widetilde{F}(s)=\theta\lbrack1-F(s)]$ and hence
\[
F(t)=\theta%
{\textstyle\int_{0}^{t}}
[1-F(s)]\,ds\text{ \quad for }t\geq0\text{.}%
\]
Therefore $F$ is $\mathcal{C}^{\infty}$ on $(0,\infty)$, and $\overline
{F}(t):=1-F(t)$ satisfies $\overline{F}^{\prime}=-\theta\overline{F}$ there.
Consequently, $\overline{F}(t)=ce^{-\theta t}$, and since $F(0^{+})=0$ we have
$c=1$.\newline\newline\textbf{b)} Using (\ref{Eq_CharDelayedRetDF}) it is
straightforward that the marginals of the iid sequence $\widetilde{\Phi
}_{(\mathrm{Exp},\theta)}$ satisfy (\ref{Eq_CharHigherDimCompoundDF}). For the
converse, assume that $\widetilde{\Phi}$ satisfies
(\ref{Eq_CharHigherDimCompoundDF}).

The $d=0$ case covered by part a) shows that $\widetilde{F}^{[1]}%
=\widetilde{F}_{(\mathrm{Exp},\theta)}$. Write $\widetilde{\Phi}%
=(\widetilde{\varphi}^{(j)})_{j\geq0}$, then by stationarity, each
$\widetilde{\varphi}^{(j)}$ has distribution $\widetilde{F}_{(\mathrm{Exp}%
,\theta)}$. We need to prove that the $\widetilde{\varphi}^{(j)}$ are
independent. Using stationarity again, we see that it suffices to check that%
\begin{equation}
\text{for }d\geq1\text{, the variable }\widetilde{\varphi}^{(0)}\text{ is
independent of }\{\widetilde{\varphi}^{(1)},\ldots,\widetilde{\varphi}%
^{(d)}\}\text{.} \label{Eq_bxcvyxvnyvnv1}%
\end{equation}
Fix any $d\geq1$, and take $(t_{1},\ldots,t_{d})$ such that $\widetilde
{F}^{[d]}(t_{1},\ldots,t_{d})>0$. Define
\begin{align*}
\widetilde{G}(s)  &  :=\Pr[\widetilde{\varphi}^{(0)}\leq s\mid\widetilde
{\varphi}^{(1)}\leq t_{1},\ldots,\widetilde{\varphi}^{(d)}\leq t_{d}]\\
&  =\widetilde{F}^{[d+1]}(s,t_{1},\ldots,t_{d})/\widetilde{F}^{[d]}%
(t_{1},\ldots,t_{d})
\end{align*}
for $s\geq0$. Then (\ref{Eq_CharHigherDimCompoundDF}) becomes
\[
\widetilde{G}(t)=(1-\theta)+\theta%
{\textstyle\int_{0}^{t}}
\,[1-\widetilde{G}(s)]\,ds\text{ \quad for }s\geq0\text{.}%
\]
But since $\widetilde{\varphi}^{(0)}$ takes values in $[0,\infty)$, part a)
ensures that $\widetilde{G}=\widetilde{F}_{(\mathrm{Exp},\theta)}$, meaning
that $\Pr[\widetilde{\varphi}^{(0)}\leq s\mid\widetilde{\varphi}^{(1)}\leq
t_{1},\ldots,\widetilde{\varphi}^{(d)}\leq t_{d}]=\Pr[\widetilde{\varphi
}^{(0)}\leq s]$ whenever the conditioning event has positive probability. This
establishes (\ref{Eq_bxcvyxvnyvnv1}).
\end{proof}

%

\vspace{0.3cm}%

We are now ready for the proof of Theorem \ref{T_CPoissonViaCompactness}. The
strategy is the same as in the case of Theorem
\ref{T_MOreFlexiblePoissonLimitFromCompactness}, but we now split off the
contribution of points which return within time $\tau_{l}$.

\begin{proof}
[Proof of Theorem \ref{T_CPoissonViaCompactness}]\textbf{(i)} Let $R_{l}%
:=\mu(A_{l})\Phi_{A_{l}}$, $l\geq1$. As in the proof of Theorem
\ref{T_MOreFlexiblePoissonLimitFromCompactness} we can assume w.l.o.g. that
\begin{equation}
R_{l}\overset{\mu_{A_{l}}}{\Longrightarrow}\widetilde{\Phi}\text{ \quad and
\quad}R_{l}\overset{\mu}{\Longrightarrow}\Phi\text{ \quad as }l\rightarrow
\infty\text{.} \label{Eq_AGoodSubSequenceHere}%
\end{equation}
We need to show that
\begin{equation}
\mathrm{law}(\widetilde{\Phi})=\mathrm{law}(\widetilde{\Phi}_{(\mathrm{Exp}%
,\theta)}). \label{Eq_TheBigPrize}%
\end{equation}
\textbf{(ii)} Take any $\varepsilon>0$ and choose $(\nu_{l})_{l\geq1}$ in
$\mathfrak{P}$ with $d_{\mathfrak{P}}(\nu_{l},\mu_{A_{l}})<\varepsilon$ for
all $l$, a compact set $\mathfrak{K}\subseteq\mathfrak{P}$, and measurable
functions $\tau_{l}:X\rightarrow\mathbb{N}_{0}$ such that
(\ref{Eq_DontWaitTooLong2})-(\ref{Eq_ControlShortReturns}) hold. Abbreviate
$\theta_{l}:=\nu_{l}(A_{l}^{\circ})$, so that
\begin{equation}
\nu_{l}=(1-\theta_{l})\nu_{l}^{\bullet}+\theta_{l}\nu_{l}^{\circ}\text{ \quad
for }l\geq1\text{.} \label{Eq_hdhjkahfsakdfSJHFFDH}%
\end{equation}
Assumptions (\ref{Eq_DontWaitTooLong2}) to
(\ref{Eq_TheFabulousCompactnessCondition2}) allow us to apply Theorem
\ref{T_ProcessesAsymptoticToHittingTimeProc} using the sequence $(\nu
_{l}^{\circ})$ of measures, to see that $D_{[0,\infty]^{\mathbb{N}_{0}}%
}(\mathrm{law}_{\nu_{l}^{\circ}}(R_{l}),\mathrm{law}_{\mu}(R_{l}%
))\rightarrow0$. Consequently,
\begin{equation}
R_{l}\overset{\nu_{l}^{\circ}}{\Longrightarrow}\Phi\text{ \quad as
}l\rightarrow\infty\text{.} \label{Eq_handiballl}%
\end{equation}
To analyse the asymptotic distribution of $R_{l}$ under $\nu_{l}^{\bullet}$ we
observe first that due to (\ref{Eq_ManyHappyReturns}) the assumption
(\ref{Eq_DontWaitTooLong2}) implies that also
\begin{equation}
\mu(A_{l})\,\varphi_{A_{l}}\overset{\nu_{l}^{\bullet}}{\longrightarrow}%
0\quad\text{as }l\rightarrow\infty. \label{Eq_hjvchgagggggg}%
\end{equation}
On the other hand, $\mathrm{law}_{\nu_{l}^{\bullet}}(R_{l}\circ T_{A_{l}%
})=\mathrm{law}_{\overline{\nu}_{l}^{\bullet}}(R_{l})$, where $\overline{\nu
}_{l}^{\bullet}:=(T_{A_{l}})_{\ast}\nu_{l}^{\bullet}$. In view of
(\ref{Eq_ControlShortReturns}) and (\ref{Eq_ererrewrerwrertrerwerq}), we thus
have $D_{[0,\infty]^{\mathbb{N}_{0}}}(\mathrm{law}_{\nu_{l}^{\bullet}}%
(R_{l}\circ T_{A_{l}}),\mathrm{law}_{\mu_{A_{l}}}(R_{l}))=D_{[0,\infty
]^{\mathbb{N}_{0}}}(\mathrm{law}_{\overline{\nu}_{l}^{\bullet}}(R_{l}%
),\mathrm{law}_{\mu_{A_{l}}}(R_{l}))\rightarrow0$. Recalling
(\ref{Eq_AGoodSubSequenceHere}) and the fact that $\widetilde{\Phi}$ is
stationary (see (\ref{Eq_BasicPropsAsyRetProc})), this shows that
\begin{equation}
R_{l}\circ T_{A_{l}}\overset{\nu_{l}^{\bullet}}{\Longrightarrow}%
\mathbf{\sigma}\widetilde{\Phi}\text{ \quad as }l\rightarrow\infty\text{,}
\label{Eq_jnjnjnjjjjjjjj}%
\end{equation}
where $\mathbf{\sigma}\widetilde{\Phi}:=(\widetilde{\varphi}^{(1)}%
,\widetilde{\varphi}^{(2)},\ldots)$ is the shifted version of $\widetilde
{\Phi}=(\widetilde{\varphi}^{(0)},\widetilde{\varphi}^{(1)},\ldots)$. Since
the limit in (\ref{Eq_hjvchgagggggg}) is constant, and hence independent of
all random variables, we can combine (\ref{Eq_hjvchgagggggg}) and
(\ref{Eq_jnjnjnjjjjjjjj}) to obtain
\begin{equation}
R_{l}=(\mu(A_{l})\,\varphi_{A_{l}},R_{l}\circ T_{A_{l}})\overset{\nu
_{l}^{\bullet}}{\Longrightarrow}(0,\mathbf{\sigma}\widetilde{\Phi
})=(0,\widetilde{\varphi}^{(1)},\widetilde{\varphi}^{(2)},\ldots)\text{ \quad
as }l\rightarrow\infty\text{.} \label{Eqwassssssaballl}%
\end{equation}
Going back to (\ref{Eq_hdhjkahfsakdfSJHFFDH}) we can employ
(\ref{Eq_handiballl}) and (\ref{Eqwassssssaballl}) to see that in
$\mathfrak{M}([0,\infty]^{\mathbb{N}_{0}})$,
\begin{align}
\mathrm{law}_{\nu_{l}}(R_{l})  &  =(1-\theta_{l})\,\mathrm{law}_{\nu
_{l}^{\bullet}}(R_{l})+\theta_{l}\,\mathrm{law}_{\nu_{l}^{\circ}}%
(R_{l})\nonumber\\
&  \rightarrow(1-\theta)\,\mathrm{law}(0,\mathbf{\sigma}\widetilde{\Phi
})+\theta\,\mathrm{law}(\Phi)\text{ \quad as }l\rightarrow\infty\text{.}
\label{Eq_okokokokokghgiuuuu2}%
\end{align}
\textbf{(iii)} On the other hand, (\ref{Eq_ererrewrerwrertrerwerq}) guarantees
that $D_{[0,\infty]^{\mathbb{N}_{0}}}(\mathrm{law}_{\nu_{l}}(R_{l}%
),\mathrm{law}_{\mu_{A_{l}}}(R_{l}))<\varepsilon$ for all $l$, and hence by
(\ref{Eq_AGoodSubSequenceHere}) that $D_{[0,\infty]^{\mathbb{N}_{0}}%
}(\mathrm{law}_{\nu_{l}}(R_{l}),\mathrm{law}(\widetilde{\Phi}))\leq
\varepsilon$. Together with (\ref{Eq_okokokokokghgiuuuu2}) this proves that
\begin{equation}
D_{[0,\infty]^{\mathbb{N}_{0}}}((1-\theta)\,\mathrm{law}(0,\mathbf{\sigma
}\widetilde{\Phi})+\theta\,\mathrm{law}(\Phi),\mathrm{law}(\widetilde{\Phi
}))\leq\varepsilon.
\end{equation}
But $\varepsilon>0$ was arbitrary, and therefore
\begin{equation}
\mathrm{law}(\widetilde{\Phi})=(1-\theta)\,\mathrm{law}(0,\mathbf{\sigma
}\widetilde{\Phi})+\theta\,\mathrm{law}(\Phi)\text{.}%
\end{equation}
For the finite dimensional distribution functions $\{F^{[d]}\}$ and
$\{\widetilde{F}^{[d]}\}$ of $\Phi$ and $\widetilde{\Phi}$ this means that for
all $d\geq0$ and $t_{0},t_{1},\ldots,t_{d}\geq0$,
\begin{equation}
\widetilde{F}^{[d+1]}(t_{0},t_{1},\ldots,t_{d})=(1-\theta)\,\widetilde
{F}^{[d]}(t_{1},\ldots,t_{d})+\theta F^{[d+1]}(t_{0},t_{1},\ldots
,t_{d})\text{.} \label{Eq_bxvcbvvvvvvvvvvaqaqaq}%
\end{equation}
However, because of (\ref{Eq_AGoodSubSequenceHere}), $\{F^{[d]}\}$ and
$\{\widetilde{F}^{[d]}\}$ are related to each other as in
(\ref{T_HittingTimeVsReturnTimeProc}) of Theorem
\ref{T_HittingTimeVsReturnTimeProc}. Together with
(\ref{Eq_bxvcbvvvvvvvvvvaqaqaq}) the latter shows that $\{\widetilde{F}%
^{[d]}\}$ satisfies condition (\ref{Eq_CharHigherDimCompoundDF}), and
(\ref{Eq_TheBigPrize}) follows by Proposition \ref{P_HandsUp} b).
\end{proof}

%

\vspace{0.3cm}%

\section{Proofs for local processes \label{Sec_PfInternalStates}}

Throughout this section, the fixed compact metric space $(\mathfrak{Z}%
,d_{\mathfrak{Z}})$ is the state space for local observables.\newline\newline%

\noindent
\textbf{The general asymptotic local process.} To prepare the proof of Theorem
\ref{T_PrescribingPSI}, we first establish an approximation result.

\begin{proposition}
[\textbf{Approximating }$d$\textbf{-dimensional marginals of a stationary
sequence}]\label{P_ApproxMarginals}Let $T$ be an ergodic measure preserving
map on the nonatomic probability space $(X,\mathcal{A},\mu)$, and let
$\widehat{\Psi}$ be an $\mathfrak{Z}$-valued stationary sequence which only
assumes finitely many different values. Then, for any $d\geq1$ and
$\varepsilon>0$, there is some measurable $\psi:X\rightarrow\mathfrak{Z}$ such
that the sequence $\Psi:=(\psi,\psi\circ T,\ldots)$ satisfies
\begin{equation}
D_{\mathfrak{Z}^{d}}(\pi_{\ast}^{d}(\mathrm{law}_{\mu}(\Psi)),\pi_{\ast}%
^{d}(\mathrm{law}(\widehat{\Psi})))<\varepsilon\text{.}%
\end{equation}

\end{proposition}

\begin{proof}
\textbf{(i)} Write $\widehat{\Psi}=(\widehat{\psi}^{(0)},\widehat{\psi}%
^{(1)},\ldots)$ and let $F\subseteq\mathfrak{Z}$ be a finite set such that
$\widehat{\psi}^{(0)}\in F$ a.s. Fix $d$ and $\varepsilon$, and pick one
particular element $y_{\ast}\in F$. It suffices to show that we can construct
a local observable $\psi$ and an arbitrarily large subset $Y$ of $X$ such that
the $d$-dimensional marginal of $\Psi$, when conditioned on $Y$, coincides
with the marginal of $\widehat{\Psi}$. That is, we prove that for every
$\delta>0$ there is some $Y\in\mathcal{A}$ with $\mu(Y^{c})<\delta$, and a
measurable $\psi:X\rightarrow\mathfrak{Z}$ such that
\begin{equation}
\pi_{\ast}^{d}(\mathrm{law}_{\mu_{Y}}(\Psi))=\pi_{\ast}^{d}(\mathrm{law}%
(\widehat{\Psi}))\text{.} \label{Eq_awsawasa}%
\end{equation}
\textbf{(ii)} Apply the classical Rokhlin Lemma (as in Lemma 7.4 of
\cite{Z11}) to obtain a Rokhlin tower $(X_{i})_{i=0}^{I}$ of height
$I>2d/\delta$ and with $\mu(X\setminus%
{\textstyle\bigcup\nolimits_{i=0}^{I}}
X_{i})<\delta/2$. This means that the $X_{i}$ are pairwise disjoint and
$X_{i}=T^{-(I-i)}X_{I}$ for $i\in\{0,\ldots,I\}$. Conditioning on the top
level $X_{I}$ of the tower we obtain the probability space $(X_{I},X_{I}%
\cap\mathcal{A},\mu_{X_{I}})$. Being nonatomic, it admits a partition into
measurable sets,
\[
X_{I}=%
{\textstyle\bigcup\nolimits_{(y_{0},\ldots,y_{I})\in F^{I+1}}}
X_{I}(y_{0},\ldots,y_{I})\text{ \quad(disjoint),}%
\]
with $\mu_{X_{I}}(X_{I}(y_{0},\ldots,y_{I}))=\Pr[(\widehat{\psi}^{(0)}%
,\ldots,\widehat{\psi}^{(I)})=(y_{0},\ldots,y_{I})]$. We define partitions of
the other levels $X_{i}$, $i\in\{0,\ldots,I-1\}$,
\[
X_{i}=%
{\textstyle\bigcup\nolimits_{(y_{0},\ldots,y_{I})\in F^{I+1}}}
X_{i}(y_{0},\ldots,y_{I})\text{ \quad(disjoint),}%
\]
by setting $X_{i}(y_{0},\ldots,y_{I}):=T^{-1}X_{i+1}(y_{0},\ldots
,y_{I})=\ldots=T^{-(I-i)}X_{I}(y_{0},\ldots,y_{I})$. Finally, define a
measurable function $\psi:X\rightarrow\mathfrak{Z}$ through
\begin{equation}
\psi:=\left\{
\begin{array}
[c]{cc}%
y_{i} & \text{on }X_{i}(y_{0},\ldots,y_{I})\text{, }0\leq i\leq I\text{,}\\
y_{\ast} & \text{otherwise.}%
\end{array}
\right.
\end{equation}
Then, for any $(y_{0},\ldots,y_{I})\in F^{I+1}$ and $j\in\{0,1,\ldots I-i\}$,
\begin{equation}
\psi\circ T^{j}=y_{i+j}\text{ \quad on }X_{i}(y_{0},\ldots,y_{I})\text{.}
\label{Eq_ajajaajkkssweq}%
\end{equation}
\textbf{(iii)} As a consequence of (\ref{Eq_ajajaajkkssweq}) we get, for
$i\in\{0,\ldots,I-d+1\}$, $(z_{0},\ldots,z_{d-1})\in F^{d}$, a decomposition
\begin{multline*}
X_{i}\cap\{(\psi,\ldots,\psi\circ T^{d-1})=(z_{0},\ldots,z_{d-1})\}\\
=%
{\displaystyle\bigcup\nolimits_{\substack{(y_{0},\ldots,y_{I})\in
F^{I+1}:\\(y_{i},\ldots,y_{i+d-1})=(z_{0},\ldots,z_{d-1})}}}
X_{i}(y_{0},\ldots,y_{I})\\
=%
{\displaystyle\bigcup\nolimits_{\substack{(y_{0},\ldots,y_{I})\in
F^{I+1}:\\(y_{i},\ldots,y_{i+d-1})=(z_{0},\ldots,z_{d-1})}}}
T^{-(I-i)}X_{I}(y_{0},\ldots,y_{I})\text{.}%
\end{multline*}
Therefore, as $T$ preserves $\mu$ and $X_{i}=T^{-(I-i)}X_{I}$, we se that
\begin{align*}
&  \mu_{X_{i}}((\psi,\ldots,\psi\circ T^{d-1})=(z_{0},\ldots,z_{d-1}))=\\
&  =%
{\displaystyle\sum\nolimits_{\substack{(y_{0},\ldots,y_{I})\in F^{I+1}%
:\\(y_{i},\ldots,y_{i+d-1})=(z_{0},\ldots,z_{d-1})}}}
\mu_{X_{i}}(T^{-(I-i)}X_{I}(y_{0},\ldots,y_{I}))\\
&  =%
{\displaystyle\sum\nolimits_{\substack{(y_{0},\ldots,y_{I})\in F^{I+1}%
:\\(y_{i},\ldots,y_{i+d-1})=(z_{0},\ldots,z_{d-1})}}}
\mu_{X_{I}}(X_{I}(y_{0},\ldots,y_{I}))\\
&  =%
{\displaystyle\sum\nolimits_{\substack{(y_{0},\ldots,y_{I})\in F^{I+1}%
:\\(y_{i},\ldots,y_{i+d-1})=(z_{0},\ldots,z_{d-1})}}}
\Pr[(\widehat{\psi}^{(0)},\ldots,\widehat{\psi}^{(I)})=(y_{0},\ldots,y_{I})]\\
&  =\Pr[(\widehat{\psi}^{(i)},\ldots,\widehat{\psi}^{(i+d-1)})=(z_{0}%
,\ldots,z_{d-1})]\\
&  =\Pr[(\widehat{\psi}^{(0)},\ldots,\widehat{\psi}^{(d-1)})=(z_{0}%
,\ldots,z_{d-1})]\text{,}%
\end{align*}
meaning that
\begin{equation}
\mathrm{law}_{\mu_{X_{i}}}((\psi,\ldots,\psi\circ T^{d-1}))=\mathrm{law}%
((\widehat{\psi}^{(0)},\ldots,\widehat{\psi}^{(d-1)}))\text{ \quad for }0\leq
i\leq I-d+1\text{.}%
\end{equation}
Hence, taking $Y:=%
{\textstyle\bigcup\nolimits_{i=0}^{I-d+1}}
X_{i}$ we have $\mu(Y^{c})=\mu(%
{\textstyle\bigcup\nolimits_{i=I-d+2}^{I}}
X_{i})+\mu(X\setminus%
{\textstyle\bigcup\nolimits_{i=0}^{I}}
X_{i})<d/I+\delta/2<\delta$ and $\mathrm{law}_{\mu_{Y}}((\psi,\ldots,\psi\circ
T^{d-1}))=\mathrm{law}((\widehat{\psi}^{(0)},\ldots,\widehat{\psi}^{(d-1)}))$,
as required in (\ref{Eq_awsawasa}) above.
\end{proof}

%

\vspace{0.3cm}%

We can now turn to the%

\vspace{0.3cm}%

\begin{proof}
[\textbf{Proof of Theorem \ref{T_PrescribingPSI}.}]Let $(A_{l})$ and
$\widetilde{\Psi}=(\widetilde{\psi}^{(0)},\widetilde{\psi}^{(1)},\ldots)$ be
given. For every $l\geq1$ there is some finite set $F_{l}\subseteq
\mathfrak{Z}$ and a Borel measurable map $\theta_{l}:\mathfrak{Z}\rightarrow
F_{l}$ such that $d_{\mathfrak{Z}}(\mathrm{Id}_{\mathfrak{Z}},\theta_{l})<1/l$
on $\mathfrak{Z}$. Setting $\widehat{\Psi}_{l}:=(\theta_{l}\circ
\widetilde{\psi}^{(0)},\theta_{l}\circ\widetilde{\psi}^{(1)},\ldots)$ we
obtain a stationary sequence in $\mathfrak{Z}$ which only assumes finitely
many values and satisfies $d_{\mathfrak{Z}}(\widehat{\Psi}_{l},\widetilde
{\Psi})<1/l$ on the underlying probability space. Therefore,
\begin{equation}
D_{\mathfrak{Z}^{\mathbb{N}_{0}}}(\mathrm{law}(\widehat{\Psi}_{l}%
),\mathrm{law}(\widetilde{\Psi}))\longrightarrow0\text{ \quad as }%
l\rightarrow\infty\text{.} \label{Eq_kmkhnmhknjjjjjjjj}%
\end{equation}
Due to (\ref{Eq_kmkhnmhknjjjjjjjj}) it suffices to construct $\psi_{A_{l}}$,
$l\geq1$, such that the corresponding local processes $\widetilde{\Psi}%
_{A_{l}}$ approximate the $\widehat{\Psi}_{l}$ and satisfy
\begin{equation}
D_{\mathfrak{Z}^{\mathbb{N}_{0}}}(\mathrm{law}_{\mu_{A_{l}}}(\widetilde{\Psi
}_{A_{l}}),\mathrm{law}(\widehat{\Psi}_{l}))\longrightarrow0\text{ \quad as
}l\rightarrow\infty\text{,}%
\end{equation}
or, equivalently, that for every $d\geq1$,%
\begin{equation}
D_{\mathfrak{Z}^{d}}(\pi_{\ast}^{d}(\mathrm{law}_{\mu_{A_{l}}}(\widetilde
{\Psi}_{A_{l}})),\pi_{\ast}^{d}(\mathrm{law}(\widehat{\Psi}_{l}%
)))\longrightarrow0\text{ \quad as }l\rightarrow\infty\text{.}
\label{Eq_vmvmvmvm}%
\end{equation}
For each $l\geq1$ apply Proposition \ref{P_ApproxMarginals} to $(A_{l}%
,A_{l}\cap\mathcal{A},\mu_{A_{l}},T_{A_{l}})$, $d:=l$, $\varepsilon:=1/l$, and
$\widehat{\Psi}:=\widehat{\Psi}_{l}$ to obtain an local observable
$\psi_{A_{l}}:A_{l}\rightarrow\mathfrak{Z}$ for $A_{l}$ for which
\begin{equation}
D_{\mathfrak{Z}^{d}}(\pi_{\ast}^{l}(\mathrm{law}_{\mu_{A_{l}}}(\widetilde
{\Psi}_{A_{l}})),\pi_{\ast}^{l}(\mathrm{law}(\widehat{\Psi}_{l})))<1/l\text{.}%
\end{equation}
Then (\ref{Eq_vmvmvmvm}) follows since $D_{\mathfrak{Z}^{d}}(\pi_{\ast}%
^{d}(\mathrm{law}(\Psi^{\prime\prime})),\pi_{\ast}^{d}(\mathrm{law}%
(\Psi^{\prime})))$ is non-decreasing in $d$ for all random sequences
$\Psi^{\prime},\Psi^{\prime\prime}$ in $\mathfrak{Z}$.
\end{proof}

%

\vspace{0.3cm}%
%

\noindent
\textbf{Towards specific limit processes.} To get started, we record some
basic properties of local processes.

\begin{proposition}
[\textbf{Asymptotic invariance and admissible delays for }$(\Psi_{A_{l}})$%
]\label{P_AsyInvarAdmissDelayInternal}Let $T$ be a measure-preserving map on
the probability space $(X,\mathcal{A},\mu)$. \newline\textbf{a)} For every set
$A\in\mathcal{A}$, any local process $\Psi_{A}$ on $A$, and any $m\geq0$,
\begin{equation}
\Psi_{A}\circ T^{m}=\Psi_{A}\text{ \quad on }\{\varphi_{A}>m\}\text{. }
\label{Eq_bxvcyvvcyvjs}%
\end{equation}
\textbf{b)} Suppose that $(A_{l})_{l\geq1}$ is a sequence of asymptotically
rare events, and $(\psi_{A_{l}})_{l\geq1}$ a sequence of local observables for
the $A_{l}$, with corresponding local processes $\Psi_{A_{l}}$. Set
$R_{l}:=\Psi_{A_{l}}:X\rightarrow\mathfrak{Z}^{\mathbb{N}}$. Then $(R_{l}%
)$\emph{\ }is asymptotically $T$-invariant in measure.\newline\textbf{c)} Let
$(\nu_{l})$ be a sequence in $\mathfrak{P}$, and let the measurable maps
$\tau_{l}:X\rightarrow\mathbb{N}_{0}$ satisfy
\begin{equation}
\nu_{l}\left(  \tau_{l}<\varphi_{A_{l}}\right)  \longrightarrow1\text{ \quad
as }l\rightarrow\infty\text{.}%
\end{equation}
Then $(\tau_{l})$ is an admissible delay sequence for $(R_{l})$ and $(\nu
_{l})$.
\end{proposition}

\begin{proof}
Statement a) is immediate from the fact that $T_{A}\circ T^{m}=T_{A}$ on
$\{\varphi_{A}>m\}$\ for every $m\geq0$. Next, $(A_{l})$ being asymptotically
rare means that $\mu(\varphi_{A_{l}}>m)\rightarrow1$ as $l\rightarrow\infty$.
In particular, $\mu(d_{\mathfrak{Z}^{\mathbb{N}}}(R_{l}\circ T,R_{l}%
)>0)\leq\mu(\Psi_{A_{l}}\circ T\neq\Psi_{A_{l}})\leq\mu(\varphi_{A_{l}%
}=1)\rightarrow0$, proving b). Turning to c) we note that
(\ref{Eq_bxvcyvvcyvjs}) entails
\[
\Psi_{A_{l}}\circ T^{\tau_{l}}=\Psi_{A_{l}}\text{ \quad on }\{\varphi_{A_{l}%
}>\tau_{l}\}\text{,}%
\]
whence $\nu_{l}(d_{\mathfrak{Z}^{\mathbb{N}}}(R_{l}\circ T^{\tau_{l}}%
,R_{l})>0)\leq\nu_{l}(\varphi_{A_{l}}\leq\tau_{l})\rightarrow0$, validating
the sufficient condition (\ref{Eq_vbvbvbvbvbvbvbvvbvbvbvbvffff}).
\end{proof}

%

\vspace{0.3cm}%

Statement b) shows that Proposition \ref{P_SDCInternal} is a special case of
Theorem \ref{T_MyOldStrongDistrCgeThm} with $R_{l}:=\Psi_{A_{l}}$. We can now
supply the easy

\begin{proof}
[\textbf{Proof of Theorem \ref{T_AsyIntStateProc}}]Set $R_{l}:=\Psi_{A_{l}}$,
$l\geq1$. By Proposition \ref{P_AsyInvarAdmissDelayInternal} b) and c) and
condition (\ref{Eq_hghghghgghghghghdhsf}), $(R_{l})$ is asymptotically
$T$-invariant in measure, and $(\tau_{l})$ is an admissible delay sequence for
$(R_{l})$ and $(\nu_{l})$. Now (\ref{Eq_cxdcdsewvxvxwfrvxsf}) allows us to
apply Proposition \ref{P_AsyInvarSeqNuLVsMu}.
\end{proof}

%

\vspace{0.3cm}%

Next we turn to the%

\vspace{0.3cm}%

\begin{proof}
[\textbf{Proof of Theorem \ref{T_AsyIntStateProcIndep}}]\textbf{(i)} For every
$k\geq1$ choose a sequence $(\nu_{l}^{(k)})_{l\geq1}$ in $\mathfrak{P}$ which
satisfies the assumptions on $(\nu_{l})_{l\geq1}$ in the final paragraph of
the theorem with $d_{\mathfrak{P}}(\nu_{l}^{(k)},\mu_{A_{l}})<1/k$ for all
$l$. By compactness of $(\mathfrak{M}(\mathfrak{Z}^{\mathbb{N}_{0}%
}),D_{\mathfrak{Z}^{\mathbb{N}_{0}}})$ and a diagonalization argument we may
assume w.l.o.g. that we work with a subsequence along which we have
distributional convergence for all measures involved. Specifically, assume
that there are random sequences $\widetilde{\Psi}=(\psi^{(0)},\psi
^{(1)},\ldots)$ and $\widetilde{\Psi}^{(k)}=(\psi^{(0,k)},\psi^{(1,k)}%
,\ldots)$, $k\geq1$, in $\mathfrak{Z}$ such that
\begin{equation}
\widetilde{\Psi}_{A_{l}}\overset{\mu_{A_{l}}}{\Longrightarrow}\widetilde{\Psi
}\text{ \quad and \quad}\widetilde{\Psi}_{A_{l}}\overset{\nu_{l}^{(k)}%
}{\Longrightarrow}\widetilde{\Psi}^{(k)}\text{ for all }k\geq1\quad\text{as
}l\rightarrow\infty\text{.} \label{Eq_wawawawawa}%
\end{equation}
Due to (\ref{Eq_uoljioljo1}) we know that $\widetilde{\Psi}$ is stationary,
obviously with $\mathrm{law}(\psi^{(0)})=\mathrm{law}(\psi)$. The main point
is to show that $\widetilde{\Psi}$ is in fact iid.\newline

To this end, write $\mathbf{\sigma}\widetilde{\Psi}:=(\psi^{(1)},\psi
^{(2)},\ldots)$ for the shifted version of $\widetilde{\Psi}$, and regard
$\widetilde{\Psi}$ as the random element $(\psi^{(0)},\mathbf{\sigma
}\widetilde{\Psi})$ of $\mathfrak{Z}\times\mathfrak{Z}^{\mathbb{N}}$. Due to
(\ref{Eq_wawawawawa}) we have
\begin{equation}
(\psi_{A_{l}},\Psi_{A_{l}})\overset{\mu_{A_{l}}}{\Longrightarrow}(\psi
^{(0)},\mathbf{\sigma}\widetilde{\Psi})\text{ \quad}\quad\text{as
}l\rightarrow\infty\text{.}%
\end{equation}
Since $\widetilde{\Psi}$ is stationary, we know it is in fact iid as soon as
\begin{equation}
(\psi^{(0)},\mathbf{\sigma}\widetilde{\Psi})\text{ is an independent pair.}
\label{Eq_FirstEntraIndependent}%
\end{equation}
\newline\textbf{(ii)} For every $k\geq1$ we can employ Theorem
\ref{T_AsyIndepForAsyInvarSequences} with $\mathfrak{E}:=\mathfrak{Z}$,
$\mathfrak{E}^{\prime}:=\mathfrak{Z}^{\mathbb{N}}$, $\nu_{l}:=\nu_{l}^{(k)}$,
$R_{l}:=\psi_{A_{l}}$, $R_{l}^{\prime}:=\Psi_{A_{l}}$, and $\mathcal{B}%
_{\mathfrak{E}}^{\pi}:=\mathcal{B}_{\mathfrak{Z}}^{\pi}$. Indeed, by
Proposition \ref{P_AsyInvarAdmissDelayInternal} b), the sequence
$(R_{l}^{\prime})$ is asymptotically $T$-invariant in measure.

Fix $k$, take any $F\in\mathcal{B}_{\mathfrak{Z}}^{\pi}$ with $\Pr[\psi\in
F]>0$, and pick $(\nu_{l,F})$, $(\tau_{l,F})$ and $\mathfrak{K}_{F}$ for
$(\nu_{l}^{(k)})$ as in the assumption of Theorem \ref{T_AsyIntStateProcIndep}%
. Observe that via Proposition \ref{P_AsyInvarAdmissDelayInternal} c) our
assumption (\ref{Eq_uiuiuiuiui}) ensures that $(\tau_{l,F})$ is always an
admissible delay sequence for $(R_{l}^{\prime})$ and the $\nu_{l,F}$. Thus,
Theorem \ref{T_AsyIndepForAsyInvarSequences} shows that for every $k\geq1$,%
\begin{equation}
(\psi^{(0,k)},\mathbf{\sigma}\widetilde{\Psi}^{(k)})\text{ is an independent
pair.} \label{Eq_IndepForK}%
\end{equation}
But since $d_{\mathfrak{P}}(\nu_{l}^{(k)},\mu_{A_{l}})<1/k$ for all $l$, it is
clear that
\begin{equation}
\widetilde{\Psi}^{(k)}=(\psi^{(0,k)},\mathbf{\sigma}\widetilde{\Psi}%
^{(k)})\Longrightarrow(\psi^{(0)},\mathbf{\sigma}\widetilde{\Psi})\text{
\quad}\quad\text{as }k\rightarrow\infty\text{.}%
\end{equation}
Together with (\ref{Eq_IndepForK}) this proves (\ref{Eq_FirstEntraIndependent}%
).\newline\newline\textbf{(iii)} The above shows that $\widetilde{\Psi}%
_{A_{l}}\overset{\mu_{A_{l}}}{\Longrightarrow}\Psi^{\ast}$, and hence also
$\Psi_{A_{l}}=\mathbf{\sigma}\widetilde{\Psi}_{A_{l}}\overset{\mu_{A_{l}}%
}{\Longrightarrow}\mathbf{\sigma}\Psi^{\ast}\overset{d}{=}\Psi^{\ast}$. From
(\ref{Eq_wawawawawa}) we see that $\Psi_{A_{l}}\overset{\nu_{l}^{(k)}%
}{\Longrightarrow}\mathbf{\sigma}\widetilde{\Psi}^{(k)}$ for each $k$, and
applying Theorem \ref{T_AsyIntStateProc} with $\nu_{l}:=\nu_{l}^{(k)}=(\nu
_{l}^{(k)})_{\{\psi_{A_{l}}\in\mathfrak{Z}\}}$ and $\tau_{l}:=\tau
_{l,\mathfrak{Z}}$, we conclude that $\Psi_{A_{l}}\overset{\mu}%
{\Longrightarrow}\mathbf{\sigma}\widetilde{\Psi}^{(k)}$ for every $k$. But
then all these limit pocesses have the same law, and hence also the same law
as their distributional limit $\mathbf{\sigma}\widetilde{\Psi}\overset{d}%
{=}\Psi^{\ast}$. This proves $\Psi_{A_{l}}\overset{\mu}{\Longrightarrow}%
\Psi^{\ast}$.
\end{proof}

%

\vspace{0.3cm}%
%

\noindent
\textbf{Robustness.} To conclude this section, we provide a

\begin{proof}
[\textbf{Proof of Theorem \ref{Thm_RobustLocProc}.}]We asume w.l.o.g. that
$A_{l}^{\prime}\subseteq A_{l}$ for all $l$. (Otherwise we can apply this
partial result to compare either of $(\psi_{A_{l}})$ and $(\psi_{A_{l}%
^{\prime}}^{\prime})$ to $(\psi_{A_{l}}\mid_{A_{l}\cap A_{l}^{\prime}})$.) Our
goal is to prove that
\begin{equation}
D_{\mathfrak{Z}^{\mathbb{N}_{0}}}(\mathrm{law}_{\mu_{A_{l}}}(\widetilde{\Psi
}_{A_{l}}),\mathrm{law}_{\mu_{A_{l}^{\prime}}}(\widetilde{\Psi}_{A_{l}%
^{\prime}}^{\prime}))\longrightarrow0\text{ \qquad as }l\rightarrow
\infty\text{.} \label{Eq_RobustLocProcMetric}%
\end{equation}
\newline\textbf{(i)} We first note that our asumptions guarantee
\begin{equation}
D_{\mathfrak{Z}}(\mathrm{law}_{\mu_{A_{l}}}(\psi_{A_{l}}),\mathrm{law}%
_{\mu_{A_{l}^{\prime}}}(\psi_{A_{l}^{\prime}}^{\prime}))\longrightarrow0\text{
\qquad as }l\rightarrow\infty\text{.} \label{Eq_idimRobust}%
\end{equation}
For this, decompose $\mathrm{law}_{\mu_{A_{l}}}(\psi_{A_{l}})=\mu_{A_{l}%
}(A_{l}^{\prime})\,\mathrm{law}_{\mu_{A_{l}^{\prime}}}(\psi_{A_{l}}%
)+\mu_{A_{l}}(A_{l}\setminus A_{l}^{\prime})\,\mathrm{law}_{\mu_{A_{l}%
\setminus A_{l}^{\prime}}}(\psi_{A_{l}})$ and $\mathrm{law}_{\mu
_{A_{l}^{\prime}}}(\psi_{A_{l}^{\prime}}^{\prime})=\mu_{A_{l}}(A_{l}^{\prime
})\,\mathrm{law}_{\mu_{A_{l}^{\prime}}}(\psi_{A_{l}^{\prime}}^{\prime}%
)+\mu_{A_{l}}(A_{l}\setminus A_{l}^{\prime})\,\mathrm{law}_{\mu_{A_{l}%
^{\prime}}}(\psi_{A_{l}^{\prime}}^{\prime})$. By (\ref{Eq_Tree2}),
\[
D_{\mathfrak{Z}}(\mathrm{law}_{\mu_{A_{l}}}(\psi_{A_{l}}),\mathrm{law}%
_{\mu_{A_{l}^{\prime}}}(\psi_{A_{l}^{\prime}}^{\prime}))\leq D_{\mathfrak{Z}%
}(\mathrm{law}_{\mu_{A_{l}^{\prime}}}(\psi_{A_{l}}),\mathrm{law}_{\mu
_{A_{l}^{\prime}}}(\psi_{A_{l}^{\prime}}^{\prime}))+\mu_{A_{l}}(A_{l}\setminus
A_{l}^{\prime})\text{,}%
\]
and this bound tends to zero as $l\rightarrow\infty$, recall (\ref{Eq_Tree1}%
).\newline\newline\textbf{(ii)} To prepare for the process version
(\ref{Eq_RobustLocProcMetric}) of (\ref{Eq_idimRobust}), we show, for every
$j\geq0$,
\begin{equation}
\mu_{A_{l}^{\prime}}(T_{A_{l}}^{j}\neq T_{A_{l}^{\prime}}^{j})\longrightarrow
0\text{\qquad as }l\rightarrow\infty\text{.} \label{Eq_SameIteratesJSteps}%
\end{equation}
Indeed, whenever $A^{\prime}\subseteq A$, then $A^{\prime}\cap\{T_{A}%
^{j}=T_{A^{\prime}}^{j}\}\supseteq%
{\textstyle\bigcap\nolimits_{i=0}^{j}}
T_{A}^{-i}A^{\prime}$. Therefore,
\begin{align*}
\mu_{A}(A^{\prime}\cap\{T_{A}^{j}\neq T_{A^{\prime}}^{j}\})  &  \leq\mu
_{A}\left(
{\textstyle\bigcup\nolimits_{i=0}^{j}}
T_{A}^{-i}(A\setminus A^{\prime})\right) \\
&  \leq%
{\textstyle\sum\limits_{i=0}^{j}}
\mu_{A}\left(  T_{A}^{-i}(A\setminus A^{\prime})\right)  =(j+1)\mu_{A}\left(
A\setminus A^{\prime}\right)  \text{,}%
\end{align*}
since $T_{A}$ preserves $\mu_{A}$. Applying this to $A_{l}^{\prime}\subseteq
A_{l}$ yields (\ref{Eq_SameIteratesJSteps}), as $\mu_{A_{l}}\left(
A_{l}\setminus A_{l}^{\prime}\right)  \rightarrow0$.

We use this to check that for every $j\geq0$,%
\begin{equation}
d_{\mathfrak{Z}}(\psi_{A_{l}}\circ T_{A_{l}}^{j},\psi_{A_{l}^{\prime}}%
^{\prime}\circ T_{A^{\prime}}^{j})\overset{\mu_{A_{l}^{\prime}}}%
{\longrightarrow}0\text{\qquad as }l\rightarrow\infty\text{.}
\label{Eq_RobyTobyMetricInMeasure}%
\end{equation}
Take any $\eta>0$, then $A_{l}^{\prime}\cap\{d_{\mathfrak{Z}}(\psi_{A_{l}%
}\circ T_{A_{l}}^{j},\psi_{A_{l}^{\prime}}^{\prime}\circ T_{A^{\prime}}%
^{j})\geq\eta\}\subseteq(A_{l}^{\prime}\cap\{T_{A_{l}}^{j}\neq T_{A_{l}%
^{\prime}}^{j}\})\cup(A_{l}^{\prime}\cap T_{A^{\prime}}^{-j}\{d_{\mathfrak{Z}%
}(\psi_{A_{l}},\psi_{A_{l}^{\prime}}^{\prime})\geq\eta\})$. As a consequence
of this and the fact that $T_{A_{l}^{\prime}}$ preserves $\mu_{A_{l}^{\prime}%
}$ we find that
\[
\mu_{A_{l}^{\prime}}(d_{\mathfrak{Z}}(\psi_{A_{l}}\circ T_{A_{l}}^{j}%
,\psi_{A_{l}^{\prime}}^{\prime}\circ T_{A^{\prime}}^{j})\geq\eta)\leq
\mu_{A_{l}^{\prime}}(T_{A_{l}}^{j}\neq T_{A_{l}^{\prime}}^{j})+\mu
_{A_{l}^{\prime}}(d_{\mathfrak{Z}}(\psi_{A_{l}},\psi_{A_{l}^{\prime}}^{\prime
})\geq\eta)\text{,}%
\]
which tends to zero due to (\ref{Eq_SameIteratesJSteps}) and assumption
(\ref{Eq_LocalObsAsySameInMeasure}). This proves
(\ref{Eq_RobyTobyMetricInMeasure}).\newline\newline\textbf{(iii)} Now recall
that we can regard the local processes $\widetilde{\Psi}_{A_{l}}$ and
$\widetilde{\Psi}_{A_{l}^{\prime}}^{\prime}$ as local observables taking
values in $\mathfrak{Z}^{\mathbb{N}_{0}}$. Therefore our assertion
(\ref{Eq_RobustLocProcMetric}) will follow from the weaker version
(\ref{Eq_idimRobust}) of the present theorem which has already been
established in step (i), as soon as we validate
\begin{equation}
d_{\mathfrak{Z}^{\mathbb{N}_{0}}}(\widetilde{\Psi}_{A_{l}},\widetilde{\Psi
}_{A_{l}^{\prime}}^{\prime})\overset{\mu_{A_{l}^{\prime}}}{\longrightarrow
}0\text{ \qquad as }l\rightarrow\infty\text{.} \label{Eq_Checky}%
\end{equation}
Given $\varepsilon>0$ choose $J\geq1$ so large that $2^{-J}\mathrm{diam}%
(\mathfrak{Z})<\varepsilon/2$. Then, for all $l\geq1$,
\begin{align*}
d_{\mathfrak{Z}^{\mathbb{N}_{0}}}(\widetilde{\Psi}_{A_{l}},\widetilde{\Psi
}_{A_{l}^{\prime}}^{\prime})  &  =%
{\textstyle\sum_{j\geq0}}
2^{-(j+1)}d_{\mathfrak{Z}}(\psi_{A_{l}}\circ T_{A_{l}}^{j},\psi_{A_{l}%
^{\prime}}^{\prime}\circ T_{A^{\prime}}^{j})\\
&  \leq%
{\textstyle\sum_{j=0}^{J}}
d_{\mathfrak{Z}}(\psi_{A_{l}}\circ T_{A_{l}}^{j},\psi_{A_{l}^{\prime}}%
^{\prime}\circ T_{A^{\prime}}^{j})+\varepsilon/2\text{,}%
\end{align*}
which in view of (\ref{Eq_RobyTobyMetricInMeasure}) leads to (\ref{Eq_Checky})
and thus gives (\ref{Eq_RobustLocProcMetric}).
\end{proof}

%

\vspace{0.3cm}%

\section{Proofs for joint processes}

We start with the easy

\begin{proof}
[\textbf{Proof of Proposition \ref{P_SDCJoint}}]Since both $(\mu(A_{l}%
)\Phi_{A_{l}})_{l\geq1}$ and $(\Psi_{A_{l}})_{l\geq1}$ are asymptotically
$T$-invariant in measure as sequences in $[0,\infty]^{\mathbb{N}}$ and
$\mathfrak{Z}^{\mathbb{N}}$ respectively (Proposition \ref{P_AsyInvarHitProc}
b) and Proposition \ref{P_AsyInvarAdmissDelayInternal} b)), we immediately see
that $(\mu(A_{l})\Phi_{A_{l}},\Psi_{A_{l}})_{l\geq1}$ is asymptotically
$T$-invariant in measure as a sequence in $[0,\infty]^{\mathbb{N}}%
\times\mathfrak{Z}^{\mathbb{N}}$. Now use Theorem
\ref{T_MyOldStrongDistrCgeThm}.
\end{proof}

%

\vspace{0.3cm}%

Now compare the laws of joint processes under $\mu$ and under the $\mu_{A_{l}%
}$.

\begin{proof}
[\textbf{Proof of Theorem \ref{T_JointLimitProcMuVsMuA}}]\textbf{(i)} Due to
stationarity of $(\mu(A_{l})\Phi_{A_{l}},\widetilde{\Psi}_{A_{l}})$ under
$\mu_{A_{l}}$, it is easy to see that (\ref{Eq_jomiasanmimradldo}) is actually
equivalent to the formally weaker statement (obtained by forgetting about the
first entry of $\widetilde{\Psi}_{A_{l}}$)
\begin{equation}
R_{l}\overset{\mu_{A_{l}}}{\Longrightarrow}\,(\Phi_{\mathrm{Exp}},\Psi^{\ast
})\quad\text{as }l\rightarrow\infty\text{.}%
\end{equation}
where $R_{l}:=(\mu(A_{l})\Phi_{A_{l}},\Psi_{A_{l}}):X\rightarrow
\lbrack0,\infty]^{\mathbb{N}_{0}}\times\mathfrak{Z}^{\mathbb{N}}%
=:\mathfrak{E}$ (equipped with $d_{\mathfrak{E}}:=d_{[0,\infty]^{\mathbb{N}%
_{0}}}+d_{\mathfrak{Z}}$). The standard subsequence argument based on
compactness of $\mathfrak{E}$ shows that we can assume w.l.o.g. that there are
random elements $(\Phi,\Psi)$ and $(\overline{\Phi},\overline{\Psi})$ of
$\mathfrak{E}$, with $\Phi=(\varphi^{(i)})_{i\geq0}$, $\overline{\Phi
}=(\overline{\varphi}^{(i)})_{i\geq0}$, $\Psi=(\psi^{(i)})_{i\geq0}$, and
$\overline{\Psi}=(\overline{\psi}^{(i)})_{i\geq0}$, such that
\begin{equation}
R_{l}\overset{\mu}{\Longrightarrow}\,(\Phi,\Psi)\quad\text{and\quad}%
R_{l}\overset{\mu_{A_{l}}}{\Longrightarrow}\,(\overline{\Phi},\overline{\Psi
})\quad\text{as }l\rightarrow\infty\text{.} \label{Eq_cwwwasqs}%
\end{equation}
To prove the theorem, we now assume that
\begin{equation}
\text{one of }(\Phi,\Psi)\ \text{and }(\overline{\Phi},\overline{\Psi})\text{
has the law of }(\Phi_{\mathrm{Exp}},\Psi^{\ast})\text{.}%
\end{equation}
In view of (\ref{Eq_cwwwasqs}) and (\ref{Eq_mnhjnggnfjfggnjf1}) \&
(\ref{Eq_mnhjnggnfjfggnjf2}), Theorem \ref{T_AsyIntStateProc} then shows that
$\Psi\overset{d}{=}\overline{\Psi}\overset{d}{=}\Psi^{\ast}$. Similarly,
Theorem \ref{T_HittingTimeVsReturnTimeProc} and Proposition
\ref{P_CharacterizeExpos} together show that if one of $\Phi\ $and
$\overline{\Phi}$ has the distribution of $\Phi_{\mathrm{Exp}}$, then so does
the other. Hence, $\Phi\overset{d}{=}\overline{\Phi}\overset{d}{=}\Phi^{\ast}$
as well, but it is not immediate that both $(\Phi,\Psi)\ $and $(\overline
{\Phi},\overline{\Psi})$ are independent pairs.\newline\newline

To prove that in fact $(\Phi,\Psi)\overset{d}{=}(\overline{\Phi}%
,\overline{\Psi})\overset{d}{=}(\Phi_{\mathrm{Exp}},\Psi^{\ast})$, we will
first check that
\begin{gather}
(\mu(A_{l})\Phi_{A_{l}}\circ T_{A_{l}},\Psi_{A_{l}})\overset{\nu_{l}%
}{\Longrightarrow}\,(\Phi_{\mathrm{Exp}},\Psi^{\ast})\text{ \quad as
}l\rightarrow\infty\label{Eq_nhnhnhnhnhnnhnhhhhh}\\
\text{holds for }(\nu_{l})=(\mu)\text{ iff it holds for }(\nu_{l})=(\mu
_{A_{l}})\text{.}\nonumber
\end{gather}
Once this is established, we show that the one component $\mu(A_{l}%
)\varphi_{A_{l}}$ of $R_{l}$ missing in (\ref{Eq_nhnhnhnhnhnnhnhhhhh}) is
asymptotically independent of the rest: Writing $\mathbf{\sigma}\Phi
:=(\varphi^{(i+1)})_{i\geq0}$ and $\mathbf{\sigma}\overline{\Phi}%
:=(\overline{\varphi}^{(i+1)})_{i\geq0}$ for the shifted versions of $\Phi$
and $\overline{\Phi}$ respectively, we claim that
\begin{equation}
\text{if }(\overline{\Phi},\overline{\Psi})\overset{d}{=}(\Phi_{\mathrm{Exp}%
},\Psi^{\ast})\text{, then }\varphi^{(0)}\text{ is independent of
}(\mathbf{\sigma}\Phi,\Psi)\text{,} \label{Eq_gehtschonu1}%
\end{equation}
whereas
\begin{equation}
\text{if }(\Phi,\Psi)\overset{d}{=}(\Phi_{\mathrm{Exp}},\Psi^{\ast})\text{,
then }\overline{\varphi}^{(0)}\text{ is independent of }(\mathbf{\sigma
}\overline{\Phi},\overline{\Psi})\text{,} \label{Eq_gehtschonu2}%
\end{equation}
Together, assertions (\ref{Eq_nhnhnhnhnhnnhnhhhhh}) - (\ref{Eq_gehtschonu2})
prove our theorem.\newline\newline\textbf{(ii)} Validating
(\ref{Eq_nhnhnhnhnhnnhnhhhhh}) is straightforward: Set $R_{l}^{\prime}%
:=(\mu(A_{l})\Phi_{A_{l}}\circ T_{A_{l}},\Psi_{A_{l}}) $, $l\geq1$. According
to Proposition \ref{P_AsyInvarHitProc} c) and Proposition
\ref{P_AsyInvarAdmissDelayInternal} b), the sequence $(R_{l}^{\prime})$ is
asymptotically $T$-invariant in measure. Recalling Proposition
\ref{P_CharAdmissTimeDelay} b) and \ref{P_AsyInvarAdmissDelayInternal} c), we
see that $(\tau_{l})$ is an admissible delay sequence for $(R_{l}^{\prime})$.
Now (\ref{Eq_mnhjnggnfjfggnjf2}) allows us to appeal to Proposition
\ref{P_AsyInvarSeqNuLVsMu} to complete the proof of
(\ref{Eq_nhnhnhnhnhnnhnhhhhh}).\newline\newline\textbf{(iii)} Preparing for
the proof of (\ref{Eq_gehtschonu1}) and (\ref{Eq_gehtschonu2}) we set, for
$M\in\mathcal{B}_{\mathfrak{E}}$, $B_{l}(M):=\{(\mu(A_{l})\Phi_{A_{l}%
},\widetilde{\Psi}_{A_{l}})\in M\}\in\mathcal{A}$. Observe then that due to
$(\mu(A)\Phi_{A}\circ T_{A},\Psi_{A})=(\mu(A)\Phi_{A},\widetilde{\Psi}%
_{A})\circ T_{A}$, independence of $\varphi^{(0)}$ and $(\mathbf{\sigma}%
\Phi,\Psi)$ follows if we show that
\begin{gather}
\mu\left(  \{\mu(A_{l})\varphi_{A_{l}}\leq t\}\cap T_{A_{l}}^{-1}%
B_{l}(M)\right)  \longrightarrow(1-e^{-t})\Pr[(\Phi_{\mathrm{Exp}},\Psi^{\ast
})\in M]\nonumber\\
\text{for }t>0\text{ and }M\in\mathcal{B}_{\mathfrak{E}}\text{ with }\Pr
[(\Phi_{\mathrm{Exp}},\Psi^{\ast})\in\partial M]=0\text{.}
\label{Eq_bhbhggggg1}%
\end{gather}
(Use a variant of Theorem 2.3 of \cite{Bi} to argue as in the proof of Theorem
2.8 of \cite{Bi}.) Analogously, independence of $\overline{\varphi}^{(0)}$ and
$(\mathbf{\sigma}\overline{\Phi},\overline{\Psi})$ is immediate if
\begin{gather}
\mu_{A}\left(  \{\mu(A_{l})\varphi_{A_{l}}>s\}\cap T_{A_{l}}^{-1}%
B_{l}(M)\right)  \longrightarrow e^{-s}\Pr[(\Phi_{\mathrm{Exp}},\Psi^{\ast
})\in M]\nonumber\\
\text{for }s>0\text{ and }M\in\mathcal{B}_{\mathfrak{E}}\text{ with }\Pr
[(\Phi_{\mathrm{Exp}},\Psi^{\ast})\in\partial M]=0\text{.}
\label{Eq_bhbhggggg}%
\end{gather}
\newline

Now Lemma 4.1 of \cite{Z11} shows that for any $A,B\in\mathcal{A}$ and
$t\geq0$,
\begin{equation}
\left\vert \int_{0}^{t}\mu_{A}\left(  \{\mu(A)\varphi_{A}>s\}\cap T_{A}%
^{-1}B\right)  \,ds-\mu\left(  \{\mu(A)\varphi_{A}\leq t\}\cap T_{A}%
^{-1}B\right)  \right\vert \leq\mu(A)\text{.}%
\end{equation}
Hence, taking $A:=A_{l}$ and $B:=B_{l}(M)$ for an arbitrary $(\Phi
_{\mathrm{Exp}},\Psi^{\ast})$-continuity set $M\in\mathcal{B}_{\mathfrak{E}}$,
we see that for every $t\geq0$,
\begin{align}
\int_{0}^{t}\mu_{A_{l}}  &  \left(  \{\mu(A_{l})\varphi_{A_{l}}>s\}\cap
T_{A_{l}}^{-1}B_{l}(M)\right)  \,ds\nonumber\\
&  -\mu\left(  \{\mu(A_{l})\varphi_{A_{l}}\leq t\}\cap T_{A_{l}}^{-1}%
B_{l}(M)\right)  \longrightarrow0\quad\text{as }l\rightarrow\infty\text{.}
\label{Eq_KeyFromZ11Nice}%
\end{align}
\newline\textbf{(iv)} To validate (\ref{Eq_gehtschonu1}), suppose that
$(\overline{\Phi},\overline{\Psi})\overset{d}{=}(\Phi_{\mathrm{Exp}}%
,\Psi^{\ast})$. This entails
\[
\mu_{A_{l}}\left(  \{\mu(A_{l})\varphi_{A_{l}}>s\}\cap T_{A_{l}}^{-1}%
B_{l}(M)\right)  \longrightarrow e^{-s}\Pr[(\Phi_{\mathrm{Exp}},\Psi^{\ast
})\in M]
\]
whenever $s>0$ and $M$ is a $(\Phi_{\mathrm{Exp}},\Psi^{\ast})$-continuity
set. But then (\ref{Eq_bhbhggggg1}) follows via the crucial relation
(\ref{Eq_KeyFromZ11Nice}) by dominated convergence, ensuring independence of
$\varphi^{(0)}$ and $(\mathbf{\sigma}\Phi,\Psi)$ as required.\newline

Similarly, to prove (\ref{Eq_gehtschonu2}), suppose that $(\Phi,\Psi
)\overset{d}{=}(\Phi_{\mathrm{Exp}},\Psi^{\ast})$. Then,
\[
\mu\left(  \{\mu(A_{l})\varphi_{A_{l}}\leq t\}\cap T_{A_{l}}^{-1}%
B_{l}(M)\right)  \longrightarrow(1-e^{-t})\Pr[(\Phi_{\mathrm{Exp}},\Psi^{\ast
})\in M]
\]
for $t>0$ and any $(\Phi_{\mathrm{Exp}},\Psi^{\ast})$-continuity set $M$.
Fixing $M$ and varying $t$, we can use (\ref{Eq_KeyFromZ11Nice}) once again
and apply Lemma 4.2 of \cite{Z11} to obtain (\ref{Eq_bhbhggggg}), and hence
the desired independence of $\overline{\varphi}^{(0)}$ and $(\mathbf{\sigma
}\overline{\Phi},\overline{\Psi})$.
\end{proof}

%

\vspace{0.3cm}%

The argument for the joint limit theorem elaborates on a principle used before.

\begin{proof}
[\textbf{Proof of Theorem \ref{T_SpatiotemporalIIDLimits}}]\textbf{(i)} By
compactness of $(\mathfrak{M}([0,\infty]^{\mathbb{N}_{0}}\times\mathfrak{Z}%
^{\mathbb{N}_{0}}),D_{[0,\infty]^{\mathbb{N}_{0}}\times\mathfrak{Z}%
^{\mathbb{N}_{0}}})$ we may assume w.l.o.g. that
\begin{equation}
(\mu(A_{l})\Phi_{A_{l}},\widetilde{\Psi}_{A_{l}})\overset{\mu_{A_{l}}%
}{\Longrightarrow}\,(\Phi,\widetilde{\Psi})\quad\text{as }l\rightarrow
\infty\text{,} \label{Eq_BasicAssmVZDGVDZEV}%
\end{equation}
with $\Phi=(\varphi^{(0)},\varphi^{(1)},\ldots)$ and $\widetilde{\Psi}%
=(\psi^{(0)},\psi^{(1)},\ldots)$ random sequences in $[0,\infty]$ and
$\mathfrak{Z}$, respectively. Theorem \ref{T_AsyIntStateProcIndep} shows that
$\widetilde{\Psi}$ is iid with $\mathrm{law}(\psi^{(0)})=\mathrm{law}(\psi)$
so that $\mathrm{law}(\widetilde{\Psi})=\mathrm{law}(\Psi^{\ast})$. Next,
taking $s=0$ and observing that $A_{l}\cap\{\mu(A_{l})\varphi_{A_{l}%
}>0\}=A_{l}$, we see that Theorem
\ref{T_MOreFlexiblePoissonLimitFromCompactness} guarantees $\mathrm{law}%
(\Phi)=\mathrm{law}(\Phi_{\mathrm{Exp}})$.\newline

The main point is to show that $\Phi$ and $\widetilde{\Psi}$ are independent.
Let $\mathbf{\sigma}\Phi:=(\varphi^{(1)},\varphi^{(2)},\ldots)$ and
$\mathbf{\sigma}\widetilde{\Psi}:=(\psi^{(1)},\psi^{(2)},\ldots)$ denote the
shifted versions of the individual limit processes. We will first show that
\begin{equation}
\psi^{(0)}\text{ \quad is independent of \quad}(\Phi,\mathbf{\sigma}%
\widetilde{\Psi})\text{,} \label{Eq_TwistIndependence1}%
\end{equation}
and then check that
\begin{equation}
\varphi^{(0)}\text{ \quad is independent of \quad}(\mathbf{\sigma}%
\Phi,\mathbf{\sigma}\widetilde{\Psi})\text{.} \label{Eq_TwistIndependence2}%
\end{equation}
Together these imply that $(\psi^{(0)},\varphi^{(0)},(\mathbf{\sigma}%
\Phi,\mathbf{\sigma}\widetilde{\Psi}))$ is an independent triple. But since
convergence in (\ref{Eq_BasicAssmVZDGVDZEV}) uses the $T_{A_{l}}$-invariant
measures $\mu_{A_{l}}$ under which, for each $l\geq1$, $(\mu(A_{l})\Phi
_{A_{l}},\widetilde{\Psi}_{A_{l}})$ is a stationary sequence in $[0,\infty
]\times\mathfrak{Z}$, we see that so is $(\Phi,\widetilde{\Psi})$, meaning
that $\mathrm{law}(\Phi,\widetilde{\Psi})=\mathrm{law}(\mathbf{\sigma}%
\Phi,\mathbf{\sigma}\widetilde{\Psi})$. Therefore the above can be iterated to
show that for any $m\geq1$, $(\psi^{(0)},\varphi^{(0)},\ldots,\psi
^{(m-1)},\varphi^{(m-1)},(\mathbf{\sigma}^{m}\Phi,\mathbf{\sigma}%
^{m}\widetilde{\Psi}))$ is an independent tuple. Hence $\{\psi^{(j)}%
,\varphi^{(k)}\}_{j,k\geq0}$ is indeed an independent family. \newline%
\newline\textbf{(ii)} To prove (\ref{Eq_TwistIndependence1}) we are going to
apply Theorem \ref{T_AsyIndepForAsyInvarSequences} with $\mathfrak{E}%
:=\mathfrak{Z}$, $\mathfrak{E}^{\prime}:=[0,\infty]^{\mathbb{N}_{0}}%
\times\mathfrak{Z}^{\mathbb{N}}$, $R_{l}:=\psi_{A_{l}}$, $R_{l}^{\prime}%
:=(\mu(A_{l})\Phi_{A_{l}},\Psi_{A_{l}})$, representing $(\mu(A_{l})\Phi
_{A_{l}},\widetilde{\Psi}_{A_{l}})$ as the map $(R_{l},R_{l}^{\prime
}):X\rightarrow\mathfrak{E}\times\mathfrak{E}^{\prime}$. The sequence
$(R_{l}^{\prime})$ is asymptotically $T$-invariant in measure since both
$(\mu(A_{l})\Phi_{A_{l}})$ and $(\Psi_{A_{l}})$ are (due to Propositions
\ref{P_AsyInvarHitProc} b) and \ref{P_AsyInvarAdmissDelayInternal} b)).

For any $F\in\mathcal{B}_{\mathfrak{Z}}^{\pi}$ with $\Pr[\psi\in F]>0$ pick
$(\nu_{l,F})$, $(\tau_{l,F})$ and $\mathfrak{K}_{F}$ as in the statement of
Theorem \ref{T_SpatiotemporalIIDLimits}. By assumption, $\overline{\nu}%
_{l,F}:=T_{\ast}^{\tau_{l,F}}\nu_{l,F}\in\mathfrak{K}_{F}$ for $l\geq1$, so
that (\ref{Eq_TwistIndependence1}) follows via Theorem
\ref{T_AsyIndepForAsyInvarSequences} once we check that
\begin{equation}
(\tau_{l,F})_{l\geq1}\text{ is an admissible delay for }(\mu(A_{l})\Phi
_{A_{l}},\Psi_{A_{l}})_{l\geq1}\text{ and }(\nu_{l,F})_{l\geq1}\text{.}
\label{Eq_naendlichpasstszamm}%
\end{equation}
Here it is enough to treat $(\mu(A_{l})\Phi_{A_{l}})$ and $(\Psi_{A_{l}})$
separately. But in the first case (\ref{Eq_vzgvzgvzgvzgvzgv1}) and
(\ref{Eq_owowowowowoowwowowow}) allow us to appeal to Proposition
\ref{P_CharAdmissTimeDelay} a), while (\ref{Eq_owowowowowoowwowowow}) alone
takes care of the second case via Proposition
\ref{P_AsyInvarAdmissDelayInternal} c). \newline\newline\textbf{(iii)} We
validate (\ref{Eq_TwistIndependence2}) analogously, this time applying Theorem
\ref{T_AsyIndepForAsyInvarSequences} with $\mathfrak{E}:=[0,\infty]$,
$\mathfrak{E}^{\prime}:=[0,\infty]^{\mathbb{N}}\times\mathfrak{Z}%
^{\mathbb{N}_{0}}$, $R_{l}:=\mu(A_{l})\varphi_{A_{l}}$, $R_{l}^{\prime}%
:=(\mu(A_{l})\Phi_{A_{l}}\circ T_{A_{l}},\Psi_{A_{l}})$, representing
$(\mu(A_{l})\Phi_{A_{l}},\Psi_{A_{l}})$ as the map $(R_{l},R_{l}^{\prime
}):X\rightarrow\mathfrak{E}\times\mathfrak{E}^{\prime}$. Again, $(R_{l}%
^{\prime})$ is asymptotically $T$-invariant in measure. We use $\mathcal{B}%
_{\mathfrak{E}}^{\pi}:=\{(s,\infty):s\geq0\}$. For any $s\geq0$ we have
$\Pr[\varphi^{(0)}>s]>0$ since we already know that $\varphi^{(0)}$ has an
exponential distribution. Take $(\nu_{l,s})$, $(\tau_{l,s})$ and
$\mathfrak{K}_{s}$ as in the statement of Theorem
\ref{T_SpatiotemporalIIDLimits}. By assumption, $\overline{\nu}_{l,s}%
:=T_{\ast}^{\tau_{l,s}}\nu_{l,s}\in\mathfrak{K}_{s}$ for $l\geq1$, and
(\ref{Eq_TwistIndependence2}) follows via Theorem
\ref{T_AsyIndepForAsyInvarSequences} if we check that
\begin{equation}
(\tau_{l,s})_{l\geq1}\text{ is an admissible delay for }(\mu(A_{l})\Phi
_{A_{l}}\circ T_{A_{l}},\Psi_{A_{l}})_{l\geq1}\text{ and }(\nu_{l,s})_{l\geq
1}\text{.}%
\end{equation}
Condition (\ref{Eq_xdrxrxxdxrxdrxd}) ensures that $(\tau_{l,s})_{l\geq1}$ is
admissible for $(\mu(A_{l})\Phi_{A_{l}}\circ T_{A_{l}})$ via Proposition
\ref{P_CharAdmissTimeDelay} b), and also for $(\Psi_{A_{l}})$ by virtue of
Proposition \ref{P_AsyInvarAdmissDelayInternal} c).\newline\newline%
\textbf{(iv)} The above shows that $(\mu(A_{l})\Phi_{A_{l}},\widetilde{\Psi
}_{A_{l}})\overset{\mu_{A_{l}}}{\Longrightarrow}\,(\Phi_{\mathrm{Exp}}%
,\Psi^{\ast})$, hence $(\mu(A_{l})\Phi_{A_{l}},\Psi_{A_{l}})=(\mu(A_{l}%
)\Phi_{A_{l}},\mathbf{\sigma}\widetilde{\Psi}_{A_{l}})\overset{\mu_{A_{l}}%
}{\Longrightarrow}(\Phi_{\mathrm{Exp}},\mathbf{\sigma}\Psi^{\ast})\overset
{d}{=}(\Phi_{\mathrm{Exp}},\Psi^{\ast})$. To prove that we also have
$(\mu(A_{l})\Phi_{A_{l}},\Psi_{A_{l}})\overset{\mu}{\Longrightarrow}%
(\,\Phi_{\mathrm{Exp}},\Psi^{\ast})$ we appeal to Proposition
\ref{P_AsyInvarSeqNuLVsMu}:

Set $\mathfrak{E}:=[0,\infty]^{\mathbb{N}_{0}}\times\mathfrak{Z}^{\mathbb{N}}%
$, $R_{l}:=(\mu(A_{l})\Phi_{A_{l}},\Psi_{A_{l}})$ and $\nu_{l}:=\mu_{A_{l}}$.
As before, $(R_{l})$ is is asymptotically $T$-invariant in measure. Take
$\tau_{l}:=\tau_{l,\mathfrak{Z}}$ and $\mathfrak{K}:=\mathfrak{K}%
_{\mathfrak{Z}}$ as in assumption b) of Theorem
\ref{T_SpatiotemporalIIDLimits}, and note that $\nu_{l}=\nu_{l,\mathfrak{Z}%
}=\mu_{A_{l}\cap\{\psi_{A_{l}}\in\mathfrak{Z}\}}$. Finally, recall that we
have already shown admissibility of this sequence $(\tau_{l})$ for the present
$(R_{l})$ and $(\nu_{l})$ in step (ii), see (\ref{Eq_naendlichpasstszamm}).
\end{proof}

%

\vspace{0.3cm}%

\section{Illustration in some easy standard situations\label{Sec_GMmaps}}

We illustrate the ease with which the above results can sometimes be applied
by studying some basic piecewise invertible dynamical systems.%

\vspace{0.3cm}%
%

\noindent
\textbf{Piecewise invertible systems.} We consider situations in which
$(X,\mathrm{d}_{X})$ is a metric space with Borel $\sigma$-field
$\mathcal{A}=\mathcal{B}_{X}$, and where $X$ comes with a partition $\xi_{0}$
(mod $\lambda$) into open \emph{components} (e.g. $X\ $may be a union of
disjoint open intervals in $\mathbb{R}$). Let $\lambda$ be a $\sigma$-finite
reference measure on $\mathcal{A}$. A \emph{piecewise invertible system} on
$X$ is a quintuple $(X,\mathcal{A},\lambda,T,\xi)$, where $\xi=\xi_{1}$ is a
(finite or) countable partition mod $\lambda$ of $X$ into open sets, refining
$\xi_{0}$, such that each \emph{branch} of $T$, i.e. its restriction to any of
its \emph{cylinders} $Z\in\xi$ is a homeomorphism onto $TZ$,
\emph{null-preserving} with respect to $\lambda$, that is, $\lambda\mid
_{Z}\circ T^{-1}\ll\lambda$. If the measure is $T$-invariant, we denote it by
$\mu$ and call $(X,\mathcal{A},\mu,T,\xi)$ a \emph{measure preserving} system.
The system is called \emph{uniformly expanding} if there is some $\rho
\in(0,1)$ such that $d_{X}(x,y)\leq\rho\cdot d_{X}(Tx,Ty)$ whenever $x,y\in
Z\in\xi$.

We let $\xi_{n}$ denote the family of \emph{cylinders of rank }$n$, that is,
the sets of the form $Z=[Z_{0},\ldots,Z_{n-1}]:=\bigcap_{i=0}^{n-1}T^{-i}%
Z_{i}$ with $Z_{i}\in\xi$. Write $\xi_{n}(x)$ for the element of $\xi_{n}$
containing $x$ (which is defined a.e.). Each \emph{iterate} $(X,\mathcal{A}%
,\mu,T^{n},\xi_{n})$, $n\geq1$, of the system is again piecewise invertible.
The \emph{inverse branches} will be denoted $v_{Z}:=(T^{n}\mid_{Z})^{-1}%
:T^{n}Z\rightarrow Z$, $Z\in\xi_{n}$. All $v_{Z}$ have Radon-Nikodym
derivatives $v_{Z}^{\prime}:=d(\lambda\circ v_{Z})/d\lambda$.

The system is \emph{Markov} if $TZ\cap Z^{\prime}\neq\varnothing$ for
$Z,Z^{\prime}\in\xi$ implies $Z^{\prime}\subseteq TZ$.%

\vspace{0.3cm}%
%

\noindent
\textbf{Gibbs-Markov maps.} One important basic class of piecewise invertible
systems $(X,\mathcal{A},\mu,T,\xi)$ is that of probability preserving
\emph{Gibbs-Markov maps} (\emph{GM maps}). This means that $\mathrm{diam}%
(X)<\infty$, and $\mu$ is an invariant probability, that the system has a
uniformly expanding iterate $T^{N}$, and satisfies the \emph{big image
property}, so that $\flat:=\inf_{Z\in\xi}\mu(TZ)>0$. Moreover, the
$v_{Z}^{\prime}$, $Z\in\xi$, have well behaved versions in that there exists
some $r>0$ such that $\left\vert v_{Z}^{\prime}(x)/v_{Z}^{\prime
}(y)-1\right\vert \leq r\,d_{X}(x,y)$ whenever $x,y\in TZ$, $Z\in\xi$ (see
\cite{AD}). In this case $r$ can be chosen in such a way that in fact
\begin{equation}
\left\vert \frac{v_{Z}^{\prime}(x)}{v_{Z}^{\prime}(y)}-1\right\vert \leq
r\,d_{X}(x,y)\text{ \quad whenever }n\geq1\text{ and }x,y\in T^{n}Z,Z\in
\xi_{n}\text{.} \label{Eq_RegularityForGMSys}%
\end{equation}
In this context, $v_{Z}^{\prime}$ will always denote such versions of the a.e.
defined Radon-Nikodym derivatives $d(\mu\circ v_{Z})/d\mu$.\newline

We recall a few well known basic properties of such systems, all of which are
obtained by elementary routine arguments. Let $\beta$ be the partition
generated by $T\xi$. By (\ref{Eq_RegularityForGMSys}) the normalized image
measures $T_{\ast}^{n}\mu_{Z}$ with $n\geq1$ and $Z\in\xi_{n}$ have densities
belonging to $\mathcal{U}:=\{u\in\mathcal{D}(\mu):\left\vert
u(x)/u(y)-1\right\vert \leq r\,d_{X}(x,y)$ whenever $x,y\in B\in\beta\}$, that
is,%
\begin{equation}
T_{\ast}^{n}\mu_{Z}\in\mathfrak{K}\text{ \qquad for all }n\geq1\text{ and
}Z\in\xi_{n}\text{,} \label{Eq_GMgoodK}%
\end{equation}
where $\mathfrak{K}:=\{\nu\in\mathfrak{P}:d\nu/d\mu\in\mathcal{U}\}$. But
$\mathcal{U}$ is compact in $L_{1}(\mu)$ (Arzela-Ascoli, as in \S 4.7 of
\cite{A0}) and convex, so that $\mathfrak{K}$ is a compact convex set in
$\mathfrak{P}$.

Property (\ref{Eq_RegularityForGMSys}) also implies \emph{bounded distortion}
in that
\begin{equation}
\mu_{Z}(Z\cap T^{-n}A)=e^{\pm r}\mu_{T^{n}Z}(A)\text{ \quad for all }%
n\geq1\text{, }Z\in\xi_{n}\text{, and }A\in\mathcal{A}\text{.}
\label{Eq_AbstractAdlerInfty}%
\end{equation}
In particular,
\begin{equation}
\mu(Z\cap T^{-n}A)\leq\flat^{-1}e^{r}\mu(Z)\mu(A)\text{ \quad for all }%
n\geq1\text{, }Z\in\xi_{n}\text{, and }A\in\mathcal{A}\text{.}
\label{Eq_UpperDistortBound}%
\end{equation}
Also, an easy argument provides constants $\kappa\geq1$ and $q\in(0,1)$ such
that%
\begin{equation}
\mu(Z)\leq\kappa q^{n}\text{ \qquad for all }n\geq1\text{ and }Z\in\xi
_{n}\text{.} \label{Eq_GMMapsContractionOfCylMeasure}%
\end{equation}
\newline%

\noindent
\textbf{Poisson asymptotics for (unions of) cylinders of GM-maps.} To
demonstrate how convenient the assumptions of our limit theorems are, we first
illustrate their use in the setup of cylinders of GM-maps, re-proving the well-known

\begin{theorem}
[\textbf{Poisson asymptotics for shrinking cylinders of GM maps}%
]\label{T_PoissonAsyGMCyls}Let $(X,\mathcal{A},\mu,T,\xi)$ be an ergodic
probability preserving Gibbs-Markov system. Let $x^{\ast}\in X$ be a point
such that the cylinder $A_{l}:=\xi_{l}(x^{\ast})$ is defined for all $l\geq
1$.\newline\textbf{a)} Assume that $x^{\ast}\in X$ is not periodic, then
$(A_{l})$ exhibits Poisson asymptotics,
\begin{equation}
\mu(A_{l})\Phi_{A_{l}}\overset{\mu}{\Longrightarrow}\,\Phi_{\mathrm{Exp}%
}\text{ \quad and \quad}\mu(A_{l})\Phi_{A_{l}}\overset{\mu_{A_{l}}%
}{\Longrightarrow}\,\Phi_{\mathrm{Exp}}\text{ \quad as }l\rightarrow
\infty\text{.}%
\end{equation}
\textbf{b)} On the other hand, if $x^{\ast}\in X$ is periodic, $x^{\ast}%
=T^{p}x^{\ast}$ with $p\geq1$ minimal, then
\begin{equation}
\mu(A_{l})\Phi_{A_{l}}\overset{\mu_{A_{l}}}{\Longrightarrow}\,\Phi
_{(\mathrm{Exp},\theta)}\text{ \quad as }l\rightarrow\infty\text{,}%
\end{equation}
where $\theta:=1-v_{A_{p}}^{\prime}(x^{\ast})\in(0,1)$.
\end{theorem}

\vspace{0.3cm}%

Our proof via Theorems \ref{T_MOreFlexiblePoissonLimitFromCompactness} and
\ref{T_CPoissonViaCompactness} will only employ the basic elementary facts
about Gibbs-Markov maps mentioned before. The well-known strong mixing
properties (valid in aperiodic situations) which follow (by more sophisticated
arguments) from related observations are not used.%

\vspace{0.3cm}%

\begin{proof}
\textbf{a)} We are going to apply Theorem
\ref{T_MOreFlexiblePoissonLimitFromCompactness}, using $\nu_{l}:=\mu_{A_{l}}$,
$\mathfrak{K}$ as in (\ref{Eq_GMgoodK}), and the obvious delay times $\tau
_{l}:=l$. Set $\widetilde{\kappa}:=\flat^{-1}e^{r}\kappa$.

Condition (\ref{Eq_DontWaitTooLong1}) is trivially satisfied since $\mu
(A_{l})$ is exponentially small, while condition
(\ref{Eq_TheFabulousCompactnessCondition1}) is taken care of by
(\ref{Eq_GMgoodK}).

To validate (\ref{Eq_HardlyAnyReturnsUpToTimeM1}), take any $\varepsilon>0$.
Choose $K\geq1$ so large that $\widetilde{\kappa}q^{K}/(1-q)<\varepsilon$.
Since $x^{\ast}$ is not periodic, and $T$ is continuous on cylinders, there is
some $l^{\prime}$ such that $\varphi_{A_{l}}>K$ on $A_{l}$ whenever $l\geq
l^{\prime}$ (recall that $\mathrm{diam}(A_{l})\rightarrow0$). For every
$k\in\{1,\ldots,l\}$ we have $A_{l}\subseteq A_{k}=\xi_{k}(x^{\ast})\in\xi
_{k}$, and therefore
\begin{equation}
\mu_{A_{l}}(T^{-k}A_{l})\leq\mu(A_{l})^{-1}\mu(A_{k}\cap T^{-k}A_{l})\leq
\flat^{-1}e^{r}\mu(A_{k})\leq\widetilde{\kappa}q^{k}%
\end{equation}
due to (\ref{Eq_UpperDistortBound}) and
(\ref{Eq_GMMapsContractionOfCylMeasure}). For $l\geq l^{\prime}$ we then find
that
\begin{align}
\mu_{A_{l}}\left(  \varphi_{A_{l}}\leq l\right)   &  =\mu_{A_{l}}\left(
K\leq\varphi_{A_{l}}\leq l\right)  =\mu_{A_{l}}(%
{\textstyle\bigcup\nolimits_{k=K}^{l}}
T^{-k}A_{l})\nonumber\\
&  \leq%
{\textstyle\sum\nolimits_{k=K}^{l}}
\mu_{A_{l}}\left(  T^{-k}A_{l}\right)  \leq\widetilde{\kappa}%
{\textstyle\sum\nolimits_{k=K}^{l}}
q^{k}<\varepsilon\text{,}%
\end{align}
and (\ref{Eq_HardlyAnyReturnsUpToTimeM1}) follows because $\varepsilon>0$ was arbitrary.%

\vspace{0.3cm}%
%

\noindent
\textbf{b)} We employ Theorem \ref{T_CPoissonViaCompactness}, using $\nu
_{l}:=\mu_{A_{l}}$ and $\mathfrak{K}$ as in (\ref{Eq_GMgoodK}). Define
$A_{l}^{\bullet}:=A_{l}\cap T^{-p}A_{l}=A_{l+p}\in\xi_{l+p}$ and $A_{l}%
^{\circ}:=A_{l}\setminus A_{l}^{\bullet}$, so that
\[
\mu(A_{l}^{\bullet})=\mu(A_{p}\cap T^{-p}A_{l})=T_{\ast}^{p}(\mu\mid_{A_{p}%
})(A_{l})=\mu(A_{l})\cdot%
{\textstyle\int\nolimits_{A_{l}}}
v_{A_{p}}^{\prime}\,d\mu_{A_{l}}\text{.}%
\]
Now $x^{\ast}\in A_{l}\subseteq T^{p}A_{p}$ for $l>p$, and $\mathrm{diam}%
(A_{l})\searrow0$. As $v_{A_{p}}^{\prime}$ is continuous on $T^{p}A_{p}$ with
$v_{A_{p}}^{\prime}(x^{\ast})=1-\theta$, we get $\mu(A_{l}^{\bullet}%
)\sim(1-\theta)\mu(A_{l})$ as $l\rightarrow\infty$, proving
(\ref{Eq_CompundingTheProportion}).

Observe that $\varphi_{A_{l}}=p$ on $A_{l}^{\bullet}$, and accordingly define
$\tau_{l}:=p$ on $A_{l}^{\bullet}$. Then (\ref{Eq_ManyHappyReturns}) is clear.
Moreover, $(T_{A_{l}})_{\ast}\mu_{A_{l}^{\bullet}}=\mu(A_{l}^{\bullet}%
)^{-1}\left(  T_{\ast}^{p}(\mu\mid_{A_{l}})\right)  \mid_{A_{l}}$, and this
measure is given by the probability density $\mu(A_{l}^{\bullet})^{-1}%
1_{A_{l}}v_{A_{l}}^{\prime}=:h_{l}^{\bullet}$. Comparing these to the
densities $\mu(A_{l})^{-1}1_{A_{l}}=:h_{l}$ of the $\mu_{A_{l}}$ we obtain
(\ref{Eq_ControlShortReturns}), because of $\mathrm{diam}(A_{l})\searrow0$ and
(\ref{Eq_RegularityForGMSys}).

Turning to the escaping part $A_{l}^{\circ}$, note that it is $\xi_{l+p}%
$-measurable (mod $\mu$). Define $\tau_{l}:=l+p$ on $A_{l}^{\circ}$, then
(\ref{Eq_TheFabulousCompactnessCondition2}) is immediate from
(\ref{Eq_GMgoodK}) and convexity of $\mathfrak{K}$. We finally check
(\ref{Eq_HardlyAnyReturnsUpToTimeM2}). Up to a set of measure zero,
$A_{l}^{\circ}=%
{\textstyle\bigcup\nolimits_{W\in\xi_{p}\setminus\{A_{p}\}}}
A_{l}\cap T^{-l}W$, so that $\varphi_{A_{l}}>l$ on $A_{l}^{\circ}$.
Therefore,
\begin{align*}
\mu_{A_{l}^{\circ}}\left(  \varphi_{A_{l}}\leq\tau_{l}\right)   &  =\mu
_{A_{l}^{\circ}}\left(  l<\varphi_{A_{l}}\leq\tau_{l}\right)  \leq%
{\textstyle\sum\nolimits_{k=1}^{p}}
\mu_{A_{l}^{\circ}}(T^{-(l+k)}A_{l})\\
&  \leq\mu(A_{l}^{\circ})^{-1}%
{\textstyle\sum\nolimits_{k=1}^{p}}
\mu(A_{l}\cap T^{-(l+k)}A_{l})\\
&  \leq\mu(A_{l}^{\circ})^{-1}%
{\textstyle\sum\nolimits_{k=1}^{p}}
{\textstyle\sum\nolimits_{V\in\xi_{k}}}
\mu((A_{l}\cap T^{-l}V)\cap T^{-(l+k)}A_{l})\\
&  \leq\flat^{-1}e^{r}\mu_{A_{l}}(A_{l}^{\circ})^{-1}%
{\textstyle\sum\nolimits_{k=1}^{p}}
{\textstyle\sum\nolimits_{V\in\xi_{k}}}
\mu\left(  A_{l}\cap T^{-l}V\right) \\
&  =p\,\flat^{-1}e^{r}\mu_{A_{l}}(A_{l}^{\circ})^{-1}\,\mu(A_{l}%
)\longrightarrow0\text{ \quad as }l\rightarrow\infty\text{,}%
\end{align*}
where we have used $A_{l}=%
{\textstyle\bigcup\nolimits_{V\in\xi_{k}}}
A_{l}\cap T^{-l}V$ (disjoint), (\ref{Eq_UpperDistortBound}),
(\ref{Eq_GMMapsContractionOfCylMeasure}), and the fact that $\mu_{A_{l}}%
(A_{l}^{\circ})\rightarrow\theta\in(0,1)$.
\end{proof}

Another simple situation is that of small sets $A_{l}$ which consist of (fewer
and fewer) rank-one cylinders. Partitioning the $A_{l}$ into subsets of the
same type, and recording which of those subsets an orbit hits, leads us to the
study of basic discrete\ local processes.

\begin{theorem}
[\textbf{Poisson asymptotics and local processes for unions of cylinders of GM
maps}]\label{T_PoissonAsyGMUnionCyls}Let $(X,\mathcal{A},\mu,T,\xi)$ be an
ergodic probability preserving Gibbs-Markov system. Let $(A_{l})$ be a
sequence of asymptotically rare events such that each $A_{l}$ is $\xi
$-measurable. \newline\textbf{a)} Then $(A_{l})$ exhibits Poisson
asymptotics,
\begin{equation}
\mu(A_{l})\Phi_{A_{l}}\overset{\mu}{\Longrightarrow}\,\Phi_{\mathrm{Exp}%
}\text{ \quad and \quad}\mu(A_{l})\Phi_{A_{l}}\overset{\mu_{A_{l}}%
}{\Longrightarrow}\,\Phi_{\mathrm{Exp}}\text{ \quad as }l\rightarrow
\infty\text{.} \label{Eq_tptptptptptpt}%
\end{equation}
\textbf{b)} Assume further that, for some integer $m\geq2$, each $A_{l}$ is
partitioned into $\xi$-measurable subsets $A_{l}^{(1)},\ldots A_{l}^{(m)}$
which satisfy $\mu_{A_{l}}(A_{l}^{(j)})\rightarrow\vartheta_{j}\in(0,1)$ as
$l\rightarrow\infty$. Define local observables by letting $\psi_{A_{l}}(x):=j$
if $x\in A_{l}^{(j)}$. Then,%
\begin{equation}
\left.
\begin{array}
[c]{c}%
(\mu(A_{l})\Phi_{A_{l}},\widetilde{\Psi}_{A_{l}})\overset{\mu_{A_{l}}%
}{\Longrightarrow}\\
(\mu(A_{l})\Phi_{A_{l}},\Psi_{A_{l}})\overset{\mu}{\Longrightarrow}%
\end{array}
\right\}  (\,\Phi_{\mathrm{Exp}},\Psi^{\ast})\quad\text{as }l\rightarrow
\infty\text{,} \label{Eq_GMJointLimits}%
\end{equation}
where $(\,\Phi_{\mathrm{Exp}},\Psi^{\ast})$ is an independent pair with
$\Psi^{\ast}$ a $(\vartheta_{1},\ldots,\vartheta_{m})$-Bernoulli sequence.
\end{theorem}

This, too, is an easy consequence of the results above.

\begin{proof}
\textbf{a)} We check that Theorem
\ref{T_MOreFlexiblePoissonLimitFromCompactness} applies with $\nu_{l}%
:=\mu_{A_{l}}$, $\tau_{l}:=1$ and $\mathfrak{K}$ as in (\ref{Eq_GMgoodK}).
Since $(\tau_{l})$ is uniformly bounded, condition (\ref{Eq_DontWaitTooLong1})
is trivial. In view of (\ref{Eq_GMgoodK}) and $\xi$-measurability of the
$A_{l}$, (\ref{Eq_TheFabulousCompactnessCondition1}) is also satisfied. To
validate (\ref{Eq_HardlyAnyReturnsUpToTimeM1}), recall
(\ref{Eq_UpperDistortBound}) to see that indeed%
\begin{align*}
\mu_{A_{l}}\left(  \varphi_{A_{l}}=1\right)   &  =\mu_{A_{l}}\left(
T^{-1}A_{l}\right)  =\mu(A_{l})^{-1}%
{\textstyle\sum\nolimits_{Z\in\xi\cap A_{l}}}
\mu(Z\cap T^{-1}A_{l})\\
&  \leq\mu(A_{l})^{-1}%
{\textstyle\sum\nolimits_{Z\in\xi\cap A_{l}}}
\flat e^{r}\mu(Z)\mu(A_{l})\\
&  \leq\flat^{-1}e^{r}\mu(A_{l})\longrightarrow0\quad\text{as }l\rightarrow
\infty\text{.}%
\end{align*}
\textbf{b)} To establish the joint convergence asserted in
(\ref{Eq_GMJointLimits}), we will appeal to Theorem
\ref{T_SpatiotemporalIIDLimits}. The local observables $\psi_{A_{l}}$ take
their values in the compact discrete space $\mathfrak{Z}:=\{1,\ldots,m\}$, and
we are assuming that $\psi_{A_{l}}\overset{\mu_{A_{l}}}{\Longrightarrow}\psi$
with $\Pr[\psi=j]=\vartheta_{j}$ for all $j\in\mathfrak{Z}$. We use the same
$\mathfrak{K}$ as above.

To check condition (A) of Theorem \ref{T_SpatiotemporalIIDLimits}, take any
$s\in\lbrack0,\infty)$ and define $\tau_{l,s}:=\left\lfloor s/\mu
(A_{l})\right\rfloor $. Then $A_{l}\cap\{\varphi_{A_{l}}>s/\mu(A_{l}%
)\}=A_{l}\cap\{\varphi_{A_{l}}>\tau_{l,s}\}$, so that
(\ref{Eq_xdrxrxxdxrxdrxd}) is trivially fulfilled. On the other hand, this set
is $\xi_{\tau_{l,s}}$-measurable because $A_{l}$ is $\xi$-measurable.
Therefore, (\ref{Eq_GMgoodK}) and convexity of $\mathfrak{K}$ ensure that
$T_{\ast}^{\tau_{l,s}}\mu_{A_{l}\cap\{\mu(A_{l})\varphi_{A_{l}}>s\}}%
\in\mathfrak{K}$ for $l\geq1$.

In order to validate condition (B) of Theorem \ref{T_SpatiotemporalIIDLimits},
we use $\mathcal{B}_{\mathfrak{Z}}^{\pi}:=\{\{j\}:j\in\mathfrak{Z}\}$ and, for
arbitrary $F=\{j\}$, take $\tau_{l,F}:=1$, so that (\ref{Eq_vzgvzgvzgvzgvzgv1}%
) is automatically satisfied. As $A_{l}\cap\{\psi_{A_{l}}\in F\}=A_{l}^{(j)}$
is $\xi$-measurable, (\ref{Eq_GMgoodK}) and convexity of $\mathfrak{K}$
immediately show that $T_{\ast}^{\tau_{l,F}}\mu_{A_{l}\cap\{\psi_{A_{l}}\in
F\}}\in\mathfrak{K}$ for $l\geq1$. Finally, (\ref{Eq_owowowowowoowwowowow})
follows since
\begin{align*}
\mu_{A_{l}\cap\{\psi_{A_{l}}\in F\}}\left(  \varphi_{A_{l}}\leq\tau
_{l,F}\right)   &  =\mu_{A_{l}}(A_{l}^{(j)})^{-1}\mu(A_{l}^{(j)}\cap
T^{-1}A_{l})\\
&  =\mu_{A_{l}}(A_{l}^{(j)})^{-1}%
{\textstyle\sum\nolimits_{Z\in\xi\cap A_{l}^{(j)}}}
\mu\left(  Z\cap T^{-1}A_{l}\right) \\
&  \leq\flat^{-1}e^{r}\,\mu_{A_{l}}(A_{l}^{(j)})^{-1}%
{\textstyle\sum\nolimits_{Z\in\xi\cap A_{l}^{(j)}}}
\mu\left(  Z\right)  \mu\left(  A_{l}\right) \\
&  =\flat^{-1}e^{r}\,\mu\left(  A_{l}\right)  \longrightarrow0\text{ \quad as
}l\rightarrow\infty\text{,}%
\end{align*}
where we used (\ref{Eq_UpperDistortBound}) again.
\end{proof}

%

\vspace{0.3cm}%

We next provide some specific applications of this theorem in the context of
continued fraction expansions. But sequences $(A_{l})$ as in our theorem do
appear naturally in a variety of other situations. We mention one particular instance:

\begin{remark}
Under the assumptions of Theorem \ref{T_PoissonAsyGMUnionCyls} a), it is
immediate from Theorem \ref{T_HTSforVaryingMeasures} and Arzela-Ascoli that
(\ref{Eq_tptptptptptpt}) implies
\begin{equation}
\mu(A_{l})\Phi_{A_{l}}\overset{\nu_{l}}{\Longrightarrow}\,\Phi_{\mathrm{Exp}%
}\text{\quad as }l\rightarrow\infty\text{,} \label{Eq_bhbshbhds}%
\end{equation}
whenever $(\nu_{l})$ is a sequence of probabilities with $\sup_{l\geq
1}\mathrm{Lip}_{\xi}(d\nu_{l}/d\mu)<\infty$, where $\mathrm{Lip}_{\xi
}(w):=\sup_{Z\in\xi}\mathrm{Lip}_{Z}(w)$ with $\mathrm{Lip}_{Z}(w):=\sup
_{x,y\in Z,x\neq y}\mid w(x)-w(y)\mid/d(x,y)$. This is a stronger (functional)
version of Proposition 3.13 in \cite{PT}.
\end{remark}

%

\vspace{0.3cm}%
%

\noindent
\textbf{The continued fraction map. Variations on a theme of Doeblin.} We will
illustrate the use of Theorem \ref{T_PoissonAsyGMUnionCyls} in the setup of a
particularly prominent system. Set $X:=[0,1]$, $\mathcal{A}:=\mathcal{B}_{X}$,
and let $T:X\rightarrow X$ be the \emph{Gauss map} with $T0:=0$ and
\begin{equation}
Tx:=\dfrac{1}{x}-\left\lfloor \dfrac{1}{x}\right\rfloor =\dfrac{1}{x}-k\text{
for }x\in\left(  \dfrac{1}{k+1},\dfrac{1}{k}\right]  =:I_{k}\text{, }%
k\geq1\text{,}%
\end{equation}
which, since Gauss \cite{G}, is known to preserve the probability density%
\begin{equation}
h(x):=\frac{1}{\log2}\frac{1}{1+x}\text{, \quad}x\in X\text{.}%
\end{equation}
The invariant \emph{Gauss measure} $\mu$ on $\mathcal{A}$ defined by the
latter, $\mu(A):=\int_{A}h(x)\,dx$, is exact (and hence ergodic). Iteration of
$T$ reveals the \emph{continued fraction (CF) digits} of any $x\in X$, in
that
\begin{equation}
x=\frac{1}{\mathsf{a}_{1}(x)+\dfrac{1}{\mathsf{a}_{2}(x)+\cdots}}\text{ \quad
with \quad}\mathsf{a}_{n}(x)=\mathsf{a}\circ T^{n-1}(x)\text{, }n\geq1\text{,}%
\end{equation}
where $\mathsf{a}:X\rightarrow\mathbb{N}$ is the \emph{digit function}
corresponding to $\xi:=\{I_{k}:k\geq1\}$, i.e. $\mathsf{a}(x):=\left\lfloor
1/x\right\rfloor =k$ for $x\in I_{k}$. It is a standard fact that the ergodic
measure preserving piecewise invertible \emph{CF-system} $(X,\mathcal{A}%
,\mu,T,\xi)$ is Gibbs-Markov.

We shall focus on the simple sequence of $\xi$-measurable asymptotically rare
events given by $A_{l}:=\{\mathsf{a}\geq l\}=%
{\textstyle\bigcup\nolimits_{k\geq l}}
I_{k}$, which satisfy $\mu(A_{l})\sim1/(l\log2)$ as $l\rightarrow\infty$, and
allow to immediately apply Theorem \ref{T_PoissonAsyGMUnionCyls} a) to obtain
Poisson asymptotics,
\begin{equation}
\mu(A_{l})\Phi_{A_{l}}\overset{\mu}{\Longrightarrow}\,\Phi_{\mathrm{Exp}%
}\text{ \quad and \quad}\mu(A_{l})\Phi_{A_{l}}\overset{\mu_{A_{l}}%
}{\Longrightarrow}\,\Phi_{\mathrm{Exp}}\text{ \quad as }l\rightarrow
\infty\text{.}%
\end{equation}
This is a well-known classical fact with non-trivial history (\cite{Doeblin},
\cite{I}) and various extensions, see e.g. \cite{IK}. In the following we
refine this statement, using two different sequences of local observables in
order to obtain extra information on the distributions of the particular
digits observed when the orbit hits $A_{l}$. In either case, Theorem
\ref{T_PoissonAsyGMUnionCyls} b) applies without difficulties.%

\vspace{0.3cm}%

To get a better understanding of the actual size of those digits which happen
to exceed some large $l$, we show that, asymptotically, whether these large
digits are even of order $l/\vartheta$ (for some $\vartheta\in(0,1)$) is
determined by an independent sequence of $(1-\vartheta,\vartheta)$-coin flips.

\begin{proposition}
[\textbf{Just how large are large CF-digits?}]\label{P_CF1}Let $(X,\mathcal{A}%
,\mu,T,\xi)$ be the CF-system, and take any $\vartheta\in(0,1)$. Set
$\psi_{A_{l}}:=1_{\{\mathsf{a}\geq l/\vartheta\}}$ on $A_{l}$, which
identifies those digits $\geq l$ which are in fact $\geq l/\vartheta$. Then
\begin{equation}
(\mu(A_{l})\Phi_{A_{l}},\Psi_{A_{l}})\overset{\mu}{\Longrightarrow}%
(\,\Phi_{\mathrm{Exp}},\Psi^{\ast})\quad\text{as }l\rightarrow\infty\text{,}%
\end{equation}
where $(\,\Phi_{\mathrm{Exp}},\Psi^{\ast})$ is an independent pair with
$\Psi^{\ast}$ a $(1-\vartheta,\vartheta)$-Bernoulli sequence.
\end{proposition}

\begin{proof}
Let $A_{l}^{\prime}:=\{\mathsf{a}\geq l/\vartheta\}=%
{\textstyle\bigcup\nolimits_{k\geq l/\vartheta}}
I_{k}\subseteq A_{l}$, which is again $\xi$-measurable. Trivially, $\mu
(A_{l}^{\prime})\sim\vartheta\mu(A_{l}^{\prime})$ as $l\rightarrow\infty$, and
the assertion follows via Theorem \ref{T_PoissonAsyGMUnionCyls} b).
\end{proof}

%

\vspace{0.3cm}%

\begin{remark}
This result is closely related to the fact that ergodic sums of the digit
function satisfy a functional stable limit theorem,
\begin{equation}
\left(  \frac{\log2}{n}\sum_{k=0}^{\left\lfloor nt\right\rfloor -1}%
\mathsf{a}\circ T^{k}-t(\log n-\gamma)\right)  _{t\geq0}\overset{\mu
}{\Longrightarrow}\mathrm{G}\text{ \quad as }n\rightarrow\infty\text{,}
\label{Eq_FSLTforCF}%
\end{equation}
where $\gamma$ is Euler's constant, $\mathrm{G}=\left(  \mathrm{G}_{t}\right)
_{t\geq0}$ is an $\alpha$-stable motion with $\alpha=1$ and skewness $\beta
=1$, and convergence takes place on the Skorohod space $\mathbb{D}[0,\infty)$
equipped with the $J_{1}$-topology. Proposition \ref{P_CF1} can also be
derived using information on the convergence (\ref{Eq_FSLTforCF}), see the
approach to that limit theorem developed in \cite{T}.
\end{remark}

%

\vspace{0.3cm}%

Turning to a different feature of individual CF-digits which cannot be
extracted from (\ref{Eq_FSLTforCF}), we now observe that the residue classes
mod $m$ of large CF-digits are asymptotically equidistributed and independent
of each other (and of the waiting times).

\begin{proposition}
[\textbf{Residue classes of large CF-digits}]\label{P_CF2}Let $(X,\mathcal{A}%
,\mu,T,\xi)$ be the CF-system, and take an integer $m\geq2$. Define
$\psi:X\rightarrow\{0,\ldots,m-1\}$ by $\psi(x):=j$ if $\mathsf{a}(x)\equiv j$
$(\operatorname{mod}m)$, so that $\psi\circ T^{n-1}$ identifies the residue
class mod $m$ of the digit $\mathsf{a}_{n}$, and set $\psi_{A_{l}}:=\psi
\mid_{A_{l}}$. Then%
\begin{equation}
(\mu(A_{l})\Phi_{A_{l}},\Psi_{A_{l}})\overset{\mu}{\Longrightarrow}%
(\,\Phi_{\mathrm{Exp}},\Psi^{\ast})\quad\text{as }l\rightarrow\infty\text{,}%
\end{equation}
where $(\,\Phi_{\mathrm{Exp}},\Psi^{\ast})$ is an independent pair with
$\Psi^{\ast}$ a $(\frac{1}{m},\ldots,\frac{1}{m})$-Bernoulli sequence.
\end{proposition}

\begin{proof}
Setting $I_{0}:=\varnothing$ we have $\{\psi=j\}=%
{\textstyle\bigcup\nolimits_{i\geq0}}
I_{im+j}$ for all $j$. From our explicit knowledge of $\mu$ and the $I_{k} $
it is easily seen that, for each $j\in\{0,\ldots,m-1\}$, the $\xi$-measurable
sets $A_{l}^{(j)}:=A_{l}\cap\{\psi=j\}$ satisfy $\mu_{A_{l}}(A_{l}%
^{(j)})\rightarrow1/m$ as $l\rightarrow\infty$. Now apply Theorem
\ref{T_PoissonAsyGMUnionCyls} b).
\end{proof}

%

\vspace{0.3cm}%
%

\vspace{0.3cm}%
%

\noindent
\textbf{Local processes of interval maps.} We now turn to a basic situation in
which the geometry of the underlying space suggests a natural way of
describing the relative position inside small sets by specific local
observables. Call $(X,\mathcal{A},\mu,T,\xi)$ a (probability-preserving)
\emph{Gibbs-Markov interval map} provided that it is a GM-system as in the
previous section, where $X$ and each $Z\in\xi$ is an open interval, and the
invariant probability $\mu$ is absolutely continuous w.r.t. one-dimensional
Lebesgue measure $\lambda$. In this setup, we shall study asymptotically rare
sequences $(A_{l})$ of subintervals, and take the normalizing interval charts
$\psi_{A_{l}}:A_{l}\rightarrow\lbrack0,1]$ as our local observables.

Only using the elementary properties employed in the previous section, we are
going to prove

\begin{theorem}
[\textbf{Small intervals in GM interval maps}]%
\label{T_WhenWhereGMIntervalMaps}Let $(X,\mathcal{A},\mu,T,\xi)$ be a
probability-preserving ergodic Gibbs-Markov interval map, $(A_{l})$ an
asymptotically rare sequence of subintervals, and $\psi_{A_{l}}:A_{l}%
\rightarrow\lbrack0,1]$ the corresponding normalizing interval charts, giving
local processes $\Psi_{A_{l}}$. Assume that $x^{\ast}\in X$ is not periodic
and such that each $\xi_{l}(x^{\ast})$ is well defined, and the $A_{l}$ are
contained in neighbourhoods $I_{l}$ of $x^{\ast}$ with $\mathrm{diam}%
(I_{l})\rightarrow0$. Then,%
\begin{equation}
(\mu(A_{l})\Phi_{A_{l}},\widetilde{\Psi}_{A_{l}})\overset{\mu_{A_{l}}%
}{\Longrightarrow}(\,\Phi_{\mathrm{Exp}},\Psi^{\ast})\quad\text{as
}l\rightarrow\infty\text{,}%
\end{equation}
and
\begin{equation}
(\mu(A_{l})\Phi_{A_{l}},\Psi_{A_{l}})\overset{\mu}{\Longrightarrow}%
(\,\Phi_{\mathrm{Exp}},\Psi^{\ast})\quad\text{as }l\rightarrow\infty\text{,}%
\end{equation}
where $(\,\Phi_{\mathrm{Exp}},\Psi^{\ast})$ is an independent pair with
$\Psi^{\ast}$ an iid sequence of uniformly distributed elements of $[0,1]$.
\end{theorem}

To establish this result we can essentially argue as in the preceding section,
once we replace the intervals $A_{l}$ by more convenient sets $A_{l}^{\prime}$
which are unions of cylinders of rank $\jmath(A_{l})$, where
\begin{equation}
\jmath(A):=\frac{-2\log(\mu(A))}{-\log q}\text{, \quad}A\in\mathcal{A}\text{
with }\mu(A)>0\text{,} \label{Eq_DefRankFunctionJ}%
\end{equation}
where $q\in(0,1)$ is as in (\ref{Eq_GMMapsContractionOfCylMeasure}). Note that%
\begin{equation}
\mu(A_{l})\jmath(A_{l})\longrightarrow0\text{ \quad as }l\rightarrow
\infty\text{,} \label{Eq_mkhmnzknmznmznmz}%
\end{equation}
whenever $(A_{l})$ is an asymptotically rare sequence.

\begin{lemma}
[\textbf{Approximating intervals by cylinders}]\label{L_ApproxByCyls}Let
$(X,\mathcal{A},\mu,T,\xi)$ be a probability-preserving ergodic Gibbs-Markov
interval map and $(A_{l})$ an asymptotically rare sequence of subintervals.
Define
\begin{equation}
A_{l}^{\prime}:=%
{\textstyle\bigcup\nolimits_{W\in\xi_{\jmath(A_{l})}:W\subseteq A_{l}}}
W\text{, \quad}l\geq1\text{,}%
\end{equation}
then the $\xi_{\jmath(A_{l})}$-measurable sets $A_{l}^{\prime}\subseteq A_{l}$
satisfy $\mu(A_{l}\bigtriangleup A_{l}^{\prime})=o(\mu(A_{l}))$ as
$l\rightarrow\infty$.
\end{lemma}

\begin{proof}
If $A$ is an interval, $\jmath\geq1$, and $A^{\prime}:=%
{\textstyle\bigcup\nolimits_{W\in\xi_{\jmath}:W\subseteq A_{l}}}
W$, then $A_{l}\setminus A_{l}^{\prime}$ consists of at most two subintervals
of measures not exceeding $\kappa q^{\jmath}$ (recall
(\ref{Eq_GMMapsContractionOfCylMeasure})).
\end{proof}

We are now ready for the

\begin{proof}
[\textbf{Proof of Theorem \ref{T_WhenWhereGMIntervalMaps}.}]\textbf{(i)} In
view of the Lemma and Remark \ref{R_RobustJoints}, we can assume w.l.o.g. that
(up to a set of measure zero) each $A_{l}$ is $\xi_{\jmath(A_{l})}%
$-measurable. (Note that $\jmath(A_{l}^{\prime})\geq\jmath(A_{l})$.)

We prove our result by a direct application of Theorem
\ref{T_SpatiotemporalIIDLimits}. Note first that
\begin{equation}
\psi_{A_{l}}\overset{\mu_{A_{l}}}{\Longrightarrow}\psi\text{ \qquad as
}l\rightarrow\infty\text{,}%
\end{equation}
where $\psi$ is uniformly distributed in $\mathfrak{Z}:=[0,1]$ (see the
discussion following (\ref{Eq_cnbvmcbvmbmybvmyvbyb})).\newline\newline%
\textbf{(ii)} We will validate condition (B) of Theorem
\ref{T_SpatiotemporalIIDLimits} for every $F\in\mathcal{B}_{\mathfrak{Z}}%
^{\pi}:=\{[a,b]:0\leq a\leq b\leq1\}$ of positive measure. In the particular
case of $F=\mathfrak{Z}$ this verifies the assumptions of Theorem
\ref{T_MOreFlexiblePoissonLimitFromCompactness}, and hence implies Poisson
asymptotics,
\begin{equation}
\mu(A_{l})\Phi_{A_{l}}\overset{\mu_{A_{l}}}{\Longrightarrow}\,\Phi
_{\mathrm{Exp}}\text{ \quad and \quad}\mu(A_{l})\Phi_{A_{l}}\overset{\mu
}{\Longrightarrow}\,\Phi_{\mathrm{Exp}}\text{ \quad as }l\rightarrow
\infty\text{.} \label{Eq_PoiAsy}%
\end{equation}
Throughout, we use $\mathfrak{K}$ as in (\ref{Eq_GMgoodK}). Now fix any
$F\in\mathcal{B}_{\mathfrak{Z}}^{\pi}$ with $\Pr[\psi\in F]=\lambda(F)>0$.

The sets $B_{l,F}:=A_{l}\cap\{\psi_{A_{l}}\in F\}$ are intervals. Define
$B_{l,F}^{\prime}:=%
{\textstyle\bigcup\nolimits_{V\in\xi_{\jmath(B_{l,F})}:V\subseteq A_{l}}}
V$, $l\geq1$, which satisfy $\mu(B_{l,F}\bigtriangleup B_{l,F}^{\prime}%
)=o(\mu(B_{l,F}))$ as $l\rightarrow\infty$ (Lemma \ref{L_ApproxByCyls}).
Setting $\nu_{l,F}:=\mu_{B_{l,F}^{\prime}}$ we therefore have $d_{\mathfrak{P}%
}(\nu_{l,F},\mu_{A_{l}\cap\{\psi_{A_{l}}\in F\}})\rightarrow0$.

Next, $\lambda(B_{l,F})\sim\lambda(F)\lambda(A_{l})$ as $l\rightarrow\infty$
by definition of $B_{l,F}$. Since the invariant density $d\mu/d\lambda=:h$ of
$T$ is continuous at $x^{\ast}$ ($x^{\ast}$ being in the interior of
$\xi(x^{\ast})$) with $h(x^{\ast})>0$, we see that also $\mu(B_{l,F}%
)\sim\lambda(F)\mu(A_{l})$.

Let\ $\tau_{l,F}:=\jmath(B_{l,F})$, $l\geq1$, then it is immediate from
(\ref{Eq_mkhmnzknmznmznmz}) that $\mu(A_{l})\tau_{l,F}\sim\lambda(F)^{-1}%
\mu(B_{l,F})\jmath(B_{l,F})\rightarrow0$, and hence
(\ref{Eq_vzgvzgvzgvzgvzgv1}). Also, since $B_{l,F}^{\prime}$ is $\xi
_{\jmath(B_{l,F})}$-measurable, it is clear from (\ref{Eq_GMgoodK}) and
convexity of $\mathfrak{K}$ that $T_{\ast}^{\tau_{l,F}}\nu_{l,F}%
\in\mathfrak{K}$ for $l\geq1$. It only remains to check
(\ref{Eq_owowowowowoowwowowow}) or, equivalently, that
\begin{equation}
\nu_{l,F}\left(  \varphi_{A_{l}}\leq\tau_{l,F}\right)  \longrightarrow
0\quad\text{as }l\rightarrow\infty\text{.} \label{Eq_Wuhuhu}%
\end{equation}

This is done by an argument slightly extending that of Theorem
\ref{T_PoissonAsyGMCyls}. Take any $\varepsilon>0$. Choose $K\geq1$ so large
that $\widetilde{\kappa}q^{K}/(1-q)<\varepsilon/2$ with $\widetilde{\kappa
}:=2\lambda(F)^{-1}\flat^{-1}e^{r}\kappa$. For every $k\in\{1,\ldots,l\} $ we
have, using (\ref{Eq_UpperDistortBound}) and
(\ref{Eq_GMMapsContractionOfCylMeasure}),
\begin{align}
\mu_{B_{l,F}^{\prime}}(T^{-k}A_{l})  &  \leq\mu(B_{l,F}^{\prime})^{-1}%
{\textstyle\sum\nolimits_{Z\in\xi_{k}:Z\cap A_{l}\neq\varnothing}}
\mu(Z\cap T^{-k}A_{l})\nonumber\\
&  \leq\flat^{-1}e^{r}\mu(B_{l,F}^{\prime})^{-1}\mu(A_{l})%
{\textstyle\sum\nolimits_{Z\in\xi_{k}:Z\cap A_{l}\neq\varnothing}}
\mu(Z)\nonumber\\
&  \leq\flat^{-1}e^{r}\mu(B_{l,F}^{\prime})^{-1}\mu(A_{l})(\mu(A_{l})+2\kappa
q^{k})\nonumber\\
&  \leq\widetilde{\kappa}(\mu(A_{l})+q^{k})\text{ \quad for }l\geq l^{\prime
}\text{,}%
\end{align}
where we note that at most two of the $Z\in\xi_{k}$ which intersect the
interval $A_{l}$ are not covered by $A_{l}$. Since $x^{\ast}$ is not periodic,
and a continuity point of each $T^{n}$, there is some $l^{\prime\prime}$ such
that $\varphi_{A_{l}}>K$ on $A_{l}$ whenever $l\geq l^{\prime\prime}$. As seen
before, we also have $\mu(A_{l})\tau_{l,F}<\varepsilon/(2\widetilde{\kappa})$
whenever $l\geq l^{\prime\prime\prime}$. We thus find that
\begin{align}
\nu_{l,F}\left(  \varphi_{A_{l}}\leq\tau_{l,F}\right)   &  =\mu_{B_{l,F}%
^{\prime}}\left(  K\leq\varphi_{A_{l}}\leq\tau_{l,F}\right)  =\mu
_{B_{l,F}^{\prime}}(%
{\textstyle\bigcup\nolimits_{k=K}^{\tau_{l,F}}}
T^{-k}A_{l})\nonumber\\
&  \leq%
{\textstyle\sum\nolimits_{k=K}^{\tau_{l,F}}}
\mu_{B_{l,F}^{\prime}}\left(  T^{-k}A_{l}\right)  \leq\widetilde{\kappa
}\left(  \mu(A_{l})\tau_{l,F}+%
{\textstyle\sum\nolimits_{k=K}^{\tau_{l,F}}}
q^{k}\right) \nonumber\\
&  <\varepsilon\text{ \quad for }l\geq l^{\prime}\vee l^{\prime\prime}\vee
l^{\prime\prime\prime}\text{,}%
\end{align}
and (\ref{Eq_Wuhuhu}) follows as $\varepsilon>0$ was arbitrary. Condition (B)
of Theorem \ref{T_SpatiotemporalIIDLimits} is fulfilled.\newline%
\newline\textbf{(iii)} Turning to condition (A) of Theorem
\ref{T_SpatiotemporalIIDLimits}, fix any $s\in\lbrack0,\infty)$. For
$\theta>0$ consider the sets $C_{l}(\theta):=A_{l}\cap\{\varphi_{A_{l}}%
>\theta\} $, $l\geq1$. Since the $A_{l}$ are $\xi_{\jmath(A_{l})}$-measurable,
each $C_{l}(\theta)$ is $\xi_{\jmath(A_{l})+\theta}$-measurable.

We approximate the $B_{l}:=A_{l}\cap\{\mu(A_{l})\varphi_{A_{l}}>s\}=C_{l}%
(\theta_{l})$ with $\theta_{l}:=s/\mu(A_{l})$ by the sets $B_{l}^{\prime
}:=C_{l}(\theta_{l}^{\prime})$ with $\theta_{l}^{\prime}:=\theta_{l}%
-\jmath(A_{l})$. It is clear from the definition of $\jmath(A)$ that
$\theta_{l}^{\prime}\sim\theta_{l}$ as $l\rightarrow\infty$. In view of step
(ii) above, we can already use (\ref{Eq_PoiAsy}). The latter shows that
$\mu(B_{l}\bigtriangleup B_{l}^{\prime})=o(\mu(B_{l}))$, and hence
$d_{\mathfrak{P}}(\mu_{B_{l}^{\prime}},\mu_{B_{l}})\rightarrow0$.

But for the $B_{l}^{\prime}$ condition (A) is very easy if we take $\tau
_{l,s}:=\theta_{l}$. Indeed, by (\ref{Eq_PoiAsy}), $\mu_{B_{l}^{\prime}%
}(\varphi_{A_{l}}\leq\tau_{l,s})=\mu_{B_{l}^{\prime}}(s-\mu(A_{l})\jmath
(A_{l})<\mu(A_{l})\varphi_{A_{l}}\leq s)\rightarrow0$ as $l\rightarrow\infty$.
On the other hand, $T_{\ast}^{\tau_{l,s}}\mu_{B_{l}^{\prime}}\in\mathfrak{K}$
for $l\geq1$, because each $B_{l}^{\prime}$ is $\theta_{l}$-measurable.
\end{proof}

%

\vspace{0.3cm}%

\section{Inducing and further examples}%

\noindent
\textbf{Induced versions of the processes.} When studying specific systems,
one often tries to find some good reference set $Y\in\mathcal{A}$ such that
the first-return map $T_{Y}:Y\rightarrow Y$ is more convenient than $T$. In
this case, it often pays to prove a relevant property first for $T_{Y}$, and
to transfer it back to $T$ afterwards.

In the following, we let $\varphi_{A}^{Y}:Y\rightarrow\overline{\mathbb{N}} $
denote the hitting time of $A\in\mathcal{A}\cap Y$ under the first-return map
$T_{Y}$, that is,
\begin{equation}
\varphi_{A}^{Y}(x):=\inf\{j\geq1:T_{Y}^{j}x\in A\}\text{,\quad}x\in Y\text{,}%
\end{equation}
and write $\Phi_{A}^{Y}:=(\varphi_{A}^{Y},\varphi_{A}^{Y}\circ T_{A}%
,\varphi_{A}^{Y}\circ T_{A}^{2},\ldots)\ $on $Y$ for the hitting-time process
of $A$ under $T_{Y}$. Since $\mu_{Y}$ is the natural invariant probability
measure for $T_{Y}$, the canonical normalization for $\varphi_{A}^{Y}$ and
$\Phi_{A}^{Y}$ is $\mu_{Y}(A)$.

Given an $\mathfrak{Z}$-valued local observable on $A$ with corresponding
local process $\Psi_{A}=(\psi_{A}\circ T_{A},\psi_{A}\circ T_{A}^{2},\ldots)$,
we can also consider the corresponding object for the first-return map,
$\Psi_{A}^{Y}=(\psi_{A}\circ(T_{Y})_{A},\psi_{A}\circ(T_{Y})_{A}^{2},\ldots)$
on $Y$. But since the first-return maps on $A$ respectively induced by $T$ and
$T_{Y}$ coincide, $T_{A}=(T_{Y})_{A}$, we have $\Psi_{A}^{Y}=\Psi_{A}\mid_{Y}$
and there is no need for this extra notation.%

\vspace{0.3cm}%
%

\noindent
\textbf{Relating original and induced processes.} Inducing was first used to
deal with limit laws for normalized return- or hitting times $\mu
(A)\varphi_{A}$ in \cite{BruSauTrouVai}. A more general abstract form of their
result was given in \cite{HWZ}, and \cite{Zexceptional} contains an even more
flexible version. The theorem below confirms that the same strategy can also
be employed when dealing with joint processes $(\mu(A)\Phi_{A},\Psi_{A})$ for
small sets. The argument closely follows that of \cite{HWZ}, and its process
variant from \cite{FFTV}, but compares the two hitting-time processes in
probability rather than just in distribution, thus keeping track of their
relation to the second process $\Psi_{A}$.

\begin{theorem}
[\textbf{Joint limit processes under }$\mu$\textbf{\ via inducing}%
]\label{T_InduceSpatioTempMu}Let $(X,\mathcal{A},\mu,T)$ be an ergodic
probability preserving system, $Y\in\mathcal{A}$, $(A_{l})$ an asymptotically
rare sequence in $\mathcal{A}\cap Y$, and $(\psi_{A_{l}})_{l\geq1}$ a sequence
of $\mathfrak{Z}$-valued local observables for the $A_{l}$ with corresponding
local processes $\Psi_{A_{l}}$.

Assume that $(\Phi,\Psi)$ is a random element of $[0,\infty)^{\mathbb{N}%
}\times\mathfrak{Z}^{\mathbb{N}}$. Then, as $l\rightarrow\infty$,
\begin{equation}
(\mu(A_{l})\Phi_{A_{l}},\Psi_{A_{l}})\overset{\mu}{\Longrightarrow}%
(\,\Phi,\Psi)\quad\text{iff\quad}(\mu_{Y}(A_{l})\Phi_{A_{l}}^{Y},\Psi_{A_{l}%
})\overset{\mu_{Y}}{\Longrightarrow}(\,\Phi,\Psi)\text{.}
\label{Eq_cvcvcvcvcvbbbbbbbb}%
\end{equation}

\end{theorem}

\begin{proof}
\textbf{(i)} According to Proposition \ref{P_SDCJoint}, we can replace $\mu$
by $\mu_{Y}$ in the first convergence statement of
(\ref{Eq_cvcvcvcvcvbbbbbbbb}). Therefore it suffices to show that for every
$d\geq1$,
\[
(\mu(A_{l})\Phi_{A_{l}}^{[d]},\Psi_{A_{l}}^{[d]})\overset{\mu_{Y}%
}{\Longrightarrow}(\,\Phi^{\lbrack d]},\Psi^{\lbrack d]})\quad\text{iff\quad
}(\mu_{Y}(A_{l})\Phi_{A_{l}}^{Y,[d]},\Psi_{A_{l}}^{[d]})\overset{\mu_{Y}%
}{\Longrightarrow}(\,\Phi^{\lbrack d]},\Psi^{\lbrack d]})\text{.}%
\]
Since $\Phi^{\lbrack d]}=(\varphi^{(0)},\ldots,\varphi^{(d-1)})$ is
finite-valued by assumption, we do not lose information if instead we work
with $\Phi^{\Sigma\lbrack d]}:=(\varphi^{(0)},\varphi^{(0)}+\varphi
^{(1)},\ldots,\varphi^{(0)}+\cdots+\varphi^{(d-1)})$. Define $\Phi_{A_{l}%
}^{\Sigma\lbrack d]}$ and $\Phi_{A_{l}}^{Y,\Sigma\lbrack d]}$ analogously as
vectors of partial sums of $\Phi_{A_{l}}^{[d]}$ and $\Phi_{A_{l}}^{Y,[d]}$,
respectively. Then,
\[
(\mu(A_{l})\Phi_{A_{l}}^{[d]},\Psi_{A_{l}}^{[d]})\overset{\mu_{Y}%
}{\Longrightarrow}(\,\Phi^{\lbrack d]},\Psi^{\lbrack d]})\quad\text{iff\quad
}(\mu(A_{l})\Phi_{A_{l}}^{\Sigma\lbrack d]},\Psi_{A_{l}}^{[d]})\overset
{\mu_{Y}}{\Longrightarrow}(\,\Phi^{\Sigma\lbrack d]},\Psi^{\lbrack
d]})\text{,}%
\]
while
\[
(\mu_{Y}(A_{l})\Phi_{A_{l}}^{Y,[d]},\Psi_{A_{l}}^{[d]})\overset{\mu_{Y}%
}{\Longrightarrow}(\,\Phi^{\lbrack d]},\Psi^{\lbrack d]})\quad\text{iff\quad
}(\mu_{Y}(A_{l})\Phi_{A_{l}}^{Y,\Sigma\lbrack d]},\Psi_{A_{l}}^{[d]}%
)\overset{\mu_{Y}}{\Longrightarrow}(\,\Phi^{\Sigma\lbrack d]},\Psi^{\lbrack
d]})\text{.}%
\]
Therefore our assertion (\ref{Eq_cvcvcvcvcvbbbbbbbb}) follows once we check
that for every $d\geq1$,
\[
(\mu(A_{l})\Phi_{A_{l}}^{\Sigma\lbrack d]},\Psi_{A_{l}}^{[d]})\overset{\mu
_{Y}}{\Longrightarrow}(\,\Phi^{\lbrack d]},\Psi^{\lbrack d]})\ \text{iff }%
(\mu_{Y}(A_{l})\Phi_{A_{l}}^{Y,\Sigma\lbrack d]},\Psi_{A_{l}}^{[d]}%
)\overset{\mu_{Y}}{\Longrightarrow}(\,\Phi^{\lbrack d]},\Psi^{\lbrack
d]})\text{.}%
\]
The latter is immediate if we check that for every $d\geq1$,%
\begin{equation}
d_{[0,\infty]^{d}}\left(  \mu(A_{l})\Phi_{A_{l}}^{\Sigma\lbrack d]},\mu
_{Y}(A_{l})\Phi_{A_{l}}^{Y,\Sigma\lbrack d]}\right)  \overset{\mu_{Y}%
}{\longrightarrow}0\text{ \quad as }l\rightarrow\infty\text{,}
\label{Eq_hshshashhshdshhadh}%
\end{equation}
\newline because this convergence in probability trivially carries over to the
joint processes,
\[
d_{[0,\infty]^{d}\times\mathfrak{Z}^{d}}\left(  (\mu(A_{l})\Phi_{A_{l}%
}^{\Sigma\lbrack d]},\Psi_{A_{l}}^{[d]}),(\mu_{Y}(A_{l})\Phi_{A_{l}}%
^{Y,\Sigma\lbrack d]},\Psi_{A_{l}}^{[d]})\right)  \overset{\mu_{Y}%
}{\longrightarrow}0\text{ \quad as }l\rightarrow\infty\text{.}%
\]
\textbf{(ii)} To validate (\ref{Eq_hshshashhshdshhadh}), we take some $d\geq1$
and any $\varepsilon>0$. Note that $\varkappa(\delta):=\sup_{t\geq
0}d_{[0,\infty]}(t,e^{\delta}t)\rightarrow0$ as $\delta\searrow0$. Now fix
some $\delta>0$ so small that $d\varkappa(\delta)<\varepsilon$. By the Ergodic
theorem and Kac' formula, we have
\begin{equation}
m^{-1}%
{\textstyle\sum\nolimits_{j=0}^{m-1}}
\varphi_{Y}\circ T_{Y}^{j}\longrightarrow\mu(Y)^{-1}\text{\quad a.e. on
}Y\text{,}%
\end{equation}
so that the increasing sequence $(E_{M})_{M\geq1}$ of sets given by
\[
E_{M}:=\{%
{\textstyle\sum\nolimits_{j=0}^{m-1}}
\varphi_{Y}\circ T_{Y}^{j}=e^{\pm\delta}\mu(Y)^{-1}m\text{ for }m\geq
M\}\in\mathcal{A}\cap Y
\]
satisfies $\mu_{Y}(E_{M}^{c})\rightarrow0$ as $M\rightarrow\infty$. Now fix
some $M$ such that $\mu_{Y}(E_{M}^{c})<\varepsilon/2$. Next, let
$F_{l}:=\{\varphi_{A_{l}}^{Y}\geq M\}\in\mathcal{A}\cap Y$, $l\geq1$. Then,
$F_{l}=Y\cap%
{\textstyle\bigcap\nolimits_{j=1}^{M-1}}
T_{Y}^{-j}A_{l}^{c}$, and hence $\mu_{Y}(F_{l}^{c})\leq\sum_{j=1}^{M-1}\mu
_{Y}(T_{Y}^{-j}A_{l})\leq M\,\mu_{Y}(A_{l})\rightarrow0$ as $l\rightarrow
\infty$. Pick $L\geq1$ so large that $\mu_{Y}(F_{l}^{c})<\varepsilon/2$ for
$l\geq L$.\newline\newline\textbf{(iii)} Observe now that for $A\in
\mathcal{A}\cap Y$ we have $\varphi_{A}=\sum_{j=0}^{\varphi_{A}^{Y}-1}%
\varphi_{Y}\circ T_{Y}^{j}\,\ $on $Y$, and, more generally, for any $i\geq1$,
\begin{equation}
\varphi_{A}+\cdots+\varphi_{A}\circ T_{A}^{i-1}=\sum_{j=0}^{\varphi_{A}%
^{Y}+\cdots+\varphi_{A}^{Y}\circ T_{A}^{i-1}-1}\varphi_{Y}\circ T_{Y}%
^{j}\text{ \quad on }Y\text{.}%
\end{equation}
By definition of $E_{M}$ and $F_{l}$ we have, for every $i\in\{1,\ldots,d\}$,%
\[
\varphi_{A_{l}}+\cdots+\varphi_{A_{l}}\circ T_{A_{l}}^{i-1}=e^{\pm\delta}%
\mu(Y)^{-1}\left(  \varphi_{A_{l}}^{Y}+\cdots+\varphi_{A_{l}}^{Y}\circ
T_{A_{l}}^{i-1}\right)  \text{ on }E_{M}\cap F_{l}\text{.}%
\]
Note that the left-hand expression is the $i$th component of $\Phi_{A_{l}%
}^{\Sigma\lbrack d]}$, while the sum on the right-hand side is the $i$th
component of $\Phi_{A_{l}}^{Y,\Sigma\lbrack d]}$. Due to our choice of
$\delta$ we thus find that
\[
d_{[0,\infty]^{d}}\left(  \mu(A_{l})\Phi_{A_{l}}^{\Sigma\lbrack d]},\mu
_{Y}(A_{l})\Phi_{A_{l}}^{Y,\Sigma\lbrack d]}\right)  \leq d\varkappa
(\delta)<\varepsilon\text{ on }E_{M}\cap F_{l}\text{.}%
\]
But since $\mu_{Y}((E_{M}\cap F_{l})^{c})<\varepsilon$ for $l\geq L$, this
proves (\ref{Eq_hshshashhshdshhadh}).
\end{proof}

%

\vspace{0.3cm}%

We also provide an inducing principle for limits under the measures
$\mu_{A_{l}}$. Recall from (\ref{Eq_AWarningLocObs}) that transferring
information about the asymptotics of spatiotemporal processes under one of
$(\mu)_{l\geq1}$ and $(\mu_{A_{l}})_{l\geq1}$ to the other requires extra
information. Therefore this principle is less general than Theorem
\ref{T_InduceSpatioTempMu}, and we content ourselves with the case relevant
for typical applications.

\begin{proposition}
[\textbf{Joint limit processes under }$\mu_{A_{l}}$\textbf{\ via inducing}%
]\label{P_InduceSpatioTempFromAl}Let $(X,\mathcal{A},\mu,T)$ be an ergodic
probability preserving system, $Y\in\mathcal{A}$, $(A_{l})$ an asymptotically
rare sequence in $\mathcal{A}\cap Y$, and $(\psi_{A_{l}})_{l\geq1}$ a sequence
of $\mathfrak{Z}$-valued local observables for the $A_{l}$ with corresponding
local processes $\Psi_{A_{l}}$. Assume there are measurable $\tau_{l}%
^{Y}:A_{l}\rightarrow\mathbb{N}_{0}$ and a compact set $\mathfrak{K}%
\subseteq\mathfrak{P}$ such that
\begin{equation}
\mu_{A_{l}}\left(  \tau_{l}^{Y}<\varphi_{A_{l}}^{Y}\right)  \longrightarrow
1\text{ as }l\rightarrow\infty\text{ \quad while\quad}(T_{Y})_{\ast}^{\tau
_{l}^{Y}}\mu_{A_{l}}\in\mathfrak{K}\text{ for }l\geq1\text{,}
\label{Eq_xcxcxcxc}%
\end{equation}
and that
\begin{equation}
(\mu_{Y}(A_{l})\Phi_{A_{l}}^{Y},\widetilde{\Psi}_{A_{l}})\overset{\mu_{A_{l}}%
}{\Longrightarrow}(\Phi_{\mathrm{Exp}},\Psi^{\ast})\quad\text{as }%
l\rightarrow\infty\text{,} \label{Eq_xcxcxcxcx}%
\end{equation}
with $(\Phi_{\mathrm{Exp}},\Psi^{\ast})$ an independent pair of iid sequences.
Then
\begin{equation}
(\mu(A_{l})\Phi_{A_{l}},\widetilde{\Psi}_{A_{l}})\overset{\mu_{A_{l}}%
}{\Longrightarrow}(\Phi_{\mathrm{Exp}},\Psi^{\ast})\quad\text{as }%
l\rightarrow\infty\text{.} \label{Eq_khkhkhkkhhkhkhk}%
\end{equation}

\end{proposition}

\begin{proof}
Applying Theorem \ref{T_JointLimitProcMuVsMuA} to the first-return map $T_{Y}
$ we see that (\ref{Eq_xcxcxcxc}) and (\ref{Eq_xcxcxcxcx}) together imply
$(\mu_{Y}(A_{l})\Phi_{A_{l}}^{Y},\Psi_{A_{l}})\overset{\mu_{Y}}%
{\Longrightarrow}(\Phi_{\mathrm{Exp}},\Psi^{\ast})$. In view of Thorem
\ref{T_InduceSpatioTempMu} this entails
\begin{equation}
(\mu(A_{l})\Phi_{A_{l}},\Psi_{A_{l}})\overset{\mu}{\Longrightarrow}%
(\Phi_{\mathrm{Exp}},\Psi^{\ast})\quad\text{as }l\rightarrow\infty\text{.}
\label{Eq_hghghghghghghg}%
\end{equation}
Now define $\tau_{l}:=\sum_{j=0}^{\tau_{l}^{Y}-1}\varphi_{Y}\circ T_{Y}^{j}$
for $l\geq1$. This ensures that $T^{\tau_{l}}=(T_{Y})^{\tau_{l}^{Y}}$ on
$A_{l}$ and that $A_{l}\cap\{\tau_{l}<\varphi_{A_{l}}\}=A_{l}\cap\{\tau
_{l}^{Y}<\varphi_{A_{l}}^{Y}\}$. Having thus found measurable $\tau_{l}%
:A_{l}\rightarrow\mathbb{N}_{0}$ satisfying
\[
\mu_{A_{l}}\left(  \tau_{l}<\varphi_{A_{l}}\right)  \longrightarrow1\text{ as
}l\rightarrow\infty\text{ \quad while\quad}T^{\tau_{l}}\mu_{A_{l}}%
\in\mathfrak{K}\text{ for }l\geq1\text{,}%
\]
we can apply Theorem \ref{T_JointLimitProcMuVsMuA} again, this time to $T$, to
derive (\ref{Eq_khkhkhkkhhkhkhk}) from (\ref{Eq_hghghghghghghg}).
\end{proof}

%

\vspace{0.3cm}%
%

\noindent
\textbf{Application: Spatiotemporal Poisson limits for systems inducing GM
maps.} Using the results of this section, we can easily provide further
examples of systems exhibiting spatiotemporal Poisson limits, including some
with arbitrarily slow mixing rates. We first record

\begin{theorem}
[\textbf{Small intervals via induced GM maps}]%
\label{T_SpatioTempIntervalMapsInducing}Let $(X,\mathcal{A},\mu,T)$ be a
probability-preserving system, and $Y\in\mathcal{A}$ a set with $\mu(Y)>0$ on
which it induces a Gibbs-Markov interval map $(Y,\mathcal{A}\cap Y,\mu
_{Y},T_{Y},\xi_{Y})$. Let $(A_{l})$ an asymptotically rare sequence of
subintervals of $Y$, and $\psi_{A_{l}}:A_{l}\rightarrow\lbrack0,1]$ the
corresponding normalizing interval charts, giving local processes $\Psi
_{A_{l}}$. Assume that $x^{\ast}\in Y$ is not periodic and such that each
$\xi_{Y,l}(x^{\ast})$ is well defined, and the $A_{l}$ are contained in
neighbourhoods $I_{l}$ of $x^{\ast}$ with $\mathrm{diam}(I_{l})\rightarrow0 $.
Then,%
\begin{equation}
(\mu(A_{l})\Phi_{A_{l}},\widetilde{\Psi}_{A_{l}})\overset{\mu_{A_{l}}%
}{\Longrightarrow}(\,\Phi_{\mathrm{Exp}},\Psi^{\ast})\quad\text{as
}l\rightarrow\infty\text{,} \label{Eq_slslslslsslsl}%
\end{equation}
and
\begin{equation}
(\mu(A_{l})\Phi_{A_{l}},\Psi_{A_{l}})\overset{\mu}{\Longrightarrow}%
(\,\Phi_{\mathrm{Exp}},\Psi^{\ast})\quad\text{as }l\rightarrow\infty\text{,}
\label{Eq_jhbvfjhfedbvhdfabvjhadvahg}%
\end{equation}
where $(\,\Phi_{\mathrm{Exp}},\Psi^{\ast})$ is an independent pair with
$\Psi^{\ast}$ an iid sequence of uniformly distributed elements of $[0,1]$.
\end{theorem}

\begin{proof}
It is immediate from our assumptions that Theorem
\ref{T_WhenWhereGMIntervalMaps} applies to the induced system, thus proving
\begin{equation}
(\mu_{Y}(A_{l})\Phi_{A_{l}}^{Y},\Psi_{A_{l}})\overset{\mu_{Y}}{\Longrightarrow
}(\,\Phi_{\mathrm{Exp}},\Psi^{\ast})\quad\text{as }l\rightarrow\infty\text{,}%
\end{equation}
and%
\begin{equation}
(\mu_{Y}(A_{l})\Phi_{A_{l}}^{Y},\widetilde{\Psi}_{A_{l}})\overset{\mu_{A_{l}}%
}{\Longrightarrow}(\Phi_{\mathrm{Exp}},\Psi^{\ast})\quad\text{as }%
l\rightarrow\infty\text{.} \label{Eq_ubububugugugbug}%
\end{equation}
The first of these immediately implies (\ref{Eq_jhbvfjhfedbvhdfabvjhadvahg})
via Theorem \ref{T_InduceSpatioTempMu}. We then check that
(\ref{Eq_slslslslsslsl}) can be derived from (\ref{Eq_ubububugugugbug}) using
Proposition \ref{P_InduceSpatioTempFromAl}. In view of Remark
\ref{R_RobustJoints} and Lemma \ref{L_ApproxByCyls} we can replace $(A_{l})$
by $(A_{l}^{\prime})$ in both (\ref{Eq_slslslslsslsl}) and
(\ref{Eq_ubububugugugbug}), where $A_{l}^{\prime}:=%
{\textstyle\bigcup\nolimits_{W\in\xi_{Y,\jmath(A_{l})}:W\subseteq A_{l}}}
W$, $l\geq1$, (note the use of the partition $\xi_{Y}$ of the induced system)
with $\jmath(A_{l})$ as in (\ref{Eq_DefRankFunctionJ}). To apply Proposition
\ref{P_InduceSpatioTempFromAl}, recall that condition (\ref{Eq_xcxcxcxc}) for
$(A_{l}^{\prime})$ and the induced Gibbs-Markov interval map $T_{Y}$ has
already been validated in the proof of Theorem \ref{T_WhenWhereGMIntervalMaps}.
\end{proof}

%

\vspace{0.3cm}%

We illustrate the use of this theorem by applying it to prototypical
nonuniformly expanding interval maps with indifferent fixed points.
(Everything said below generalized in a trivial way to Markovian interval maps
with several indifferent fixed points satisfying the obvious analogous
analytical conditions.)

\begin{example}
[\textbf{Probability preserving intermittent interval maps}]Let $(X,T,\xi)$ be
piecewise increasing with $X=[0,1]$ and $\xi=\{(0,c),(c,1)\}$, mapping each
$Z\in\xi$ onto $(0,1)$. Assume that $T\mid_{(c,1)}$ admits a uniformly
expanding $\mathcal{C}^{2}$ extension to $[c,1]$, while $T\mid_{(0,c)}$
extends to a $\mathcal{C}^{2}$ map on $(0,c]$ and is expanding except for an
indifferent fixed point at $x^{\ast}=0$: for every $\varepsilon>0$ there is
some $\rho(\varepsilon)>1$ such that $T^{\prime}\geq\rho(\varepsilon)$ on
$[\varepsilon,c]$, while $T0=0$ and $\lim\nolimits_{x\searrow0}T^{\prime}x=1$
with $T^{\prime}$ increasing on some $(0,\delta)$. Suppose also that
\begin{gather}
\text{there is a continuous decreasing function }g\text{ on }(0,c]\text{
with}\nonumber\\
\int_{0}^{c}g(x)\,dx<\infty\text{ \quad and \quad}\left\vert T^{\prime\prime
}\right\vert \leq g\text{ on }(0,c]\text{.} \label{Eq_MaxMax}%
\end{gather}
Let $c=:c_{0}>c_{1}>c_{2}>\ldots>0$ be the points with $Tc_{m}=c_{m-1}$ for
$m\geq1$. By an analytic argument which goes back to \cite{T1}, see for
example \S 3 of \cite{Z3} or \S 4 of \cite{T3}, assumption (\ref{Eq_MaxMax})
ensures that the induced system on any $Y:=Y^{(m)}:=(c_{m},1)$ is
Gibbs-Markov, and that $T$ posesses a unique absolutely continuous (w.r.t.
Lebesgue measure $\lambda$) invariant probability measure $\mu$ with density
$h$ strictly positive and continuous on $(0,1]$. This family of maps contains
systems with very slow decay of correlations, see for example \cite{Holland}.

Now take any $x^{\ast}\in X$ which is not periodic and such that each $\xi
_{j}(x^{\ast})$ is well defined, let $(A_{l})_{l\geq1}$ be a sequence of
non-degenerate intervals contained in neighbourhoods $I_{l}$ of $x^{\ast}$
with $\mathrm{diam}(I_{l})\rightarrow0$, and let $\psi_{A_{l}}:A_{l}%
\rightarrow\lbrack0,1]$ denote the corresponding normalizing interval charts,
giving local processes $\Psi_{A_{l}}$. Then Theorem
\ref{T_SpatioTempIntervalMapsInducing} applies (with $Y=Y^{(m)}$ for $m$ so
large that $x^{\ast}\in Y$), so that (\ref{Eq_slslslslsslsl}) and
(\ref{Eq_jhbvfjhfedbvhdfabvjhadvahg}) hold in the present situation. (We only
need to observe that all cylinders $\xi_{Y,l}(x^{\ast})$ of the induced system
around this particular point are well defined since the cylinders of the
original system are.)
\end{example}

%

\vspace{0.3cm}%

\section{Relation to the tail-$\sigma$-algebra}

The abstract limit theorems of the present paper merely require the
probability preserving system $(X,\mathcal{A},\mu,T)$ to be ergodic. On top of
this we only impose conditions on the specific asymptotically rare sequence
$(A_{l})$ under consideration. These assumptions do not imply that the system
is mixing, as is clear from the basic GM-map examples of Section
\ref{Sec_GMmaps}, which can be taken to have a periodic structure. However, it
is well-known that this is the only way in which an ergodic probability
preserving GM-map can fail to be mixing, since it always has a discrete
tail-$\sigma$-algebra.

We conclude by showing that if an ergodic probability preserving map $T$
admits a (one-sided) generating partition (mod $\mu$) such that the cylinders
$\xi_{l}(x)$ around a.e. point $x\in X$ (the element of $\xi_{l}%
:=\bigvee_{j=0}^{l-1}T^{-j}\xi$ containing $x$) satisfy a condition similar to
that used above, with constant delay times and a common compact set of image
measures, then it is exact up to a cyclic permutation.

\begin{proposition}
[\textbf{Abundance of good cylinders implies discrete tail}]Let
$(X,\mathcal{A},\mu,T)$ be an ergodic probability preserving system. Assume
that there is a compact subset $\mathfrak{K}$ of $(\mathfrak{P}%
,d_{\mathfrak{P}})$ and a countable partition $\xi$ of $X$ (mod $\mu$), with
$\mathcal{A}=\sigma(\xi_{n}:n\geq1)$ (mod $\mu$) and the following property:
For a.e. $x\in X$ the sequence $(\xi_{l}(x))_{l\geq1}$ is well defined with
$0<\mu(\xi_{l}(x))\searrow0$, and it admits a sequence of constants
$\tau_{x,l}\in\mathbb{N}_{0}$, such that
\begin{equation}
T_{\ast}^{\tau_{x,l}}\mu_{\xi_{l}(x)}\in\mathfrak{K}\text{ \quad for }%
l\geq1\text{.} \label{Eq_jshdcsjadhvbh}%
\end{equation}
Then there are $p\in\mathbb{N}$ and $X_{1},\ldots,X_{p}\in\mathcal{A}$ such
that $T^{-1}X_{i+1}=X_{i}$. The tail-$\sigma$-algebra has the form
$\bigcap_{n\geq0}T^{-n}\mathcal{A}=\sigma(X_{1},\ldots,X_{p})$ (mod $\mu$),
and $T^{p}\mid_{X_{i}}$ is exact.
\end{proposition}

\begin{remark}
\textbf{a)} The proof only requires $\mathfrak{K}$ to be \emph{weakly} compact
in $(\mathfrak{P},d_{\mathfrak{P}})$ (with the latter identified with
$(\mathcal{D}(\mu),\left\Vert \centerdot\right\Vert _{L_{1}(\mu)})$).
\newline\textbf{b)} Condition (\ref{Eq_jshdcsjadhvbh}), and its generalization
pointed out in a), can be interpreted as weak bounded distortion conditions.
We exploit them through a Rohlin type argument, see \cite{Ro}.\newline%
\textbf{c)} If we drop the assumption of ergodicity, the argument below still
shows that the tail-$\sigma$-algebra is discrete.
\end{remark}

\begin{proof}
Abbreviate $\mathcal{A}_{\infty}:=\bigcap_{n\geq0}T^{-n}\mathcal{A}$. Our
assertion follows by easy routine arguments once we show that there is some
constant $\eta>0$ such that
\begin{equation}
\mu(A)>\eta\text{ \quad for all }A\in\mathcal{A}_{\infty}\text{ with }%
\mu(A)>0\text{.} \label{Eq_ConstEtaForTail}%
\end{equation}
\textbf{(i)} Let $\mathfrak{U}\subseteq L_{1}(\mu)$ be the set of
(probability) densities of all measures $\nu\in\mathfrak{K}$. Then
$\mathfrak{U}$ is strongly compact in $L_{1}(\mu)$, and hence also uniformly
integrable. Consequently, there is some $\eta>0$ such that
\begin{equation}
\mu(h>2\eta)>2\eta\text{ \quad for }u\in\mathcal{U}\text{.}
\label{Eq_scvcadsxdasxdaxdasxd}%
\end{equation}
(Otherwise there are $u_{m}\in\mathfrak{U}$ such that the sets $B_{m}%
:=\{u_{m}>2/m\}$ satisfy $\mu(B_{m})\leq2/m$ for $m\geq1$. But $\mu
(B_{m})\rightarrow0$ implies $\int_{B_{m}}u_{m}\,d\mu\rightarrow0$ by uniform
integrability, contradicting $\int u_{m}\,d\mu=1$.) For probability densities
$u,v$ with $u\in\mathfrak{U}$ we then have
\[
\mu(v>\eta)\geq\mu(u>2\eta)-\mu(\{u>2\eta\}\setminus\{v>\eta\})>2\eta
-\eta^{-1}\left\Vert u-v\right\Vert _{1}\text{,}%
\]
so that
\begin{equation}
\mu(v>\eta)>\eta\text{ \quad for probability densities }v\text{ with
}\mathrm{dist}_{L_{1}(\mu)}(v,\mathfrak{U})<\eta^{2}\text{. }
\label{Eq_hvbfdhvbjdfvbjshbvbv}%
\end{equation}
\textbf{(ii)} Now fix any $A\in\mathcal{A}_{\infty}$ (meaning that there are
$A_{n}\in\mathcal{A}$ such that $A=T^{-n}A_{n}$ for $n\geq0$) for which
$\mu(A)>0$. By the a.s. martingale convergence theorem (see Theorem 35.5 and
Exercise 35.17 of \cite{BilliPM}),\ we have
\begin{equation}
\mu_{\xi_{l}(x)}(\xi_{l}(x)\cap A)\longrightarrow1\text{\quad as }%
l\rightarrow\infty\text{ for a.e. }x\in A\text{. } \label{Eq_kjhgmkhgmjjhg}%
\end{equation}
Take a point $x\in A$ satisfying both (\ref{Eq_jshdcsjadhvbh}) and
(\ref{Eq_kjhgmkhgmjjhg}). Abbreviating $\tau_{l}:=\tau_{x,l}$, we consider the
probability densities given by $u_{l}:=\mu(\xi_{l}(x))^{-1}\widehat{T}%
^{\tau_{l}}1_{\xi_{l}(x)}$, $v_{l}:=\mu(\xi_{l}(x)\cap A)^{-1}\widehat
{T}^{\tau_{l}}1_{\xi_{l}(x)\cap A}$, and $w_{l}:=\mu(\xi_{l}(x)\setminus
A)^{-1}\widehat{T}^{\tau_{l}}1_{\xi_{l}(x)\setminus A}$, $l\geq1$. Since
\[
v_{l}=u_{l}+\frac{\mu(\xi_{l}(x)\setminus A)}{\mu(\xi_{l}(x)\cap A)}%
(u_{l}-w_{l})\text{\quad for }l\geq1\text{,}%
\]
where $u_{l}\in\mathfrak{U}$ and $\mu(\xi_{l}(x)\setminus A)/\mu(\xi
_{l}(x)\cap A)\rightarrow0$, there is some $l_{0}\geq1$ such that
\[
\mathrm{dist}_{L_{1}(\mu)}(v_{l},\mathfrak{U})<\eta^{2}\quad\text{for }l\geq
l_{0}\text{.}%
\]
Because $\xi_{l}(x)\cap A\subseteq A=T^{-\tau_{l}}A_{\tau_{l}}$, we have
$\{v_{l}>0\}\subseteq A_{\tau_{l}}$ (mod $\mu$). Therefore
(\ref{Eq_hvbfdhvbjdfvbjshbvbv}) entails
\[
\mu(A_{\tau_{l}})\geq\mu(v_{l}>0)\geq\mu(v_{l}>\eta)>\eta\text{ \quad for
}l\geq l_{0}\text{.}%
\]
But $\mu(A)=\mu(A_{\tau_{l}})$ for all $l$ since $T$ preserves $\mu$. This
proves (\ref{Eq_ConstEtaForTail}).
\end{proof}

%

\vspace{0.3cm}%

\textbf{Acknowledgements.} I am grateful to the refereree for comments and
questions which helped to clarify some points, and to Francoise P\`{e}ne,
Benoit Saussol and Max Auer for discussions on related matters.

\end{document}